\documentclass[11pt, oneside]{amsart}

\newcommand{\hide}[1]{}

\usepackage{amssymb}
\usepackage{graphicx}

\usepackage{xcolor}

\usepackage{geometry}
\usepackage[colorlinks=true]{hyperref} 
\usepackage{enumitem}

\usepackage{caption}

\usepackage{pgfplots}
\pgfplotsset{compat=1.15}
\usepackage{mathrsfs}
\usetikzlibrary{arrows}

\newcommand{\N}{\mathbb{N}}
\newcommand{\R}{\mathbb{R}}
\newcommand{\C}{\mathbb{C}}

\newcommand{\disk}{\mathbb{D}}
\newcommand{\Cc}{\widehat{{\C}}}

\DeclareMathOperator{\Dom}{\mathcal R} 
\DeclareMathOperator{\DomE}{\mathcal L}		
\DeclareMathOperator{\DomL}{\hat{\mathcal L}}	

\newcommand{\U}{\mathcal U}
\newcommand{\V}{\mathcal V}
\newcommand{\W}{\mathcal W}
\newcommand{\Y}{\mathcal Y}
\newcommand{\Pp}{\mathcal P}
\newcommand{\X}{\mathcal X}

\def\J{\mathcal J} 
\def\K{\mathcal K} 

\newcommand{\bV}{V} 
\newcommand{\bI}{I} 
\newcommand{\bU}{U} 

\newcommand{\ld}{\ldots}

\DeclareMathOperator{\Df}{D\!}

\DeclareMathOperator{\modulus}{mod}
\DeclareMathOperator{\inter}{int}
\DeclareMathOperator{\cl}{cl\,}

\renewcommand{\tilde}{\widetilde}
\renewcommand{\rho}{\varrho}
\renewcommand{\phi}{\varphi}

\renewcommand{\theta}{\vartheta}

\newcommand{\Kwi}{K_{\text{well-inside}}}

\newcommand{\Koc}{K_{\text{off-crit}}}

\DeclareMathOperator{\fib}{fib}

\DeclareMathOperator{\Crit}{Crit}
\DeclareMathOperator{\CritP}{Crit_{ac}}

\DeclareMathOperator{\orb}{orb}
\DeclareMathOperator{\diam}{diam}
\DeclareMathOperator{\dist}{dist}
\DeclareMathOperator{\area}{meas}
\DeclareMathOperator{\Comp}{Comp}
\DeclareMathOperator{\PC}{PC}
\DeclareMathOperator{\Back}{Back}
\DeclareMathOperator{\Forw}{Forw}

\newcommand\ovl[1]{\overline{#1}}
\newcommand{\sm}{\setminus}

\renewcommand{\ge}{\geqslant}
\renewcommand{\le}{\leqslant}

\theoremstyle{theorem}
\newtheorem{theorem}{Theorem}[section]
\newtheorem{lemma}[theorem]{Lemma}

\newtheorem{proposition}[theorem]{Proposition}
\newtheorem{corollary}[theorem]{Corollary}

\newtheorem*{QCrigidity}{QC Rigidity for Complex Box Mappings}
\newtheorem*{Pathologies}{Pathologies of General Complex Box Mappings}
\newtheorem*{ComplexBounds}{Complex bounds for non-renormalizable complex box mappings}
\newtheorem*{ErgodicComponents}{Number of ergodic components}
\newtheorem*{ManeTheorem}{Ma\~n\'e-type theorem for box mappings}

\theoremstyle{definition}
\newtheorem{definition}[theorem]{Definition}
\newtheorem{question}{{Question}}

\theoremstyle{remark}
\newtheorem*{remark}{\textsc{Remark}}

\theoremstyle{claim}

\newtheorem*{claim}{Claim}

\usepackage{tikz-cd} 

\numberwithin{equation}{section}

\title{The dynamics of complex box mappings}
\author{Trevor Clark}
\author{Kostiantyn Drach}
\author{Oleg Kozlovski}
\author{Sebastian van Strien}
\date{{\tiny \today}}

\address{Open University, School of Mathematics and Statistics, Milton Keynes MK7 6AA, UK}
\email{trevorcclark@gmail.com}

\address{IST Austria, Am Campus 1, 3400 Klosterneuburg, Austria} 
\email{kostya.drach@gmail.com}

\address{University of Warwick, Mathematics Institute, Zeeman Building, Coventry CV4 7AL, UK}
\email{oleg.kozlovski@warwick.ac.uk}

\address{Imperial College London, Department of Mathematics, 180 Queen's Gate, London SW7 2AZ, UK}
\email{s.van-strien@imperial.ac.uk}

\thanks{{\tiny\textbf{Acknowledgments.}
The research of the second author was partially supported by the advanced grants 695\,621 ``HOLOGRAM'' and 885\,707 ``SPERIG'' of the European Research Council (ERC), which is gratefully acknowledged. We would also like to thank Dzmitry Dudko and Dierk Schleicher for many stimulating  discussions and encouragement during our work on this project, and Weixiao Shen, Mikhail Hlushchanka and the referee for helpful comments. We are grateful to Leon Staresinic who carefully read the revised version of the manuscript and provided many helpful suggestions.}
}

\begin{document}

\maketitle

\begin{abstract}
In holomorphic dynamics, complex box mappings arise as first return maps to well-chosen 
domains. They are a generalization of polynomial-like mapping, where the domain of the
return map can have infinitely many components. They turned out to be extremely useful
in tackling diverse problems. 
The purpose of this paper is:
\smallskip
\begin{enumerate}[leftmargin=1mm,labelsep=2mm,itemsep=1mm]
\item[-]  To illustrate some pathologies that can occur when a complex box mapping is
not induced by a globally defined map and when its domain  has infinitely many components, and to give conditions to avoid these issues. 
\item[-] To show that once one has a box mapping for a rational map, 
these conditions can be assumed to hold  in a very natural setting.  
Thus we call such complex box mappings  \emph{dynamically natural}. Having such box mappings is the first step in tackling many problems in one-dimensional dynamics.
\item[-] 
Many results in holomorphic dynamics rely on an interplay between 
combinatorial and analytic techniques. In this setting some of these tools are: 
\begin{itemize}[leftmargin=5mm,labelsep=1mm,itemsep=1mm,topsep=1mm]
\item[-] the Enhanced Nest (a  nest of puzzle pieces around critical points)  from \cite{KSS};
\item[-] the Covering Lemma (which controls the moduli of pullbacks of annuli)
 from \cite{KL09}; 
 \item[-] the QC-Criterion and the Spreading Principle from \cite{KSS}.
 \end{itemize}
 The purpose of this paper is to make these tools more accessible so that they can be  used as a {\lq}black box{\rq}, so one does not have to redo the proofs  in new settings.
\item[-] To give an  intuitive, but also rather detailed,  outline of the proof from \cite{KvS,KSS}  of the following results for non-renormalizable dynamically natural complex box mappings:
\begin{itemize}[leftmargin=5mm,labelsep=2mm,itemsep=1mm,topsep=1mm]
\item[-] puzzle pieces shrink to points,
\item[-]  (under some assumptions) topologically conjugate 
non-renormalizable polynomials and box mappings are quasiconformally conjugate.
\end{itemize}
\item[-] We prove the fundamental ergodic properties for dynamically natural box mappings. This leads to some necessary conditions for when such a box mapping supports a measurable invariant line field on its filled Julia set. These mappings are the analogues of Latt\`es maps in this setting.
\item[-] We prove a version of Ma\~n\'e's Theorem for complex box mappings concerning expansion along orbits of points that avoid a neighborhood of the set of critical points.
\end{enumerate} 
\end{abstract}


\newpage

\setcounter{tocdepth}{1} 
{\hypersetup{linktoc=page}
\tableofcontents
}
\section{Introduction}

\definecolor{fgreen}{rgb}{0.2,0.6,0.}


\begin{figure}[htbp]
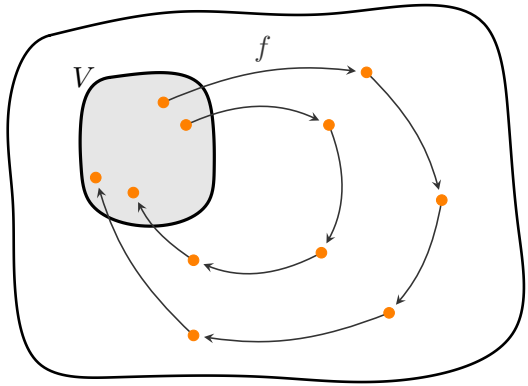

\begin{center}
\definecolor{ffxfqq}{rgb}{1.,0.5,0.}
\definecolor{ttzzqq}{rgb}{0.2,0.6,0.}
\definecolor{ttzzqq}{rgb}{0.,0.,0.}

\caption{The first return map to $V$.}
\label{Fig:FirstReturn}
\end{center}
\end{figure}


In dynamical systems, a natural and effective way to understand the 
detailed behavior of a given map $f$ is to consider the first return map under $f$ to a certain set $V$, see Figure~\ref{Fig:FirstReturn}. In this way, if the orbit of a point under iteration of $f$ intersects $V$ infinitely many times, then the first return map will send each point of intersection to the next point of intersection along the orbit. Therefore, by iterating the first return map instead of $f$, one can study properties of certain orbits of $f$ ``at a faster speed''. However, there is a trade-off: the first return map might have a rather complicated, even undesirable,\footnote{For example, a component of the domain might not properly map onto a component of the range.} structure. 

In the analytic setting, a return map 
might have the structure of a polynomial-like map. A \emph{polynomial-like map} is a holomorphic branched covering $F \colon U \to V$ of degree at least two between a pair of open topological disks $U \subset V$ such that $U$ is relatively compact in $V$. The restriction of a complex polynomial to a sufficiently large topological disk in $\C$ is an example of a polynomial-like map. Such mappings are an indispensable tool in the field of holomorphic dynamics due to their fundamental role in renormalization and self-similarity phenomena.

However, in general, the topological dynamics of a return mapping even under a polynomial cannot be described by a polynomial-like mapping.
This motivates and explains the need for a more flexible class of mappings, namely, \emph{complex box mappings} (see Figure~\ref{Fig:ExampleBoxMapping}):

\begin{definition}[Complex box mapping] 
\label{Def:BM}
A holomorphic map $F \colon \U \to \V$ between two open sets $\U \subset \V \subset \C$ is a \emph{complex box mapping} if the following holds:
\begin{enumerate}
\item
\label{It:DefBM1}
$F$ has finitely many critical points;
\item
\label{It:DefBM2}
$\V$ is the union of finitely many open Jordan disks with disjoint closures;
\item
\label{It:DefBM3}
every component $V$ of $\V$ is either a component of $\U$, or $V \cap \U$ is a union of Jordan disks with pairwise disjoint closures, each of which is compactly contained in $V$;
\item
\label{It:DefBM4}
for every component $U$ of $\U$ the image $F(U)$ is a component of $\V$, and the restriction $F \colon U \to F(U)$ is a proper map\footnote{In \cite{KvS}, where the definition of complex box mapping was given, the requirement that $F$ maps each component of $\U$ onto a component of $\V$ \emph{properly} was implicit.}.
\end{enumerate}
\end{definition}

\begin{figure}[h]
\begin{center}
\includegraphics[scale=.5, trim=0 0 0 0,clip]{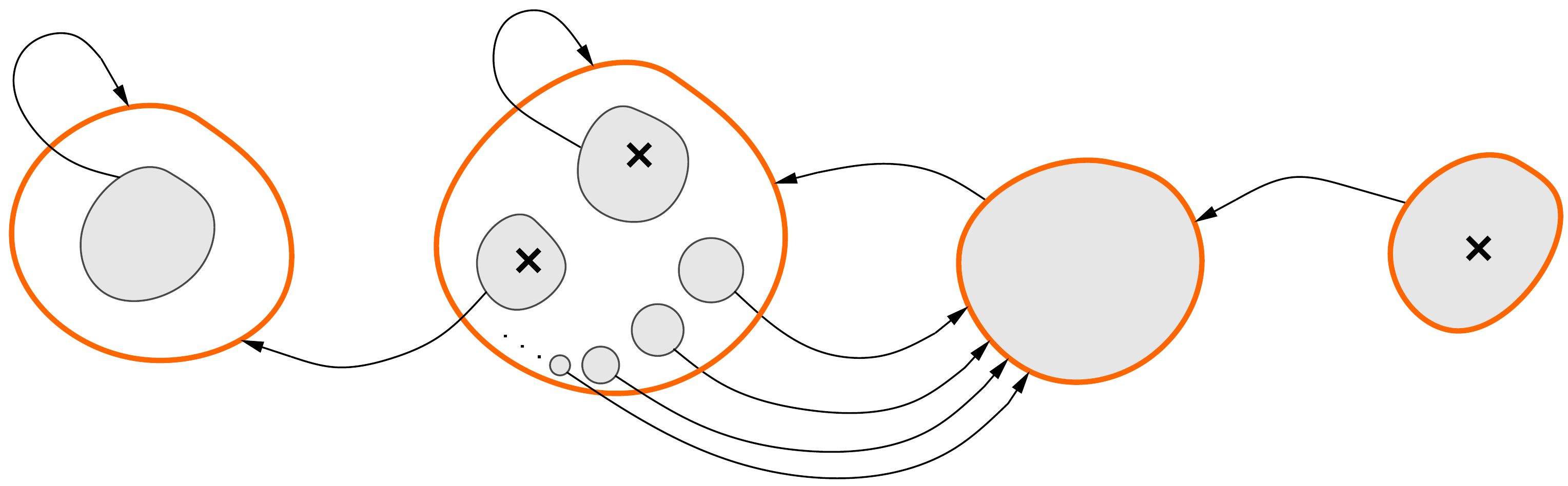}
\caption{An example of a complex box mapping $F \colon \U \to \V$. The components of $\U$ are shaded in grey, there might be infinitely many of those; the components of $\V$ are the topological disks bounded by orange curves. The critical points of $F$ are marked with crosses.} 
\label{Fig:ExampleBoxMapping}
\end{center}
\end{figure}

In the above definition of complex box mapping we assumed that each component of $\U$ and $\V$  is a Jordan domain. In some settings it is convenient to relax this, and assume that
each component $U$ of $\U$ and $V$ of $\V$
\begin{itemize}
\item[a)] is simply connected;
\item[b)] when $U\Subset V$, then $V\setminus \ovl{U}$ is a topological annulus;
\item[c)] has a locally connected boundary.
\end{itemize}
In Section~\ref{SSec:NotJordan} we will discuss why and when 
the theorems in this paper will go through in this setting. 

In one-dimensional holomorphic dynamics, it is often the case that the first step in understanding the dynamics of a family of mappings is to obtain a good combinatorial model for the dynamics in the family. Once that is at hand, one can go on to study deeper properties of the dynamics in the family, for example, rigidity, ergodic properties and the geometric properties of the Julia sets. It turns out that complex box mappings are often part of these combinatorial models.
Indeed,
recent progress on the rigidity question for a large family of non-polynomial rational maps, so-called \emph{Newton maps} \cite{DS} made it clear that complex box mappings can be effectively used in the holomorphic setting well past polynomials: by ``boxing away'' the most essential part of the ambient dynamics into a box mapping one can readily apply the existing rigidity results without redoing much of the theory from  scratch.

These developments motivated us to give a comprehensive survey of the dynamics of such mappings. The subtlety is that Definition~\ref{Def:BM} is extremely flexible and thus allows for undesirable pathologies:
\begin{Pathologies}
There exists
\begin{itemize}
\item a complex box mapping with  empty Julia set,
\item a complex box mapping with a set of positive measure that does not accumulate on the postcritical set, and
\item a complex box mapping with a wandering domain.
\end{itemize}
\end{Pathologies}
These examples are given in Section~\ref{Sec:Pathologies}. We go on to introduce a class of complex box mappings, called \emph{dynamically natural} (see Definition~\ref{Def:NaturalBoxMapping}), for which none of these pathologies can occur. This class of mappings includes many of the complex box mappings induced by rational maps and real-analytic maps (and even in some weaker sense by a broad class of $C^3$ interval maps), see Section~\ref{Sec:RealLifeExamples}. Moreover, quite often general complex box mappings induce dynamically natural complex box mappings, see Proposition~\ref{Prop:Inducing}.

We go on to give a detailed outline of the proof of qc (quasiconformal) rigidity of dynamically natural complex box mappings. This result was first proved in \cite{KvS}  relying on techniques introduced in \cite{KSS}:

\begin{QCrigidity}
Every two combinatorially equivalent non-renormalizable dynamically natural complex box mappings (satisfying some \emph{standard} necessary conditions) are quasiconformally conjugate.
\end{QCrigidity}

We state a precise version of this result in Theorem~\ref{Thm:BoxMappingsMain2}. For the definition of combinatorial equivalence of box mappings see Definition~\ref{DefA:CombEquivBoxMappings}. A dynamically natural complex box mapping is \emph{non-renormalizable} if none of its iterates  admit a polynomial-like restriction with connected filled Julia set; renormalization of box mappings is discussed in detail in Section~\ref{subsec:renorm}. 

Since the proofs in \cite{KvS} and \cite{KSS} are quite involved, and part of \cite{KSS} only deals with real polynomials, 
we will give a quite detailed outline of the proof of quasiconformal rigidity for complex box mappings (our Theorem~\ref{Thm:BoxMappingsMain2}). We will pay particular attention to the places in~\cite{KvS} where the assumption that the complex box mappings are dynamically natural is used while applying the results from~\cite[Sections 5-6]{KSS}. 

The main technical result that underlies qc rigidity are the complex or {\it a priori} bounds. These results give compactness properties for certain return mappings and control on the geometry of their domains.

\begin{ComplexBounds}
Suppose that $F\colon\U\to\V$ is a non-renormalizable dynamically natural complex box mapping. Then there exist arbitrarily small combinatorially defined critical puzzle pieces $P$ (components of $F^{-n}(\V)$ for some $n\in\mathbb N$) with complex bounds.
\end{ComplexBounds}

See Theorem~\ref{Prop:UC} for a precise statement of what we mean by complex bounds.  	
Roughly speaking, 
for a combinatorially defined sequence ${\V_n}$  of pullbacks of $\V$,
complex bounds express how well return domains to $\V_n$ are inside of $\V_n$ in terms of a lower bound of the modulus of the corresponding annulus. 

One of the first applications of complex bounds in one-dimensional dynamics was in the study of the renormalization of interval mappings. Suppose that $f$ is unimodal mapping with critical point at $0$. One says that $f$ is {\em infinitely renormalizable} if there exist a sequence $\{I_i\}$ of intervals about $0$ and an increasing sequence $\{s_i\},$ $s_i\in\mathbb N$, so that $s_i$ is the first return time of 0 to $I_i$ and $f^{s_i}(\partial I_i)\subset\partial I_i$. For certain analytic infinitely renormalizable mappings with bounded combinatorics, Sullivan \cite{Sullivan} proved that for all $i$ sufficiently large, we have that the first return mapping $f^{s_i}\colon I_i\to I_i$  extends to a polynomial-like mapping $F_i\colon U_i\to V_i$ (see also \cite{RealBook}). Moreover, he proved that 
 one has the following {\em complex bounds}: 
there exist  a constant $\delta>0$ and $i_0$ so that  $\modulus(V_i\sm \ovl U_i)>\delta$ for all $i\ge i_0$.   In fact, here $\delta$ is universal in the sense that it can be chosen so that it only depends on the order 
of the critical point of $f$, but $i_0$ does depend on $f$.  This property is known as {\em beau bounds}\footnote{Beau, from French \emph{beautiful, nice}, is a mixed acronym ``{\bf a} priori {\bf b}ounds that are {\bf e}ventually {\bf u}niversal''.}.  If there are several critical points we say that such a $\delta$ is beau if it depends only on the number and degrees of the critical points of $f$ (and not on $f$ itself).

In fact, the complex bounds given by Theorem~\ref{Prop:UC} for non-renormalizable box mappings are also beau for persistently recurrent critical points, see Section~\ref{SSec:Rec} for the definition.
 For reluctantly recurrent or non-recurrent critical points
the estimates are not beau -- they depend on the initial geometry of the box mapping, and there is no general mechanism for them to improve at small scales.

We will not go into the history of results on complex bounds both for real and 
complex maps, but refer for references to \cite{CvST}. That paper deals with 
real polynomial maps and shows that  beau complex bounds hold
in both the finitely renormalizable and the infinitely renormalizable cases, regardless
whether critical points have even or odd order, or are persistently recurrent or not. This result
is proved in \cite{CvST}.  In fact, in \cite{CvST} these complex bounds are also established for real analytic maps (and even more general maps). 
Note that in the  infinitely renormalizable case,  such bounds cannot be obtained from Theorem~\ref{Prop:UC}.  Indeed, in that case puzzle pieces do not shrink to points, and so one has to  restart the initial puzzle partition at deeper and deeper levels.  It turns out that  for  non-real infinitely renormalizable polynomials such complex bounds do not hold in general \cite{milnor-non-locally}.  

As discussed below, complex bounds have many applications:  that high renormalizations of infinitely renormalizable maps belong to a compact family of polynomial-like maps; local connectivity of the Julia sets;  absence of measurable invariant line fields; quasiconformal rigidity for real polynomial mappings with real critical points;  density of hyperbolicity in real one-dimensional dynamics, etc.

In addition to their use in proving quasiconformal rigidity, we use complex bounds to study ergodic properties of complex box mappings. For example, we prove:

\begin{ErgodicComponents}
If $F \colon \U \to \V$ is a non-renormalizable dynamically natural complex box mapping, then 
for each ergodic component $E$ of $F$ there exist one or more critical points
$c$ of $F$ so that $c$ is a Lebesgue density point of $E$. In particular, the number of ergodic components of $F$ is bounded above by the number of critical points of $F$.
\end{ErgodicComponents}

The analogous result in the real case was proved in \cite{vSV}. For rational maps there is the theorem of Ma\~n\'e~\cite{Man} which states that each forward invariant set of positive Lebesgue measure
accumulates to the $\omega$-limit set of a recurrent critical point. The above result strengthens
this in the setting of non-renormalizable dynamically natural complex box mappings. 

The above result is proved in Corollary~\ref{cor:num erg comps} of Theorem~\ref{Thm:ErgodicNatural}. In that theorem, we prove some fundamental ergodic properties of non-renormalizable dynamically natural complex box mappings. Our study of the ergodic properties of such box mappings leads us to some necessary conditions for when such a box mapping supports an invariant line field on its filled Julia set, see Proposition~\ref{prop:lattes-description}.  Complex box mappings with a measurable invariant line field we call {\em Latt\`es box mappings}. Their properties are analogous to those of Latt\`es rational maps. Latt\`es box mappings cannot arise for complex box mappings induced by polynomials, real maps or Newton maps; however, we give an example of a dynamically natural complex box mapping that is Latt\`es, see 
Proposition~\ref{prop:lattes}.

The rigidity theorem stated above deals only with non-renormalizable complex box mappings. In the setting of dynamically natural mappings, if a map is non-renormalizable, then all its periodic points are repelling, see Section~\ref{subsec:renorm}. However, it is possible, and in fact quite useful to work with complex box mappings that do admit attracting or neutral periodic points. Some results in this direction were obtained in \cite[Section 3]{DS}. In this paper, we push it a bit further and establish in Section~\ref{Sec:Mane} a Ma\~n\'e-type theorem for complex box mappings. For non-renormalizable maps this result reads as follows:

\begin{ManeTheorem}
Let $F\colon \U\to \V$ be a non-renormalizable dynamically natural complex box mapping so that {\em each} component of $\U$  is  either {\lq}$\delta$-well-inside{\rq} or equal to a component of $\V$, see equation (\ref{eq:maneassump}).
For each neighborhood $\mathcal B$ of $\Crit(F)$ and each $\kappa>0$ there exist $\lambda>1$ and $C>0$ so that for all $k\ge 0$
and each $x$ so that $x,\dots,F^{k-1}(x)\in \U\setminus \mathcal B$ and $d(F^k(x),\partial \V)\ge \kappa$, one has 
$$|\Df F^k(x)|\ge C \lambda^k.$$
\end{ManeTheorem} 

In fact, we prove a slightly more general version of this theorem that does not depend on $F$ being non-renormalizable, see Theorem~\ref{Thm:Mane}. Also, this version of the theorem does not require that the orbits of points are bounded away from $\omega(c)$ for $c\in\Crit(F)$ as in the classical Ma\~n\'e Theorem for rational maps, see \cite{Man}.

When a holomorphic mapping induces a complex box mapping, such expansivity results are useful in the study of the measurable dynamics of the mapping and the fractal geometry of their (filled) Julia sets.
For example, this result can be applied to rational maps for which there is an induced complex box mapping that contains the orbits of all recurrent critical  points intersecting the Julia set, see~\cite{Dr}.

\subsection{Some history of the notion}
\label{sec:history} 

The idea of considering successive first return maps to neighborhoods of a critical point of {\em an interval map} 
is extremely natural, and was used extensively to show absence of wandering intervals and to obtain various
metric properties. The notion of {\lq}box mapping{\rq} was implicitly used in papers such as  \cite{BL,GJ, J, JS,MMvS,NvS}\footnote{To the best of our knowledge, the term ``box mapping'' was introduced in \cite{GJ}, but with a meaning slightly different from ours.}
and the terminology {\lq}nice interval{\rq} (i.e.\ an interval $V$ so that $f^n(\partial V)\cap \inter(V) = \emptyset$ for $n\in\mathbb N$)  was introduced in Martens' PhD thesis \cite{M}.  A natural way 
of obtaining such intervals is by taking intervals in the complement of $\cup_{j=0}^n f^{-j}(p),$ where $p$ is a periodic point and $n\ge 1$.

For {\em complex mappings}, box mappings are of course already implicit in Julia and Fatou's work, after all that is how you can show 
that the Julia set of a  quadratic map with the escaping critical point forms a Cantor set: in this case, one sees that the filled Julia set of the polynomial is the filled Julia set of a complex box mapping with no critical point and whose domain has exactly two components. More generally, Douady and Hubbard went further by introducing the notion of a {\em polynomial-like mapping} where $\U$ and $\V$ each consist of just one component. Using this notion they were able to explain some of the similarities that one sees between the Julia sets of different mappings, similarities within the Mandelbrot set, and the appearance of sets which look like the Mandelbrot set in various families of mappings \cite{PolyLike}. One beautiful property of polynomial-like mappings
is the Douady--Hubbard Straightening Theorem:  a polynomial-like mapping $F$ on a neighborhood of its filled Julia set $K(F)$ is {\em hybrid} conjugate to a polynomial; that is, there exist a polynomial $P$, a neighborhood $U'$ of $K(F)$ and a quasiconformal mapping $H$ so that 
$P\circ H(z) = H\circ F(z)$ for $z\in U'$ and $\bar\partial H$ vanishes on $K(F)$. When the Julia set of $F$ is connected, $P$ is unique. Thus from a topological point of view, the family of polynomial-like mappings is no richer than the family of polynomials.

Suppose one has a unicritical map $F$ with a critical point $c$ so that some iterate $F^s$ maps
some domain $\U\ni c$ as a branched covering onto $\V\Supset  \U$. If $F^{is}(c)\in \U$ 
for all $i\ge 0$ then $F$ is called {\em renormalizable} and $F^s\colon \U\to \V$
is a  polynomial-like map. When $F\colon \U\to \V$ is not renormalizable, then there exists an element in the forward orbit of $c$ that intersects the annulus $\V \sm \U$. In order to keep track of the full forward orbit of critical points, one would need to include in the domain of $F$  all components of the domain of the return mapping which intersect the postcritical set. This is a typical use of a complex box mapping.
If there are at most a finite number of such components, then the analogy between complex box mappings and polynomial-like mappings is strongest. 
Such mappings appeared in, for example  Branner--Hubbard \cite{BH}, Yoccoz's work on puzzle maps \cite{HY,Mil}, in work 
showing that the Julia set of a hyperbolic polynomial map is conjugate to a subshift of finite type \cite{Jak69}, and was widely used in the case of real interval maps.  They are also called \emph{polynomial-like box mappings} or \emph{generalized polynomial-like mappings}.  This terminology was introduced by Lyubich in the early 90's, who suggested to study maps that are non-renormalizable in the Douady--Hubbard sense as instances of renormalizable maps in the sense of such generalized renormalizations. This renormalization idea turned out to be extremely fruitful, see for example 
 \cite{BKNS,Kozlovski-thesis, LvS-lc, LvSBox,  Lyubich - measure, Lyu, LM,  Sm, vS, Yar}.
However, in general, one needs to consider complex box mappings whose domains have infinitely many components, and even allow for several components in $\V$. This motivates Definition~\ref{Def:BM}. By allowing $\U$ to have infinitely many components this notion becomes very flexible at the expense of admitting pathological behavior, which we will discuss in Section~\ref{Sec:Pathologies}. In the literature, complex box mappings are sometimes called \emph{puzzle mappings}, or \emph{R-mappings} (where `R' stands for \emph{r}eturn; see, for example, \cite{ALdM}). 

Yoccoz gave a practical way of constructing complex box mappings induced by polynomials.
Using the property that periodic points have {\em external rays} landing on them, Yoccoz introduced what are now known as {\em Yoccoz puzzles} \cite{HY,Mil}. In this case $\V$ consists of disks whose boundary consists of pieces of external rays and pieces of equipotentials.  
Yoccoz used these puzzles to prove local connectivity of the Julia sets of non-renormalizable quadratic polynomials, $P_c\colon z \mapsto z^2+c$, and by transferring the bounds that he used to prove local connectivity of $J(P_c)$ to the parameter plane, he proved local connectivity of the Mandelbrot set at parameters $c$ for which $P_c$ is non-renormalizable \cite{HY}.

One very important step in Yoccoz' result is to obtain lower bounds for the moduli of certain annuli which one 
encounters when taking successive first return maps to puzzle pieces containing the critical point (the {\em principal nest}). Such a property 
is usually called {\em complex bounds} or {\em a priori bounds}. Yoccoz was able to obtain these bounds for 
 non-renormalizable polynomial maps with a unique quadratic critical point that is recurrent: here one uses that the first return map 
 to a puzzle piece has at least two components intersecting the critical orbit (this is related to the notion
 of \emph{children} that we will discuss later on). 

It was clear from the early 1990s that complex methods would have wide applicability in one-dimensional dynamics, even when considering interval maps. In the 1990s complex bounds were established in the setting of real unimodal mappings, for certain infinitely renormalizable mappings in \cite{Sullivan}, and for unimodal polynomials in \cite{LvS-lc}. These results depended on additional properties of interval mappings. Soon further results were established for the quadratic family $z\mapsto z^2 + c$ \cite{Lyu}, and for classes of real-analytic mappings with a single critical point \cite{McM, McM2, LY, LvSBox}. Such results led to fantastic progress for real quadratic and unicritical maps: 
\begin{itemize}
\item local connectivity of Julia sets \cite{LvS-lc, Lyu,LY,LvSBox};
\item monotonicity of entropy \cite{milnor-thurston,Do2,Tsuji};
\item density of hyperbolicity in the {\em real quadratic family} \cite{Lyu, GS} (via quasiconformal rigidity);
\item hyperbolicity of renormalization and the Palis Conjecture for mappings with a non-degenerate critical point \cite{Lyu2, Lyu3, Lyu4, ALdM}. 
\end{itemize}

While some results held for unimodal maps with a higher
degree critical point {\it e.g.} \cite{Sullivan, LvS-lc, McM, McM2}, in general, it took some time for the theory for  interval mappings with several critical points (or a critical point with degree greater than two)
and for non-real polynomials,  to catch up to that of quadratic mappings. In the setting of interval mappings, building on the work of \cite{Sullivan} (described in \cite{RealBook}) and applying results of \cite{McM2},  complex bounds were obtained for certain infinitely renormalizable multicritical mappings together with a proof of exponential convergence of renormalization for such mappings \cite{Sm-cb, Sm-universality}. Towards the goal of proving density of hyperbolic mappings for one-dimensional mappings, \cite{Kozlovski-thesis}
proved density of hyperbolicity for smooth unimodal maps and 
\cite{Sh} proved
$C^2$ density of Axiom A interval mappings  by first proving complex bounds and certain local rigidity properties for certain multimodal interval mappings. 

To prove quasiconformal rigidity, the techniques from \cite{Lyu,GS} rely on the map to
be unimodal and the critical point to be quadratic, because that implies that the moduli
of certain annuli grow exponentially. To deal with general real multimodal maps,  \cite{KSS} introduced a sophisticated combinatorial construction (the Enhanced Nest). Using this, together with additional results for interval maps, complex bounds and quasiconformal rigidity for real polynomial mappings with real critical points were proved in \cite{KSS}.  This implies density of hyperbolicity for such polynomials. 
In  \cite{KSS-density}, complex box mappings in the non-minimal case were constructed, 
and these were used to study {\lq}global perturbations{\rq}
of real analytic maps thus establishing density of hyperbolicity for one-dimensional mappings in the $C^k$ topology for $k = 1,2,\dots, \infty, \omega$. One can also show that one has density of hyperbolicity for real transcendental maps and within full families of real analytic maps \cite{RvS2,CvS2}. 
One can apply the techniques of \cite{KSS} to prove complex bounds and qc rigidity for real analytic mappings and even in some sense for smooth interval maps, see \cite{CvST,CvSlong}.  The rigidity result from \cite{KSS} can also 
be used to extend the monotonicity of entropy for cubic interval maps \cite{milnor-tresser}
to  general multimodal interval maps (with only real critical points) \cite{BvS,Kozlovski-entropy}
and to show that zero entropy maps are on the boundary of chaos \cite{CT}.

The results mentioned above rely on complex bounds which were initially only available for 
real maps. Complex mappings lack the order structure of the real line, and some tools 
that are very useful for real mappings such as real bounds \cite{vSV} and 
Poincar\'e disks cannot be used to prove complex bounds for complex mappings. Thus, new analytic tools were required, namely 
the Quasiadditivity Law and the Covering Lemma of \cite{KL09}. 
Using this new ingredient, the theory for non-renormalizable complex polynomials was brought up to the same level  as that of real quadratic polynomials in the unicritical setting in \cite{KLUnicr, AKLS, ALS} and  in the general non-renormalizable polynomial case in \cite{KvS}. 
 Going beyond mappings with non-degenerate critical point, for real analytic unimodal maps with critical points of even degree, exponential convergence of renormalization was proved in \cite{AL} and the Palis Conjecture in this setting was proved in \cite{C}, see also 
\cite{BSvS}. The Covering Lemma was also used in \cite{KvS} to prove local connectivity of Julia sets, absence of measurable invariant line fields and qc rigidity of non-renormalizable polynomials.

\subsection{The Julia set and filled Julia set of a box mapping} 

Given a complex box mapping $F \colon \U \to \V$, the \textit{filled Julia set of $F$} is defined as $$K(F) := \bigcap_{n \ge 0} F^{-n}(\V).$$
The set $K(F)$ consists of \textit{non-escaping} points: those whose orbits under $F$ never escape the domain of the map. 

As for the \emph{Julia set of $F$}, there are a few candidates which one could take as the definition. No matter which definition one uses, the properties of the Julia set in the context of complex box mappings are often quite different from those of the Julia set of a rational mapping, mainly due to the fact that $\U$ can have infinitely many components. We discuss this in Section~\ref{sec:Julia}, where we define two sets 
$$J_\U(F) := \partial K(F)\cap\U\quad\text{ and }\quad J_K(F) := \partial K(F)\cap K(F),$$
each of which could play the role of the Julia set.

\subsection{Puzzle pieces of box mappings} 
One of the most important properties of complex box mappings is that they provide a kind of ``Markov partition" for the dynamics of $F$ on its filled Julia set: the components of $\V$ and their iterative pullbacks are analogous to Yoccoz puzzle pieces for polynomials. Having such a topological structure is a starting point for the study of many questions in dynamics. So, following \cite{KvS, KSS}, for a given $n \ge 0$, a connected component of $F^{-n}(\V)$ is called a \textit{puzzle piece of depth $n$} for $F$. 
For $x\in \V$, the {\em puzzle piece of depth $n$ containing $x$}  is the connected component of $F^{-n}(\V)$ 
containing $x$ (when $x$ escapes $\U$ in less than $n$ steps, this set will be empty). 

It is easy to see that $F$ properly maps a puzzle piece of depth $n+1$ onto a puzzle piece of depth $n$, for every $n \ge 0$, and any two puzzle pieces are either nested, or disjoint. In the former case, the puzzle piece at larger depth is contained in the puzzle piece of smaller depth.

\begin{remark}
In the literature, the puzzle pieces that are constructed for polynomials or rational maps are usually assumed to be closed. In the context of polynomial-like mappings, or generalized polynomial-like mappings, similar to ours, the puzzle pieces are often defined as open sets. 
\end{remark}

\subsection{Structure of the paper}

The paper is organized as follows.

In the last subsection of the introduction, Section~\ref{SSec:Defns}, we introduce some notation and terminology that will be used throughout the paper. We encourage the reader to consult this terminological subsection if some notion was used in some section of the paper but was not defined earlier in that section.

In Section~\ref{Sec:RealLifeExamples}, we give several examples of complex box mappings that we have alluded to above and which appear naturally in the study of real analytic mappings on the interval and rational maps on the Riemann sphere. The purpose of this section is to give additional motivation for Definition~\ref{Def:BM} and to convince the reader that in a great variety of setups the study of a dynamical system can be almost reduced to the study of an induced box mapping, and hence one can prove rigidity and ergodicity properties of various dynamical systems simply by ``importing'' the general results on box mappings discussed in this paper. We end Section~\ref{Sec:RealLifeExamples} by posing some open questions about box mappings.

In Section~\ref{Sec:Pathologies}, we present some of the pathologies that can occur for a box mapping due to the fairly general definition of this object. The discussion in that section will lead to the definition of a dynamically natural box mapping in Section~\ref{ASec:Subsets} for which such undesirable behavior cannot happen. These box mappings arise naturally in the study of both real and complex one-dimensional maps, hence the name.

In Section~\ref{Sec:CombRen}, we recall the definitions and properties of some objects of a combinatorial nature that can be associated to a given box mapping (fibers, recurrence of critical orbits, etc.). We end that section with the definitions of renormalizable and combinatorially equivalent box mappings. 

In Section~\ref{Sec:QCStatement}, we state the main result on the rigidity and ergodic properties of non-renormalizable dynamically natural box mappings, Theorem~\ref{Thm:BoxMappingsMain2}. This result includes shrinking of puzzle pieces, a necessary condition for a complex box mapping to support an invariant line field on its Julia set, and quasiconformal rigidity of topologically conjugate complex box mappings, with a compatible ``external structure". This result was proven in \cite{KvS} for box mappings induced by non-renormalizable polynomials, and the proof extends to our more general setting. We will elaborate on its proof with a two-fold goal: on one hand, to emphasize on several details and assumptions that were not mentioned or were implicit in \cite{KvS}, and on the other hand, to make the underlying ideas and proof techniques better accessible to a wider audience beyond the experts. 

The rest of the paper, namely Sections~\ref{Sec:Ing}--\ref{Sec:ILF}, is dedicated to reaching this goal. The structure of the reminder of the paper, starting from Section~\ref{Sec:Ing}, is outlined in Section~\ref{SSec:Outline}, and we urge the reader to consult that subsection for details. 

In Section~\ref{Sec:ILF} we build on the description of the ergodic properties
of complex box mappings to give conditions for the absence
of measurable invariant line fields for box mappings. In Section~\ref{sec:lattes} 
we give examples of box mappings which have invariant line fields and which are the analogue
of rational Latt\`es maps in this setting. We also derive a Ma\~n\'e Theorem showing 
that one has expansion for orbits of complex box mappings which stay away from critical points, see Section~\ref{Sec:Mane}. These results were not explicit in the literature. 

Finally, Appendix~\ref{A1} contains some ``well-known to those who know it well'' facts, in particular, about various return constructions that are routinely used to build complex box mappings.

\subsection{Notation and terminology}
\label{SSec:Defns}

We refer the reader to some standard reference background sources in one-dimensional real \cite{RealBook} and complex \cite{MiIntro, LyuBook} dynamics, as well as to our Appendix~\ref{A1}.

\subsubsection{Generalities}

We let $\C$ denote the complex plane, $\Cc$ be the Riemann sphere, we write $\disk = \disk_1$ for the open unit disk and $\disk_r$ for the open disk of radius $r$ centered at the origin. A \emph{topological disk} is an open simply connected set in $\C$. A {\em Jordan disk} is a topological disk whose topological boundary is a Jordan curve. 

An \emph{annulus} is a doubly connected open set in $\C$. The conformal modulus of an annulus $A$ is denoted by $\modulus(A)$. 

By a \emph{component} we mean a connected component of a set. For a set $B$ and a connected subset $C \subset B$ or a point $C\in B$, we write $\Comp_C B$ for the connected component of $B$ containing $C$. 

We let $\ovl A$ and $\cl A$ stand for the closure of a set $A$ in $\C$ (we will use both notations interchangeably, depending on typographical convenience); $\inter A$ is the interior of the set $A$; $\# A$ is the cardinality of $A$. An open set $A$ is \emph{compactly contained} in an open set $B$, denoted by $A \Subset B$, if $\ovl A \subset B$.

We let $\diam(A)$ be the Euclidean diameter of a set $A \subset \C$. A topological disk $P \subset \C$ has \emph{$\eta$-bounded geometry at $x \in P$} if $P$ contains the open round disk of radius $\eta \cdot \diam(P)$ centered at $x$. We simply say that $P$ has \emph{$\eta$-bounded geometry} if it has $\eta$-bounded geometry with respect to some point in it.

For a Lebesgue measurable set $X \subset \C$, we write $\area(X)$ for the Lebesgue measure of $X$.

\subsubsection{Notation and terminology for general mappings}

In this paper, we will restrict our attention to the dynamics of holomorphic mappings. For a given holomorphic map $f$, $f^n$ denotes the $n$-th iterate of the map. We let $\Crit(f)$ denote the set of critical points of $f$, and $\PC(f)$ stands for the union of forward orbits of $\Crit(f)$, {\em i.e.}
$$\PC(f) = \left\{f^n(c) : c \in \Crit(f), n \ge 0\right\}$$ 
as long as $n$ is chosen so that $f^{n}(c)$ is well-defined.  
For a point $z$ in the domain of $f$, we write $\orb(z)$ for the forward orbit of $z$ under $f$, {\em i.e.}\ $\orb(z) = \{f^n(z) : n \ge 0\}$. We also write $\omega(z)$ for the $\omega$-limit set of $\orb(z)$, i.e.\ $\omega(z) = \bigcap_{k \ge 0} \ovl{\{f^n(z) : n > k\}}$. 

For a holomorphic map $f$ and a set $B$ in the domain of $f$, a \emph{line field} on $B$ is the assignment of a real line through each point $z$ in a positive measure subset $E \subset B$ so that the slope is a measurable function of $z$. A line field is \emph{invariant} if $f^{-1}(E) = E$ and the differential $\text{D}_z f$ transforms the line at $z$ to the line at $f(z)$. 

A set $A \subset U$ is \emph{minimal} for a map $f \colon U \to V$, $U \subset V$ if the orbit of every point $z \in A$ is dense in $A$.  Periodic orbits are examples of minimal sets. 

\subsubsection{Nice sets and return mappings}
\label{SSSec:Return}

For a map $f$, an open set $B$ in the domain of $f$ is called \emph{wandering} if $f^k(B) \neq f^\ell(B)$ for every $k \neq \ell$ and $B$ is not in the basin of a periodic attractor. The set $B$ is called \emph{nice} if $f^k(\partial B) \cap \inter B = \emptyset$ for all $k \ge 1$; it is \emph{strictly nice} if $f^k(\partial B) \cap \ovl B = \emptyset$ for all $k \ge 1$.

Given a holomorphic map $f \colon U \to V$ and a nice open set $B \subset U$, define 
\[
\DomE(B) := \{z \in U \colon \exists k \ge 1, f^k(z) \in B\}, \quad \Dom(B) := \DomE(B) \cap B, \quad \DomL(B):=\DomE(B) \cup B.
\]
The components of $\DomE(B)$, $\Dom(B)$ and $\DomL(B)$ are called called, respectively, \emph{entry, return and landing domains (to $B$ under $f$)}. 

For a point $z \in U$, we will use the following short notation:
\[
\DomE_z(B) := \Comp_z \DomE(B), \quad \DomL_z(B) := \Comp_z \DomL(B).
\]

The \emph{first entry map} $E_B \colon \DomE(B) \to B$ is defined as $z \mapsto f^{k(z)}(z)$, where $k(z)$ is the minimal positive integer with $f^{k(z)}(z) \in B$. The restriction of $E_B$ to $B$ is \emph{the first return map to $B$}; it is defined on $\Dom(B)$ and is denoted by $R_B$. \emph{The first landing map} $L_B \colon \DomL(B) \to B$ is defined as follows: $L_B(z) = z$ for $z \in B$, and $L_B(z) = E_B(z)$ for $z \in \DomE(B) \sm B$. 

We will explain some basic properties of these maps in Appendix~\ref{A1}.

\subsubsection{Notation and terminology for complex box mappings}

In this paper, our object of study is a complex box mapping $F\colon\U\to\V$. Sometimes will abbreviate this terminology and simply refer to a {\em box mapping} $F$, and unless otherwise stated, by box mapping we mean a complex box mapping.

We define the puzzle pieces of depth $n$ to be the components of $F^{-n}(\V)$. If $A\subset\C$, we call a collection of puzzle pieces whose union contains $A$ a {\em puzzle neighborhood} of $A$.

The \emph{non-escaping (or filled Julia) set of $F$} is defined as $K(F) := \bigcap_{n \ge 0} F^{-n}(\V).$

A box mapping $F' \colon \U' \to \V'$ is \emph{induced} from $F \colon \U \to \V$ if $\U'$ and $\V'$ are unions of puzzle pieces of $F$ and the branches of $F'$ are compositions of the branches of $F$. 

Given a point $z \in K(F)$, a {\em nest of puzzle pieces about} $z$ or simply a {\em nest} is a sequence of puzzle pieces that contain $z$. We say that puzzle pieces {\em shrink to points} if for any infinite nest $P_1 \supset P_2 \supset P_3 \supset \ld$ of puzzle pieces, we have that $\diam(P_n)$ tends to zero as $n \to \infty$.

We will say that a complex box mapping $F\colon\U\to\V$ has {\em moduli bounds} if there exists $\delta>0$ so that for every component $U$ of $\U$ that is compactly contained in a component $V$ of $\V$ we have that
$\modulus(V\setminus\ovl U)>\delta.$ 

\subsubsection{$\delta$-nice and $\delta$-free puzzle pieces}
\label{SSSec:NiceFree}

A puzzle piece $P$ is called $\delta$-{\em nice}\label{page:nice} if for any point $z\in \PC(F)\cap P,$ whose orbit returns to $P$, we have that $\modulus(P\sm\ovl{\mathcal{L}_z(P)})>\delta$. Note that strictly nice puzzle pieces are automatically $0$-nice.

A puzzle piece $P$ is called $\delta$-{\em free}\footnote{In \cite{KvS}, the terminology $\delta$-fat was used instead of $\delta$-free. In the present paper, we prefer the latter, more friendly and more transparent terminology.}, if there exist puzzle pieces $P^-$ and $P^+$ with $P^-\subset P\subset P^+$, so that $\modulus(P^+\sm\ovl P)$ and $\modulus(P\sm\ovl{P^-})$ are both at least $\delta$ and the annulus $P^+ \sm \ovl{P^-}$ does not contain the points in $\PC(F)$.

\section{Examples of box mappings induced by analytic mappings}
\label{Sec:RealLifeExamples}

In this section, we give quite a few examples of complex box mappings that appear in the study of complex rational and real-analytic mappings.

We end this section with an outlook of further perspectives and open questions.

\subsection{Finitely renormalizable polynomials with connected Julia set, see Figure~\ref{Fig:ExampleBoxMapping2}.} 
\label{SSec:Examples}

Suppose that $P\colon\C\to\C$ is a polynomial mapping of degree $d$ with connected Julia set. 
By B\"ottcher's Theorem, $P$ restricted to $\C\setminus K(P)$ is conjugate to $z \mapsto z^d$ on $\C\setminus \overline{\disk}$ by a conformal mapping $\phi\colon\C\setminus\overline{\disk}\to\C\setminus K(P)$.
There are two foliations that are invariant under $z\mapsto z^d$: the foliation by straight rays through the origin, and the foliation by round circles centered at the origin. Pushing these foliations forward by $\phi$, one obtains two foliations of $\C\setminus K(P),$ which are invariant under $P$. The images of the round circles under $\phi$ are called {\em equipotentials} for $P$, and the images of the rays are called {\em rays} for $P$. Quite often the terminology {\em dynamic rays} is used to distinguish these rays from rays in parameter space, but we will not need to make that distinction in this paper. When constructing the Yoccoz puzzle,  an important consideration is whether the rays {\em land} on the filled Julia set (that is, the limit set of the ray on the filled Julia set consists of a single point). Fortunately, rays always land at repelling periodic points and consequently, such rays form the basis for the construction of the Yoccoz puzzle.

\begin{figure}[h]
\begin{center}
\includegraphics[scale=0.4]{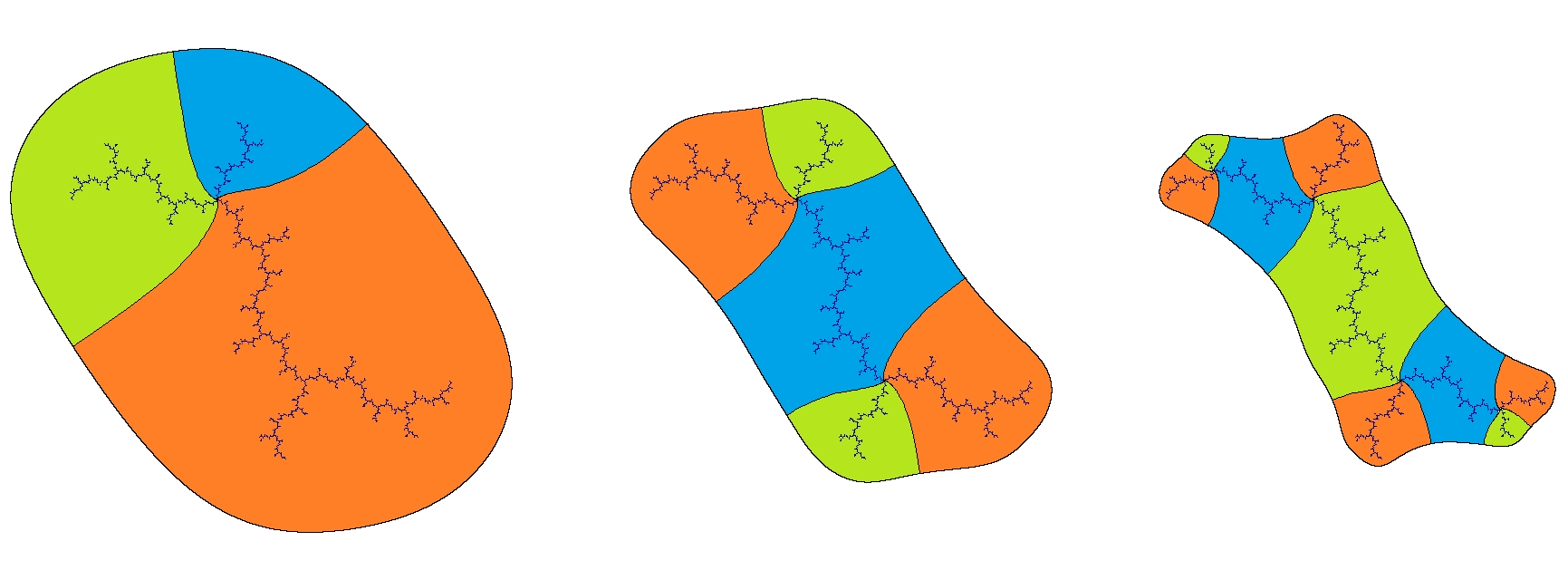}
\caption{The Yoccoz puzzle pieces of depths 0, 1 and 2 for the polynomial $z \mapsto z^2+i$ are shown in color on the left, middle and right picture respectively. In this example, $\mathcal Y^0$ consists of three pieces shown in green, blue and orange on the left. The pieces at lager depths are iterative preimages of pieces at smaller depths and are shown in corresponding colors.}
\label{Fig:ExampleBoxMapping2}
\end{center}
\end{figure}

To construct the Yoccoz puzzle, one needs to be able to  make an initial choice of an equipotential and rays which land at repelling periodic or preperiodic points and separate the plane\footnote{This can be done, for example, when all finite periodic points of the polynomial are repelling.}. The top level $\mathcal Y^0$ of the Yoccoz puzzle is then defined as the union of bounded components of $\C$ in the complement of the chosen equipotential and rays, see Figure~\ref{Fig:ExampleBoxMapping2}. For $n \ge 0$, a Yoccoz puzzle piece of depth $n$ is then a component of $P^{-n}(\mathcal Y^0)$. For unicritical mappings, $z\mapsto z^d+c$, to define $\mathcal Y^0$, one typically takes the rays landing at the dividing fixed points (fixed points $\alpha$ so that $K(P)\setminus \{\alpha\}$ is not connected), their preimages, and an arbitrary equipotential.

The main feature of Yoccoz puzzle pieces is that they are \emph{nice}, i.e.\ $P^n(\partial V) \cap V = \emptyset$ for every piece $V$ and $n \in \N$. This property guarantees that puzzle pieces at larger depths are contained in puzzle pieces at smaller depths (compare Figure~\ref{Fig:ExampleBoxMapping2}). Moreover, if all the rays in the boundary of the puzzle piece $V$ land at strictly preperiodic points, then $V$ is {\em strictly nice}: $P^n(\partial V) \cap \ovl V=\emptyset$ for $n\in\mathbb N$. In this case, all of the return domains to $V$ are compactly contained in $V$, and so the first return mapping to $V$ has the structure of a complex box mapping. This construction is the prototypical example of a complex box mapping. 

In general, if $P$ is at most finitely renormalizable, that is $P$ has at most finitely many distinct polynomial-like restrictions, then as was shown in \cite[Section 2.2]{KvS}, for such polynomials with additional property of having no neutral periodic points one can always find a strictly nice neighborhood $\V$ of the critical points of $P$ lying in the Julia set, with this neighborhood being a finite union of puzzle pieces, so that the first return mapping to $\V$ under $P$ has the structure of a complex box mapping.

\subsection{Real analytic mappings, see Figure~\ref{Fig:Dense}.} 
\label{SSec:raExamples}

For a real analytic mapping $f$ of a compact interval, in general, one does not have Yoccoz puzzle pieces. Recall that to construct the Yoccoz puzzle for a polynomial $P$ one makes use of the conjugacy between $P$ and $z^d$ in a neighborhood of $\infty$, where $d$ is the degree of $P$. Nevertheless, one can construct by hand complex box mappings, which extend the real first return mappings, in the following sense. There exists a nice neighborhood $\mathcal I\subset\mathbb R$ of the critical set of $f$ with the property that each component of $\mathcal{I}$ contains exactly one critical point of $f$, and a complex box mapping $F \colon \mathcal U\to\mathcal V$, so that the real trace of $\mathcal V$ is $\mathcal I$, and for any component $U$ of $\mathcal U$, the real trace of $U$ is a component of the (real) first return mapping to $\mathcal I$. Suppose that $f$ is an analytic unimodal mapping with critical point $c$. If there exists a nice interval $I\owns c$, so that so-called \emph{scaling factor} $|I|/|\mathcal{L}_c(I)|$ is sufficiently big, then one can construct a complex box mapping which extends the real first return mapping to $\mathcal{L}_c(I)$ by taking $\mathcal{V}$ to be a Poincar\'e lens domain with real trace $\mathcal{L}_c(I)$, i.e.\ see Figure~\ref{Fig:Dense}. It turns out that for non-renormalizable unimodal mappings with critical point of degree two or a reluctantly recurrent critical point of any even degree (see Section~\ref{SSec:Rec} for the definition), one can always find such a nice neighborhood $I$ of $c$ so that $|I|/|\mathcal{L}_c(I)|$ is arbitrarily large. For infinitely renormalizable mappings, and for non-renormalizable mappings with degree greater than 2, this need not be the case. If all the scaling factors are bounded, then the construction is trickier, see for example \cite{LvSBox}. Such complex box mappings were constructed for general real-analytic interval mappings in \cite{CvST, CvSlong}, see also 
\cite{Kozlovski-thesis} for unicritical real analytic maps with a quadratic critical point.

\begin{center}
\begin{figure}[h]
\definecolor{dgreen}{rgb}{0.0,0.5,0.}
\begin{tikzpicture}
\clip(-5,-1.4) rectangle (5,1.3);
\node[yshift=0] at (0,0) {
                \includegraphics[keepaspectratio,
                                scale=0.9]{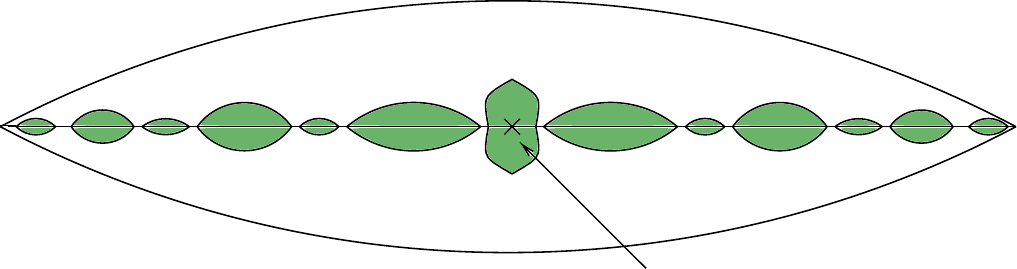}
            };
\draw (1.4,-1.32) node {$c$};
\draw[color=dgreen] (-2,0.5) node {$\U$};
\draw (-3,1) node {$\V$};
\end{tikzpicture}
\caption{An example of a complex box mapping $F \colon \U \to \V$ associated to a real-analytic unimodal mapping $f$ with critical point $c$ whose orbit is dense in the interval. $\V$ is a Poincar\'e lens domain such that $\V \cap \R$ is a nice interval containing $c$, and $\U$ is the domain of the first return map to $\V$ under $f$. In this example, $\U \cap \R$ is dense in $\V \cap \R$, i.e.\ the real trace of $\U$ ``tiles'' the real trace of $\V$.}
\label{Fig:Dense}
\end{figure}
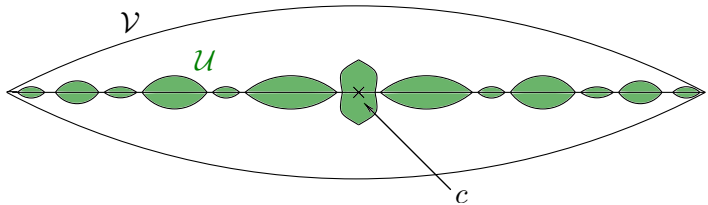
\end{center}

\subsection{Newton Maps, see Figure~\ref{Fig:NewtonPuzzles}.}

Given a complex polynomial $P \colon \C \to \C$, the \textit{Newton map} of $P$ is the rational map $N_P \colon \Cc \to \Cc$ on the Riemann sphere $\Cc$ defined as 
\[
N_P(z) := z - \frac{P(z)}{P'(z)}.
\]     
These maps are coming from Newton's iterative root-finding method in numerical analysis and hence provide examples of well-motivated dynamical systems. The roots of $P$ are attracting or super-attracting fixed points of $N_P$, while the only remaining fixed point of $N_P$ in $\Cc$ is $\infty$ and it is repelling. The set of points in $\Cc$ converging to a root is called \emph{the basin} of this root, and the component of the basin containing the root is \emph{the immediate basin}.

Newton maps are arguably the largest family of rational maps, beyond polynomials, for which several satisfactory global results are known. This progress was possible due to abundance of touching points between the boundaries of components of the root basins: these touchings provide a rigid combinatorial structure. Similarly to the real-analytic mappings, Newton maps do not have global B\"ottcher coordinates. However, the local coordinates in each of the immediate basins of roots allow one to find local equipotentials and local (internal) rays that nicely co-land from the global point of view. In this way, Newton maps posses forward invariant graphs (called \emph{Newton graphs}) that provide a partition of the Riemann sphere into pieces similar to Yoccoz puzzle pieces, see Figure~\ref{Fig:NewtonPuzzles}. Contrary to the Yoccoz construction, building the Newton puzzle is not a straightforward task. For Newton maps of degree $3$ it was carried out in \cite{Roe}, while for arbitrary degrees in was done in a series of papers \cite{DMRS, DLSS}, and independently in \cite{WYZ}. Once the Newton puzzles are constructed, one can induce a complex box mapping as the first return map to a certain nice union of Newton puzzle pieces containing the critical set; this was done in \cite{DS}, where the rigidity results for box mappings that we discuss in the present paper were applied to conclude rigidity of Newton dynamics.

\begin{figure}[htb]
\centering
\includegraphics[scale=0.22, trim=10 10 10 10, clip]{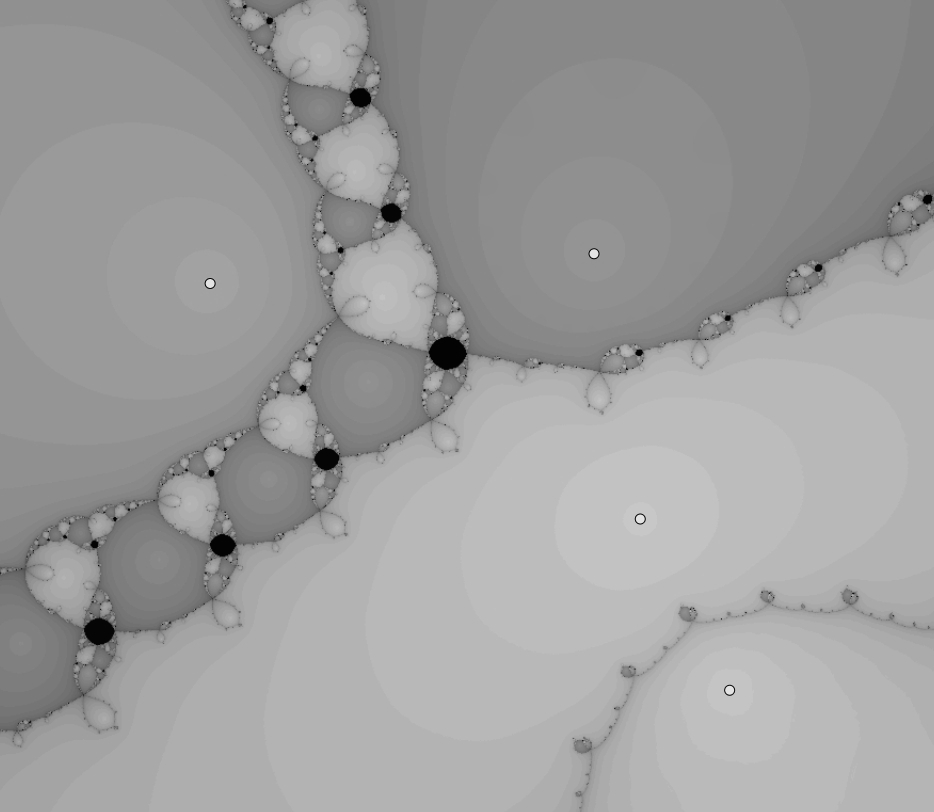}
\includegraphics[scale=0.22, trim=10 10 10 12, clip]{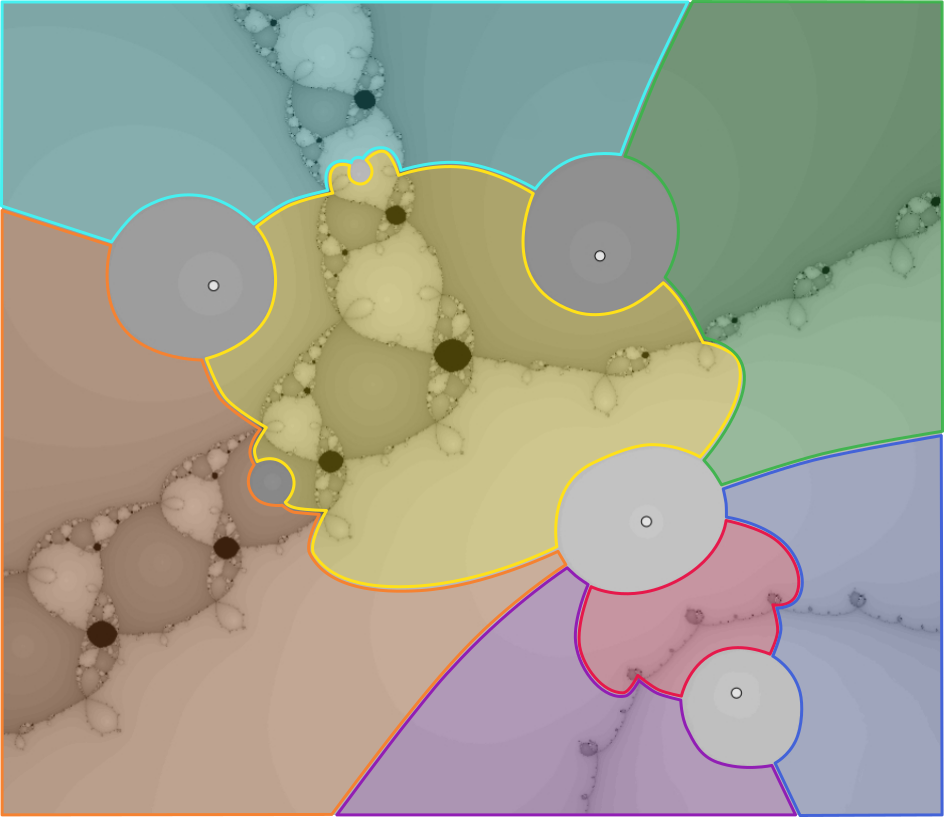}
\caption{An example of the Newton puzzle partition for the degree 4 Newton map. The dynamical plane of the map is shown on the left, with the roots of the corresponding polynomial being marked with circles. For each $n \ge 0$, the \emph{Newton puzzle partition of depth $n$} is a tiling of a neighborhood of the Julia set on the Riemann sphere with topological disks (\emph{Newton puzzle pieces}). These partitions have nice mapping properties (\emph{the Markov property}), and as $n$ grows, they become finer and finer: the number of pieces of depth $n$ grows and depth $n + 1$ pieces tile the pieces of depth $n$ (modulo truncation in the basins of roots). On the right, the Newton puzzle partition of depth $0$ is shown (7 puzzle pieces are drawn in different colors).
\vspace{-3mm}}
\label{Fig:NewtonPuzzles}
\end{figure}

\subsection{Other examples}
\label{SSec:OtherEx}

Let us move to further examples of dynamical systems existing in the literature where a complex box mapping with nice dynamical properties can be induced. 

\subsubsection{Box mappings from nice couples}
\label{SSSec:NiceCouples}

In \cite{RL}, the concept of \emph{nice couples} is defined. Given a rational map $f \colon \Cc \to \Cc$, a \emph{nice couple for $f$} is a pair of nice sets $(\widehat V, V)$ such that $V \Subset \widehat V$, each component of $\widehat V, V$ is an open topological disk that contains precisely one element of $\Crit(f) \cap J(f)$, and so that for every $n \ge 1$ one has $f^n(\partial V) \cap \widehat V = \emptyset$. Note that if $(\widehat V, V)$ is a nice couple, then $V$ is strictly nice. This implies that if $\V := V$ and $F \colon \U \to \V$ is the first return map to $\V$ under $f$, then each component of $\U$ is compactly contained in $\V$. It then follows that $F$ is a complex box mapping in the sense of Definition~\ref{Def:BM}. We see that if $f$ has a nice couple, then $f$ induces a complex box mapping $F \colon \U \to \V$ such that $\U$ is the return domain to $\V$ under $f$ and $\Crit(f) \subset \V$.

In \cite{PRL07, PRL}, the authors study thermodynamic formalism for rational maps that have arbitrarily small nice couples. They show that certain weakly expanding maps, like  
{\em topological Collet--Eckmann} rational mappings (which include rational mappings that are exponentially expanding along their critical orbits), have nice couples, and hence induce complex box mappings (see also \cite{RLS}). 

\subsubsection{Box mappings associated to Fatou components}
\label{SSec:2Fatou}

In several scenarios, one can construct a complex box mapping associated to a periodic Fatou component of a rational map.

One instance when it is particularity easy to do is when a rational map $f$ possess a \emph{fully invariant infinitely connected attracting Fatou component} $U$. 
An example of such a rational map is a complex polynomial with an escaping critical point. For such $f$, as a starting set of puzzle pieces one can use, for instance, the disks bounded by the level lines of a Green's function associated to $U$ for some sufficiently small potential (see Figure~\ref{Fig:CantorJulia}). The first return mapping under $f$ to the union of those disks that contain the non-escaping critical points would then induce a complex box mapping. The results from the current paper, namely, those described in Theorem~\ref{Thm:BoxMappingsMain2}, Sections~\ref{SSec:RigidityGeneral} and~\ref{SSec:ErgodicityDynNat}, can be then used to deduce various rigidity and ergodicity results for such $f$'s. Thus one can rephrase (and possibly shorten) 
 the proofs in \cite{QY}, \cite{YZ} and \cite{Z}, where the authors deal with the above mentioned class of rational maps,  by using the language and machinery of complex box mappings more explicitly.     

\begin{figure}[h]
\begin{center}
\definecolor{ffxfqq}{rgb}{1.,0.5,0.}
\definecolor{ffffff}{rgb}{1.,1.,1.}
\begin{tikzpicture}[line cap=round,line join=round,>=stealth,x=1.0cm,y=1.0cm]
\clip(-5.4,-2.7) rectangle (5.5,5.1);
\node[yshift=21pt,xshift=3pt] at (0,0) {
                \includegraphics[scale=0.46]{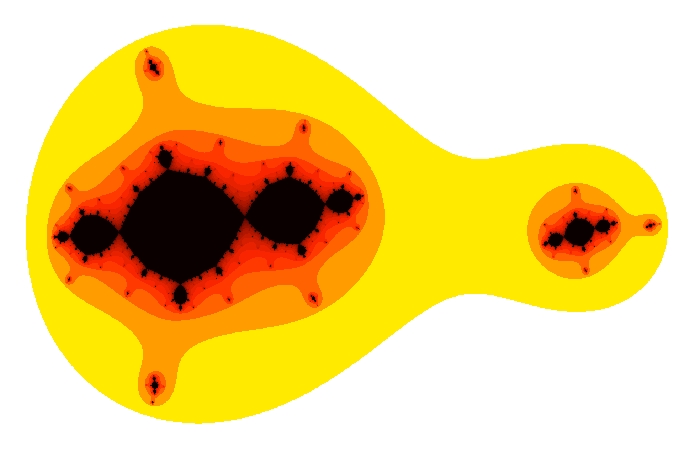}};
\draw[->,line width=1.pt,color=ffxfqq] (-1.6729580638794315,2.6535235964623425) -- (-1.6749082031115907,2.6852403077321814) -- (-1.6768560794093668,2.716976054768385) -- (-1.6787974322952002,2.748717395187883) -- (-1.6807290000631192,2.7804671256823226) -- (-1.6826475210071528,2.8122280429433504) -- (-1.684549733421329,2.8440029436626144) -- (-1.6864323755996768,2.875794624531761) -- (-1.6882921858362248,2.9076058822424375) -- (-1.6901259024250015,2.9394395134862914) -- (-1.6919302636600353,2.9712983149549697) -- (-1.693702007835355,3.003185083340119) -- (-1.6954378732449893,3.0351026153333875) -- (-1.6971345981829664,3.0670537076264215) -- (-1.6987889209433154,3.099041156910868) -- (-1.7003975798200643,3.131067759878375) -- (-1.701957313107242,3.1631363132205887) -- (-1.7034648590988772,3.1952496136291573) -- (-1.7049169560889983,3.2274104577957274) -- (-1.7063087404659367,3.2596207287078633) -- (-1.7076232863434035,3.291875429204331) -- (-1.7088381729008602,3.3241664278920457) -- (-1.7099309514421965,3.356485577478097) -- (-1.7108791732713038,3.3888247306695742) -- (-1.7116603896920721,3.4211757401735667) -- (-1.7122521520083924,3.453530458697164) -- (-1.7126320115241551,3.4858807389474555) -- (-1.7127775195432509,3.5182184336315303) -- (-1.7126662273695703,3.550535395456478) -- (-1.7122756863070039,3.5828234771293874) -- (-1.7115834476594423,3.6150745313573487) -- (-1.710567062730776,3.647280410847451) -- (-1.709204082824896,3.679432968306784) -- (-1.7074720592456925,3.7115240564424363) -- (-1.705348543297056,3.7435455279614978) -- (-1.7028110862828774,3.775489235571058) -- (-1.6998372395070473,3.807347031978206) -- (-1.696404554273456,3.8391107698900306) -- (-1.6924905818859945,3.8707723020136227) -- (-1.688072873648553,3.9023234810560705) -- (-1.6831289808650225,3.933756159724464) -- (-1.6776364548392932,3.9650621907258925) -- (-1.6715728468752558,3.9962334267674446) -- (-1.6649157082768011,4.0272617205562105) -- (-1.6576425903478194,4.058138924799279) -- (-1.6497310443922015,4.088856892203741) -- (-1.641158621713838,4.119407475476684) -- (-1.6319031973568865,4.149779917300438) -- (-1.621944154646177,4.179951300455317) -- (-1.6112613158094782,4.209895169244505) -- (-1.5998345031240433,4.239585067572226) -- (-1.587643538867125,4.268994539342704) -- (-1.5746682453159766,4.2980971284601655) -- (-1.5608884447478513,4.326866378828833) -- (-1.546283959440002,4.355275834352932) -- (-1.5308346116696818,4.383299038936687) -- (-1.5145202237141437,4.410909536484322) -- (-1.497320617850641,4.4380808709000625) -- (-1.4792156163564267,4.464786586088133) -- (-1.4601850415087536,4.491000225952757) -- (-1.440208715584875,4.516695334398159) -- (-1.4192664608620442,4.541845455328565) -- (-1.3973380996175138,4.566424132648199) -- (-1.3744034541285373,4.590404910261284) -- (-1.3504423466723674,4.6137613320720465) -- (-1.3254345995262575,4.636466941984709) -- (-1.2993600349674603,4.658495283903499) -- (-1.2722010205824539,4.679818531639739) -- (-1.2439982447264146,4.700377466013569) -- (-1.2148505938189373,4.720081540903278) -- (-1.1848594750479737,4.738838853304197) -- (-1.1541262956014684,4.756557500211651) -- (-1.1227524626673713,4.773145578620969) -- (-1.0908393834336323,4.788511185527474) -- (-1.058488465088189,4.802562417926497) -- (-1.025801114819,4.81520737281336) -- (-0.9928787398140102,4.826354147183393) -- (-0.9598227472611605,4.835910838031921) -- (-0.9267345443484061,4.843785542354274) -- (-0.8937155382636917,4.849886357145776) -- (-0.8608671361949654,4.854121379401748) -- (-0.8282907453301718,4.856398706117531) -- (-0.7960877728572591,4.856626434288435) -- (-0.7643596259641825,4.854712660909801) -- (-0.7332077118388796,4.850565482976943) -- (-0.702733437669302,4.8440929974852) -- (-0.6730382106434014,4.835203301429889) -- (-0.6442234379491154,4.823804491806344) -- (-0.6163903905908876,4.809805992258777) -- (-0.5896302513986429,4.7932155016560465) -- (-0.564017351490766,4.774204881941262) -- (-0.5396244204337828,4.752961596776874) -- (-0.5165241877942277,4.729673109825345) -- (-0.49478938313863363,4.7045268847491535) -- (-0.4744927360335245,4.677710385210755) -- (-0.4557069760454402,4.6494110748726065) -- (-0.4385048327409047,4.619816417397171) -- (-0.42295903568645254,4.589113876446927) -- (-0.40914231444861837,4.557490915684308) -- (-0.39712739859392787,4.525134998771815) -- (-0.3869870176889174,4.492233589371875) -- (-0.3787939013001127,4.458974151146972) -- (-0.3726207789940501,4.425544147759563) -- (-0.36854038033725356,4.392131042872091) -- (-0.3666254348962683,4.358922300147054) -- (-0.36694855805200177,4.3261052859656495) -- (-0.3695182121984004,4.293812711127355) -- (-0.3741719215326498,4.262031654450283) -- (-0.38071904988818517,4.230725203335631) -- (-0.38896896109861245,4.199856445184583) -- (-0.39873101899738117,4.169388467398353) -- (-0.40981458741806875,4.139284357378127) -- (-0.4220290301941816,4.109507202525105) -- (-0.4351837111591692,4.0800200902404775) -- (-0.44908799414662326,4.050786107925443) -- (-0.4635512429900359,4.0217683429811935) -- (-0.4783828215229562,3.9929298828089266) -- (-0.4933920935788194,3.9642338148098446) -- (-0.5084077500492628,3.935645979506168) -- (-0.5233837240186987,3.9071500580512772) -- (-0.5383238674576707,3.8787368424963864) -- (-0.5532321615362337,3.8503971432970534) -- (-0.5681125874244368,3.822121770908833) -- (-0.5829691262923307,3.7939015357872794);
\draw[->,line width=1.pt,color=ffxfqq] (3.535666796775557,1.395709023330939) -- (3.5289678143751497,1.4203007038470978) -- (3.522277666254425,1.4449056479073932) -- (3.515612048695255,1.46951193673595) -- (3.508983226978418,1.4941202427168985) -- (3.502403466384691,1.5187312382343698) -- (3.495885032194852,1.5433455956724946) -- (3.4894401896896787,1.5679639874154034) -- (3.4830812041499484,1.5925870858472269) -- (3.4768203408564387,1.6172155633520957) -- (3.470669865089927,1.6418500923141406) -- (3.464642042131191,1.6664913451174923) -- (3.4587491372610084,1.6911399941462815) -- (3.4530034157601572,1.7157967117846384) -- (3.447417142909414,1.7404621704166945) -- (3.442002583989557,1.7651370424265798) -- (3.4367720042813636,1.7898220001984253) -- (3.431737669065612,1.8145177161163617) -- (3.4269118436230785,1.8392248625645193) -- (3.4223067932345415,1.8639441119270292) -- (3.4179347831807787,1.8886761365880218) -- (3.4138080787425675,1.913421608931628) -- (3.409938945200685,1.9381812013419784) -- (3.40633964783591,1.9629555862032033) -- (3.4030224519290186,1.9877454358994342) -- (3.399999622760789,2.0125514228148007) -- (3.3972834256119993,2.0373742193334343) -- (3.394886125763426,2.0622144978394656) -- (3.3928199884958476,2.0870729307170253) -- (3.3910972790900415,2.1119501903502433) -- (3.3897299744061784,2.1368468754020524) -- (3.388726091811325,2.1617625724767002) -- (3.3880907731033925,2.1866961331740655) -- (3.3878290954213637,2.211646392567006) -- (3.3879461359042224,2.2366121857283785) -- (3.388446971690951,2.26159234773104) -- (3.389336679920533,2.286585713647847) -- (3.390620337731951,2.3115911185516573) -- (3.392303022264189,2.336607397515327) -- (3.3943898106562296,2.3616333856117144) -- (3.396885780047056,2.386667917913676) -- (3.3997960075756515,2.4117098294940673) -- (3.403125570380999,2.4367579554257475) -- (3.406879545602082,2.461811130781572) -- (3.4110630103778834,2.4868681906343992) -- (3.415681041847386,2.511927970057085) -- (3.420738717149574,2.5369893041224865) -- (3.4262411134234294,2.5620510279034616) -- (3.4321933078079354,2.5871119764728663) -- (3.438600377442076,2.6121709849035577) -- (3.445467399464834,2.637226888268393) -- (3.452799451015192,2.66227852164023) -- (3.4606038858466714,2.687319735834576) -- (3.46891104340597,2.712294058431213) -- (3.4777645917211393,2.737115836365405) -- (3.4872083623192935,2.7616990586192416) -- (3.4972861867275458,2.78595771417481) -- (3.508041896473011,2.809805792014201) -- (3.519519323082802,2.833157281119504) -- (3.531762298084034,2.855926170472808) -- (3.5448146530038196,2.8780264490562026) -- (3.558720219369273,2.899372105851777) -- (3.5735228287075094,2.919877129841622) -- (3.5892663125456403,2.9394555100078255) -- (3.605994502410782,2.958021235332479) -- (3.6237512298300465,2.9754882947976697) -- (3.642580326330549,2.9917706773854866) -- (3.662525623439403,3.0067823720780225) -- (3.6836117488231874,3.0204461395991635) -- (3.7057580173745155,3.0327328443579593) -- (3.7288478300648134,3.043629755154813) -- (3.7527645573306163,3.0531241547375365) -- (3.777391569608456,3.0612033258539437) -- (3.8026122373348787,3.0678545512518487) -- (3.8283099309464053,3.0730651136790623) -- (3.854368020879578,3.0768222958834) -- (3.880669877570936,3.0791133806126725) -- (3.9070988714570074,3.0799256506146944) -- (3.933538372974331,3.079246388637279) -- (3.959871752559442,3.0770628774282387) -- (3.9859823806488794,3.0733623997353874) -- (4.011753627679175,3.068132238306537) -- (4.037068864086864,3.0613596758895016) -- (4.061812383700564,3.0530330345258996) -- (4.085901964061936,3.0431783227575018) -- (4.109297532513436,3.0318689894583635) -- (4.131961701842599,3.0191815060175635) -- (4.153857084836943,3.005192343824202) -- (4.174946294284004,2.9899779742673473) -- (4.195191942971304,2.973614868736103) -- (4.21455664368637,2.9561794986195444) -- (4.233003009216729,2.937748335306754) -- (4.25049365234991,2.9183978501868317) -- (4.266991185873437,2.898204514648839) -- (4.282458222574839,2.8772448000818933) -- (4.296857375241643,2.8555951778750597) -- (4.310151256661374,2.833332119417438) -- (4.322302479621561,2.810532096098097) -- (4.33327365690973,2.7872715793061325) -- (4.343027424817732,2.7636270275053882) -- (4.3515472393630095,2.739663450207957) -- (4.358875592813641,2.715413392368023) -- (4.365065369652671,2.690903684168609) -- (4.3701694543631096,2.666161155792745) -- (4.37424073142798,2.641212637423456) -- (4.377332085330304,2.6160849592437647) -- (4.37949640055311,2.590804951436699) -- (4.3807865615794075,2.565399444185284) -- (4.381255452892235,2.5398952676725437) -- (4.380955958974591,2.5143192520814996) -- (4.379940964309528,2.488698227595185) -- (4.378263353380037,2.4630590243966233) -- (4.375976010669168,2.4374284726688384) -- (4.373131820659927,2.4118334025948553) -- (4.369783667835344,2.3863006443577004) -- (4.365984436678419,2.3608570281404) -- (4.361787011672215,2.3355293841259805) -- (4.357244252149302,2.3103445020935656) -- (4.3524020320139645,2.2853179498037566) -- (4.347289736773808,2.2604388065780068) -- (4.341934373939319,2.235692331509906) -- (4.336362951020975,2.2110637836930476) -- (4.3306024755292665,2.1865384222210285) -- (4.324679954974664,2.162101506187442) -- (4.318622396867651,2.13773829468588);
\draw [yshift=-2pt](-1.0763031646845418,-0.4587683934690605) node[anchor=north west] {$U_1$};
\draw [yshift=-2pt](4.180701471477519,0.4604011288046755) node[anchor=north west] {$U_2$};
\draw (0.7942874420479706,-0.9102902640596676) node[anchor=north west] {$\mathcal V$};
\draw (1.8908405563394435,0.7829167506551091) node[anchor=north west] {$1$};
\draw [yshift=-2pt,color=ffffff](-2.7533843983067943,0.8796714372102392) node[anchor=north west] {$-1$};
\draw [xshift=-2pt](-2.140604716790971,5.072374521265877) node[anchor=north west] {$2:1$};
\draw [xshift=-2pt](2.7938842975206564,3.330790163273535) node[anchor=north west] {$1:1$};
\draw [fill=ffffff] (-2.624378149566621,0.7506651884700658) circle (1.5pt);
\draw [fill=black] (1.9959725861721391,0.7184136262850224) circle (1.0pt);
\end{tikzpicture}
\caption{The filled Julia set of the cubic polynomial $P \colon z \mapsto c - (z^3-3z-2)/(c^2-c-2)$ for $c = -0.2 + 0.1 i$. In this example, $\Crit(P) = \{-1,1\}$ with $1$ an escaping critical point and $-1$
of period two. Since the critical point escapes, the complex box mapping $F \colon U_1 \sqcup U_2 \to \V$ can be induced using the equipotentials $\partial \V$ and $\partial U_1 \cup \partial U_2 = P^{-1}(\partial \V)$ in the basin of $\infty$. Puzzle pieces of larger depth are shown in deeper shades of red. This example illustrates a standard way of inducing a box mapping for rational maps with an infinitely connected fully invariant attracting Fatou component.}
\label{Fig:CantorJulia}
\end{center}
\end{figure}
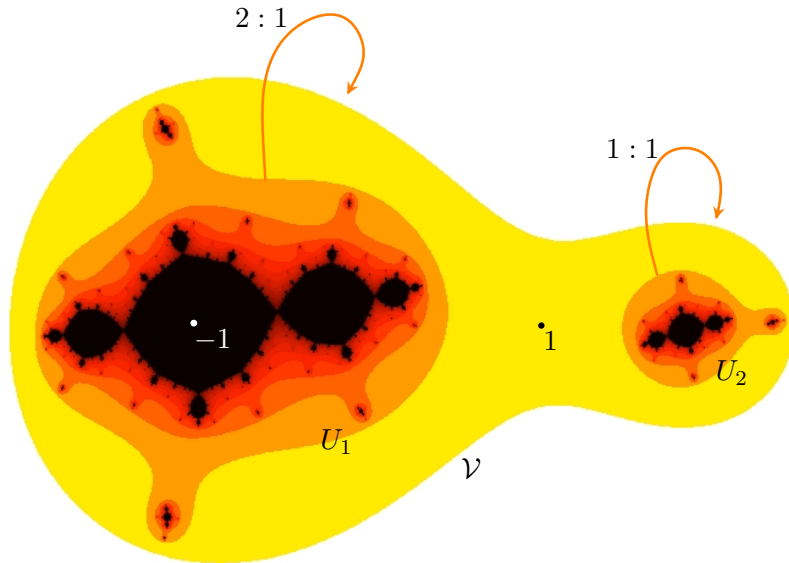

In \cite{RY08}, it was shown that for a complex polynomial $P$ every \emph{bounded Fatou component} $U$, which is not a Siegel disk, is a Jordan disk. For the proof, it is enough to consider the situation when $U$ is an immediate basin of a super-attracting fixed point (in the parabolic case, the corresponding parabolic tools should be used instead, see \cite{PR}). The authors then build puzzle pieces in a \emph{neighborhood} of $\partial U$ by using pairs of periodic internal (w.r.t.\ $U$) and external rays that co-land on $\partial U$, as well as equipotentials in $U$ and in the basin of $\infty$ for $P$. Using these puzzle pieces, it is possible to construct a strictly nice puzzle neighborhood of $\Crit(P) \cap \partial U$. The first return map to this neighborhood then defines a complex box mapping. Using this box mapping, the main result of \cite{RY08} follows from Theorem~\ref{Thm:BoxMappingsMain2} and the discussion in Section~\ref{SSec:RigidityGeneral}. A similar strategy was used in \cite{DS} in order to prove local connectivity for the boundaries of root basins for Newton maps.

\subsubsection{Box mappings in McMullen's family}
    
In \cite{QWY}, puzzle pieces were constructed for certain McMullen maps $f_{\lambda} \colon z \mapsto z^n + \lambda / z^n$, $\lambda \in \C \sm \{0\}$, $n \ge 3$.  This family includes maps
with a  Sierpi\'nski carpet Julia set.  In contrast to the previously discussed examples, where the puzzle pieces where constructed using internal and external rays, the pieces constructed in \cite{QWY} are bounded by so-called \emph{periodic cut rays}: these are forward invariant curves that intersect the Julia set in uncountably many points and whose union separates the plane. Properly truncated in neighborhoods of $0$ and $\infty$ (the latter is a super-attracting fixed point of $f_\lambda$, and $f_\lambda(0) = \infty$), these curves and their pullbacks provide an increasingly fine subdivision of a neighborhood of $J(f_\lambda)$ into puzzle pieces. Using these pieces, a complex box mapping can be induced. Similarly to the examples above, the main results of \cite{QWY} can be then obtained by importing the corresponding results on box mappings.

\subsection{Outlook and questions}

In this subsection we want to mention some research questions related to the notion of complex box mapping and to the techniques presented in this paper.

\subsubsection{Combinatorial classification of analytic mappings via box mappings} 

The classification of box mappings goes via a combinatorial construction involving itineraries with respect to curve families discussed in Section~\ref{SSec:CombEquiv}. For polynomials and rational maps one often uses 
trees and tableaux to obtain combinatorial information. This is encoded
in  the {\em pictograph} introduced in \cite{dMP} for the case where infinity is a super attracting 
fixed point whose basin is infinitely connected.  
It would be interesting to  explore the relationship in more depth. 

\subsubsection{Metric properties of analytic mappings via box mappings} 

Complex box mappings play a crucial role in the study of the measure theoretic dynamics of rational mappings and the fractal geometry of their Julia sets.  This has been a very active area of research, and here we just provide a snapshot of some of the results in these directions. Some of the natural questions in this setting concern the following:

\begin{enumerate}
\item 
\label{it:conf}
The existence and properties of conformal measures supported on the Julia set, and the existence and properties of invariant measures that are absolutely continuous with respect to a conformal measure.

\item Finding combinatorial or geometric conditions on a mapping that have consequences for the measure or Hausdorff dimension of its Julia set, {\em i.e.} if the measure is positive or zero, or whether the Hausdorff dimension is two or less than two.

\item When is the Julia set holomorphically removable\footnote{Let $\Omega \subseteq \Cc$ be a domain, $E \subset \Omega$ be a compact set, and $f \colon \Omega \sm E \to \Cc$ be a holomorphic map. A set $E$ is called \emph{(holomorphically) removable} if $f$ extends to a holomorphic mapping on the whole $\Omega$.}?

\item
\label{it:exps}
 There are several different quantities which are related to the complexity of a fractal or expansion properties of a mapping on its Julia set, among them, the Hausdorff dimension, the hyperbolic dimension, and the Poincar\'e exponent. While it is known that these quantities are not always all equal \cite{AL3}, in many circumstances they are, and it would be very interesting to characterize those mappings for which equality holds.
\end{enumerate}

The most complete results are known for rational mappings that are weakly expanding on their Julia sets, for example see \cite{PRL, RLS}. Both conformal measures and absolutely continuous invariant measures as in \eqref{it:conf} are well-understood. It is known that whenever the Julia set of such a mapping is not the whole sphere that it has Hausdorff dimension less than two, and  for such mappings all the aforementioned quantities in \eqref{it:exps} are equal. Moreover, when additionally such a mapping is polynomial, its Julia set is removable.

For infinitely renormalizable mappings with bounded geometry and bounded combinatorics, \cite{AL2} establishes existence of conformal measures; equality of the quantities mentioned in \eqref{it:exps}, when the Julia has measure zero; and that for such mappings if the Hausdorff dimension is not equal to the hyperbolic dimension, then the Julia set has positive area. The existence of mappings with positive area (and hence with Hausdorff dimension two) Julia set, but with hyperbolic dimension less than two was proved in \cite{AL3}.

A final class of mappings for which many such results are known are non-renormalizable quadratic polynomials. It was proved in \cite{Lyubich - measure, S} that their Julia sets have measure zero, in \cite{Ka} that they are removable, and in \cite{Pr} that their Hausdorff and hyperbolic dimensions coincide. 

Further progress on these questions for rational maps (or polynomials) 
will most likely involve complex box mappings, and moreover
any such questions could also be asked for complex box mappings. 
Indeed, this point of view is taken in, for example, \cite{PZ}.

\subsubsection{When can box mappings be induced?} 

Once one has a box mapping for a given analytic map, one can use the tools
discussed in this paper.  Therefore it is very interesting to find more 
classes of maps for which box mappings exist: 

\begin{question}
Is it true that for every rational (or meromorphic) map with a non-empty Fatou set one has an associated 
 non-trivial  box mapping?
\end{question} 

We say that a box mapping $F$ is \emph{associated} to a rational (meromorphic) map $f$ if every branch of $F$ is a certain restriction of an iterate of $f$, and the critical set of $F$ is a subset of the critical set of the starting map $f$. Furthermore, here we say that a box mapping is {\em non-trivial} if the critical set of the box mapping is non-empty
and it satisfies  the no permutation condition (see Definition~\ref{Def:NaturalBoxMapping}). 

In fact, the authors are not aware of any general procedure which associates a non-trivial box mapping
to a transcendental function. Examples of box mappings in the complement of the postsingular set for transcendental maps were constructed in \cite{Dob}.

\begin{question}
Are there examples of rational maps whose Julia set is the entire sphere 
and for which one cannot find an associated   non-trivial  box mapping?
\end{question} 

For example, for {\em topological Collet--Eckmann} rational maps, \cite{PRL} uses the strong expansion properties 
to construct box mappings. The previous two questions ask whether one can also do this when there is no such expansion.  Similarly:

\begin{question} Can one associate a box mapping to a rational map with Sierpi\'nski carpet
Julia set beyond the examples discussed previously (where symmetries are used)?
\end{question} 

Another class of maps for which the existence of  box mappings is not clear is when 
one has neutral periodic points: 

\begin{question}
Let $f$ be a complex polynomial with a Siegel disk $S$. Under which conditions is it  possible to construct a puzzle partition of a neighborhood of $\partial S$ and use this partition to  associate a complex box mapping to $f$?
\end{question}

There is a recent result in this direction, namely, in \cite{Jonguk} a puzzle partition was constructed for polynomial Siegel disks of bounded type rotation number. Using this partition and the Kahn--Lyubich Covering Lemma (see Lemma~\ref{Lem:CoveringLemma}), local connectivity of the boundary of such Siegel disks was established. In that paper, the author exploits the Douady--Ghys surgery and the Blaschke model for Siegel disks with bounded type rotation number. This model allows one to construct ``bubble rays'' growing out of the boundary of the disk. These rays, properly truncated, then define the puzzle partition. However, the resulting puzzle pieces have more complicated mapping properties than traditional Yoccoz puzzles (for example, they develop slits under forward iteration). Hence it is not clear at this point 
whether the tools presented in this paper can be applied (see also Section~\ref{SSec:NotJordan}). 

\subsection{Further extensions} 
The notion of complex box mapping has been extended in two directions. The first of these considers {\em multivalued} generalized polynomial-like maps $F\colon U\to V$. This means that we consider open sets $U_i$   and a holomorphic map $F_i\colon U_i\to V$ on each of these sets. If these sets  $U_i$ are not assumed to be disjoint, the map  $F(z):=F_i(z)$ when $z\in U_i$ becomes multivalued. Such maps are considered in  \cite{LvS-lc, LvSBox, Sh} as a first step to obtain a generalized polynomial-like map {\em with moduli bounds} (because the Yoccoz puzzle construction may not apply). As is shown in those papers one can often work with such multivalued generalized polynomial-like maps almost as well as with their single valued analogues. 

The other extension of the notion of complex box mapping is to assume that $F$ is \emph{asymptotically holomorphic} (along, for example,  the real line) rather than holomorphic. Here we say that $F$ is asymptotically holomorphic of order $\beta>0$ along some set $K$ if
$\frac{\partial}{\partial\bar{z}}F(z)= O(\dist(z,K)^{\beta-1})$. 
This point of view is considered in \cite{CvST} and \cite{CdFvS}. For example, in the latter paper
$C^{3+\alpha}$-interval maps with $\alpha > 0$ are considered. Such maps have an asymptotically holomorphic extension to the complex plane of order $3+\alpha$. 
 The analogue of the Fatou--Julia--Sullivan theorem and a topological straightening theorem is shown in this setting. In particular, these maps do not have wandering domains and their Julia sets are locally connected.

\section{Examples of possible pathologies of general box mappings}
\label{Sec:Pathologies}

The goal of this section is to point out some ``pathological issues'' that can occur if we consider a general box mapping, without knowing that it comes from a (more) globally defined holomorphic map. We start with the following result:

\begin{theorem}[Possible pathologies of general box mappings] 
\label{Thm:Path}
There are complex box mappings  $F_i \colon \U_i \to \V_i$, $i \in \{1,2,3\}$ with the following properties:
\begin{enumerate}
\item  
\label{It:PathAs1}
We have that $K(F_1) = \V_1$ and $J_\U(F_1) = \emptyset$.
\item  
\label{It:PathAs2}
The filled Julia set $K(F_2)$ has full Lebesgue measure in $\U_2$, empty interior,
and there exists a positive (indeed full) measure set of points in $K(F_2)$ that does not accumulate on any critical point.
Moreover, both $J_K(F_2)$ and $K(F_2)$ carry invariant line fields.
\item 
\label{wandering}
$\V_3$ is a disk and each connected component $U$ of $\U_3$ is compactly contained in $\V_3$ and 
contains a wandering disk for $F_3$.
\end{enumerate}
These examples are constructed so that $\mathcal U_i, \mathcal V_i$ and $F_i$ are
symmetric with respect to the real line. 
\end{theorem}

\begin{remark} 
Assertion (\ref{wandering}) shows that for general box mappings the diameter of an infinite sequence of distinct puzzle pieces $P_k$ does not need to shrink to zero as $k\to \infty$, even if their depths tend to infinity. 
Assertion~\eqref{It:PathAs2} shows that, even though a complex box mapping may be ``expanding", its Julia set can have positive measure.
\end{remark}

\begin{proof} To prove \eqref{It:PathAs1}, take $\U_1=\V_1 = \disk$ and $F_1$ is the identity map. Then $K(F_1) = \V_1$ and $J_\U(F_1)=\emptyset$ (in particular $K(F_1)$ is not closed).  

The example of \eqref{It:PathAs2} is based on the {\em Sierpi\'nski carpet} construction. Consider the square $\V_2 := (-1,1)\times (-1,1)$ cut into $9$ congruent sub-squares in a regular 3-by-3 grid, and let $U_1$ be the central open sub-square. The same procedure is then applied recursively to the remaining 8 sub-squares; this defines $U_2$ as the union of the $8$ central open sub-sub-squares. Repeating this ad infinitum, we define an open set $\U_2$ to be the union of all $U_i$. Note that $\U_2$ has full Lebesgue measure in $\V_2$ because the Lebesgue measure of $\V_2 \sm \left(\cup_{i \le n} U_i \right)$ is equal to $4(8/9)^n$. 

Define $F_2$ on each component $U$ of $\U_2$ as the affine conformal surjection from $U$ onto $\V_2$. Since $K(F_2) = \bigcap_{n \ge 1} K_n$, where $K_n := F_2^{-n}(\V_2)$, the set $K(F_2)$ also has full Lebesgue measure in $\U_2$, and $J_K(F_2) = K(F_2)$ as $K(F_2)$ has no interior points. Clearly, the horizontal line field in both $K(F_2)$ and $J_K(F_2)$ is invariant under $F_2$. 

To prove \eqref{wandering}, take a monotone sequence of numbers $a_i \in (0,1)$ such that $a_i \nearrow 1$ and
\begin{equation}
\label{Eq:Product}
\prod_{i=1}^\infty a_i = 1/2.
\end{equation}

Construct \emph{real} M\"obius maps $g_i \colon \disk \to \disk$ inductively as follows. Let $g_1$ be the identity map. Take $g_2$ such that it maps $\disk$ onto $\disk$ and $\disk_{a_2}$ to some disk to the right of $g_1(\disk_{a_1})=\disk_{a_1}$. Then assuming that $g_1,\ldots, g_{k-1}$ are defined for some $k \ge 2$, define $g_k$ to be so that it maps $\disk$ onto $\disk$ and so that 
$g_k(\disk_{a_k})$ is strictly to the right of $g_{k-1}(\disk_{a_{k-1}})$. It follows that $g_k(\disk_{a_k})$ is disjoint from $g_i(\disk_{a_i})$ for all $
i<k$. 

Next define a box mapping $F \colon \U \to \V$ (which will play the role of $F_3 \colon \U_3 \to \V_3$) by taking
\[
\V = \disk_1, \quad \U = \bigcup_{k \ge 1} g_k(\disk_{a_k}), \quad F(x) = g_{k+1}(g_k^{-1}(x)/a_k) \text{ if } x \in g_k(\disk_{a_k})
\]
(see Figure~\ref{Fig:WanderingDisk}). 

\begin{figure}[htbp]
\begin{center}
\includegraphics[scale=2, trim=32 20 30 20, clip]{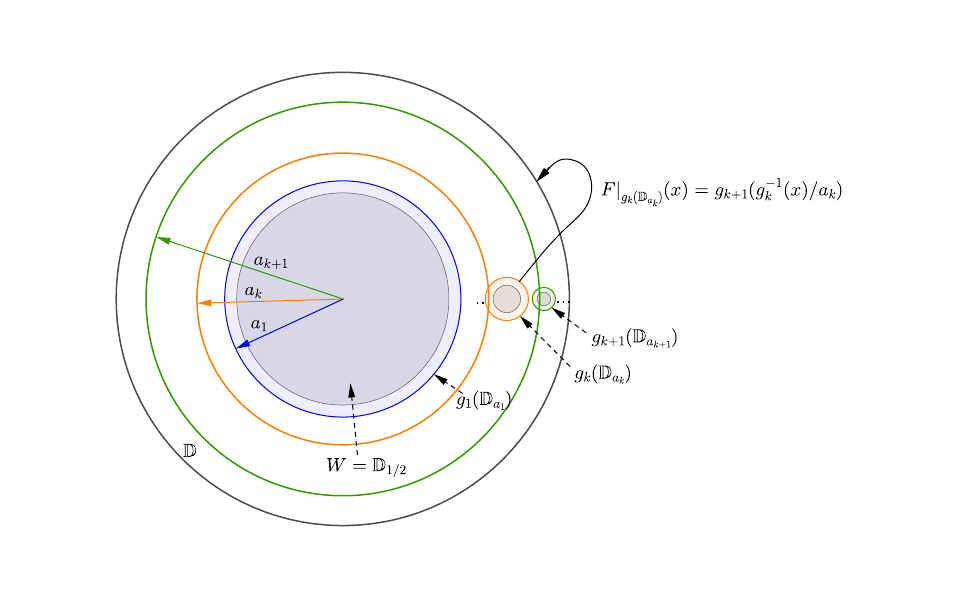}
\caption{The wandering disk construction in Theorem~\ref{Thm:Path}~\eqref{wandering}: the colored disks are components of the domain of the box mapping $F$, while the gray shaded disks are part of the trajectory of the wandering disk $W$.}
\label{Fig:WanderingDisk}
\end{center}
\end{figure}

Let us show that  $W:=\disk_{1/2}$ is a wandering disk. Observe that~\eqref{Eq:Product} implies $a_1 \cdot \ldots \cdot a_n > 1/2$ for every $n \ge 1$, and thus $a_1 \cdot \ldots \cdot a_n > |x|$ for every $x \in W$. Therefore, if $x \in W$, then $x \in g_1(W)$ and so $F(x)=g_2(g_1^{-1}(x)/a_1)=g_2(x/a_1)\in g_2(\disk_{a_2})$. Similarly,
\[
F^2(x)=g_3(g_2^{-1}(g_2(x/a_1))/a_2)=g_3(x/(a_1a_2))\in g_3(\disk_{a_3}).
\]
Continuing in this way, for each $x \in W$ and each $n\ge 0$ we have $F^n(x)=g_{n+1}(x/(a_1 \cdot \ldots \cdot a_n))$. It follows that $F^n(W)\subset g_{n+1}(\disk_{a_{n+1}})$ and therefore $W$ is a wandering disk. 
\end{proof}

\subsection{A remark on the definitions of $K(F), J_\U(F)$ and $J_K(F)$}
\label{sec:Julia}

There is no canonical definition of the Julia set of a complex box mapping, so we have given two possible contenders: $J_\U(F) = \partial K(F)\cap \U$ and $J_K(F)=\partial K(F)\cap K(F)$. In routine examples, neither $J_K(F)$ nor $K(F)$ is closed, but $J_\U(F)$ is relatively closed in $\U$. Moreover, $J_\U(F)$ 
can strictly contain $K(F)$. 

While the definitions of $K(F), J_\U(F)$ and $J_K(F)$ are similar to the definitions of the filled Julia set and Julia set of a polynomial-like mapping, when a complex box mapping has infinitely many components in its domain, the properties of its $K(F), J_\U(F)$ and $J_K(F)$ can be quite different. For example, let $F\colon\U\to\V$ be a complex box mapping associated to a unimodal, real-analytic mapping $f\colon [0,1]\to[0,1]$, with critical point $c$ and the property that its critical orbit is dense in $[f(c), f^2(c)]$ (see Figure~\ref{Fig:Dense} and the discussion about real-analytic maps in Section~\ref{SSec:raExamples} on how to construct such a box mapping). Then $\V$ will be a small topological disk containing the critical point, and $\U$, the domain of the return mapping to $\V$ will be a union of countably many topological disks contained in $\V$ with the property that $\U\cap\mathbb R$ is dense in $\V\cap\R$. In this case, one can show that
$$
K(F) \cap \R = (\V\cap\R)\setminus \bigcup_{n \ge 0} F^{-n}(E),$$ where $E$ is the hyperbolic set of points in the interval whose forward  orbits under $f$
avoid $\V$. Thus $K(F) \cap \R$ is a dense set of points in the interval $\V\cap\R$. Thus $J_\U(F)$ is the union of open intervals $\U\cap\mathbb{R}$, and it is neither forward invariant nor contained in the filled Julia set. Nevertheless it is desirable to consider $J_\U(F)$, since it agrees with the set of points in $\U$ at which the iterates of $F$ do not form a normal family.

\subsection{An example of a box mapping for which a full measure set of points converges to the boundary} 
\label{SSec:FullMeasureToBoundary}

In this section, we complement example \eqref{It:PathAs2} in Theorem~\ref{Thm:Path} by showing that not only we can have the non-escaping set of a general box mapping $F \colon \U \to \V$ to be of full measure in $\U$, but also almost all points in $K(F)$ are ``lost in the boundary'' as their orbits converges to the boundary under iteration of $F$; an example with such a pathological behavior is constructed in Proposition~\ref{Prop:TEx} below. 

Note that the box mapping $F_2 \colon \U_2 \to \V_2$ constructed in Theorem~\ref{Thm:Path}~(\ref{It:PathAs2})
had no critical points. It is not hard to modify this example so that the modified map has a non-escaping critical point. Indeed, let $\hat \V_2 := \V \sqcup U$ for some topological disk $U$, $\hat \U_2 := \U_2 \sqcup U$, and define a map $\hat F_2 \colon \hat \U_2 \to \hat \V_2$ by setting $\hat F_2|_{\U_2} = F_2$, $\hat F_2 (U) = \V_2$ and so that $\hat F_2|_U$ is a branched covering of degree at least two so that the image of a critical point lands in $K(F_2)$. This critical point for $\hat F_2$ will then be non-escaping and non-recurrent. Contrary to this straightforward modification, the box mapping constructed in the proposition below has a recurrent critical point and the construction is more intricate.

\begin{proposition}[Full measure converge to a point in the boundary]
\label{Prop:TEx}
There exists a complex box mapping $F\colon\U\to\V$ with $\Crit(F) \cap K(F) \neq \emptyset$ and with the property that the set of points $z\in K(F)$ whose orbits converge to a boundary point 
of $\V$ has full measure in $\U$.
\end{proposition}
 
 \begin{proof}
Let $\V$ the square with corners at $(3/2, 3/2), (-3/2, 3/2), (-3/2, -3/2)$ and $(3/2, -3/2)$. We will construct $\U$ so that it tiles 
$\V$.
Let $S_0$ be the open square with side-length one centered at the origin. It has corners at $(1/2,1/2), (-1/2,1/2), (-1/2,-1/2),$ and $(1/2,-1/2).$ Let $S_1$ be the union of twelve open squares, $Q_{i,j}$, each with side-length $1/2$, surrounding $S_0$,  so that $S_1$ together with $S_0$ tiles the square with corners at $(1,1), (-1,1), (-1, 1),$ and $(1,-1)$. We call $S_1$ the first {\em shell}. Inductively, we construct the $i$-th shell as the union of open squares with side length $1/2^i$, surrounding $S_{i-1}$, so that
$S_0\cup S_1\cup\dots\cup S_{i-1}\cup S_i$ tiles the square centered at the origin with side length $2(3/2 - 1/2^i)$. Inside of each square $Q_{i,j}$ in shell $i$, we repeat the Sierpi\'nksi carpet construction of 
Theorem~\ref{Thm:Path}~\eqref{It:PathAs2}. 
These open sets, which consist of small open squares,  together with the central component $S_0$, will be the domain of the complex box mapping we are constructing.

Let us fix a uniformization $\psi:\disk\to \hat Q$, where $\hat Q$ is the square with side length $3$, which we may identify with $\V$ by translation. Fix any $i\in\mathbb N$, and let $R$ be a square given by the Sierpi\'nski carpet construction in the $i$-th shell. 

Let $A_i$ be the linear mapping that rescales $R$ so that $A_i(R)$ has side length 3. 
We will define $F|_{R}\colon R\to \V$ by
$$F|_R =  \psi \circ M_i \circ \psi^{-1}\circ A_i,$$
where $M_i$ is a M\"obius transformation  that we pick inductively. 
To make the construction explicit, we take $M_i(-1)=-1$, $M_i(1)=1$ and determine
$M_i$ by choosing a point $z_i\in (-1,0)$ so that $M_i(0)=z_i$. 
Note that for any disks $D,D'$ centered at $-1$ respectively $1$
one can choose $z_i$ so that $M_i(\disk\setminus D')\subset \disk\cap D$. 
For later use, let $K$ be so that for any univalent map $\phi\colon U\to \V$
and for any square $Q_{i,j}$ from the initial partition the following inequality holds: 
$$\dfrac{|\Df\phi(z)|}{|\Df\phi(z')|} \le K$$
for all $z,z'\in U$ with $\phi(z),\phi(z')\in R$.

For any $j\in\mathbb N\cup\{0\}$, we let $T_j = \overline{\cup _{k=0}^j S_k}$.
To construct $M_1$,
we choose $z_1$ so close to $\partial \disk$ that the set of points in $R$ which are mapped to $T_1$ by $F|_R = \psi\circ M_1\circ \psi^{-1}\circ A_1$ has area at most $\mathrm{meas}(R)/2$. Up to different choices of rescaling, we define $F$ in the same way on each component $R$ contained in $S_1$.
Assuming that $M_{i-1}, i \ge 2,$ has been chosen, 
let $R$ be a square in $S_i$ and
pick $z_{i}\in (-1,0)$ close enough to $-1$ so that  
\begin{equation} 
\frac{\area(\{z\in R : F(z)\in L_i
\})}
{\area(R)} >1 -  \frac{1}{K^2 \cdot 2^{i}}, \label{eq:ind-li}
\end{equation} 
where 
$$L_i:= (-3/2,-3/2+2^{-i})\times (-2^{-i},2^{-i}).$$
Again, up to rescaling, we define $F$ identically on each component of the domain in $S_i$.
Continuing in this way, we extend $F$ to each shell.

Let $z_0\in S_{i_0}$, and for $k\in\mathbb N$ define $i_k$ so that  $F^k(z_0)\in S_{i_k}$.
We say that the orbit of $z_0$  {\em escapes monotonically to} $\partial \V$ if $i_k$ is a strictly increasing sequence. Let $W$ denote the set of points in $\cup_{i=1}^\infty S_i$ whose orbits escape monotonically to $(-3/2,0)\in \partial V$.

\begin{claim}
There exists $C_0>0$ so that for any $i\ge 1$ and any component of the domain $R$ in $S_i$,
${\area(W\cap R)}/{\area(R)}$ is bounded from below by $C_0$. 
\end{claim}

\medskip
\noindent
\textit{Proof of the claim.}
Note that the square $L_i$ is disjoint from $T_i$ and that $L_i$ is a union of squares from the initial
partition used to define the shells. 
Hence $W\cap R$ contains the set of points so that 
$\{z\in R; F(z)\in L_i, F^2(z)\in L_{i+1},\dots,\}$. To obtain a lower bound for the Lebesgue measure of this set, let $R(F^k(z))$
be the rectangle containing $F^k(z)$ and 
notice that by (\ref{eq:ind-li}) and the Koebe Distortion Theorem (see Appendix~\ref{A1}), 
$$\frac{\mathrm{meas}(\{z\in R; F(z)\in L_i, F^2(z)\in L_{i+1},\dots,F^k(z)\in L_{k}, F^{k+1}(z)\in L_{k+1}\})}{ \mathrm{meas}(\{z\in R; F(z)\in L_i, F^2(z)\in L_{i+1},\dots,F^k(z)\in L_{k})\}}\ge 1-2^{-i}.
$$
The claim follows. 
\checkmark

\medskip

Let us now define $F$ on $S_0$. Let $f\colon z\mapsto z^2$,
$Y$ be the disk of radius 16 centered at the origin and $X = f^{-1}(Y)$. Choose real symmetric conformal mappings $H_0:S_0\to X$ and $H_1:\V\to Y$ which sends the origin to itself. Let $m_{a} = \frac{z-a}{1-\bar a z}$ be the family of real symmetric M\"obius transformations of $Y$.
Consider the family of mappings 
$$F_a = H_1^{-1} \circ m_a \circ f  \circ H_0: S_0\to\V.$$ For $a=0,$ $F_a$ is a real-symmetric, polynomial-like  mapping with a super-attracting critical point at 0 of degree 2, which is a minimum for the mapping restricted to its real trace. As $a$ varies along the positive real axis from 0, the critical value of the mapping $F_a$ varies along the negative real axis from 0 until it escapes $\V\sm S_0$. Thus the family of mappings $F_a|_{S_0}$ is a full real family of mappings \cite{RealBook}, and hence it contains a mapping conjugate to the real Fibonacci mapping. Let $F\colon S_0\to\V$ denote this mapping. Now we have defined $F\colon\U\to\V$, where $\U$ tiles $\V$.

Let us now show that under $F$ a full measure set of points in $\U$ converge to $(-3/2,0)\in \partial \V$.
First, recall that the filled Julia set $K(f)$ of the quadratic Fibonacci mapping has measure zero \cite{Lyubich - measure}. Since $F:S_0\to \V$ is a polynomial-like mapping that is quasiconformally conjugate to the quadratic Fibonacci mapping, and quasiconformal mappings are absolutely continuous, we have that the filled Julia set of $F\colon S_0\to \V$ has measure zero. We will denote this set by $K_0$. Since the set of points whose orbits eventually enter $K_0$ is contained in the union of the countably many preimages of $S_0$, we have that the set of points that eventually enter $K_0$ has measure zero too. From the construction of $F$, we have that every puzzle piece contains points that map to $S_0$, so we have that almost every point in $K(F)$ either accumulates on the critical point of $F$ or converges to $\partial \V$. By the comment above, we may assume that this set of points is disjoint from the preimages of $K_0$.

Let us show that a.e.\ point in  $X_1=\{z; F^i(z)\notin S_0, i\ge 0\}$ converges to $(-3/2,0)$. 
Suppose not. Then there exists a set $X_1'\subset X_1$ for which this is not the case. 
Let $z_0$  be a Lebesgue density point of $X_1'$.  Let $Q_k$ be the puzzle piece containing 
$z_0$ of level $k$. Then $F^{k-1}(Q_k)=R$ for some rectangle $R$ with $R\cap S_0=\emptyset$. 
By the above claim  and the Koebe Distortion 
Theorem a definite proportion of $Q_k$ is mapped into the set $W$ of points which converge
monotonically to $(-3/2,0)$, thus contradicting that $z_0$ is a Lebesgue density point of $X_1'$. 

Let $X_0$ be the set of points which enter $S_0$ and 
which are not eventually mapped into the zero Lebesgue measure set $K_0$. 
Let $X_0'$ be the set of points in $X_0$ which do not converge to $(-3/2,0)$
and let $z_0$ be  a Lebesgue density point of $X_0'$. Note that by the previous paragraph a.e. point 
in $X_0'$ enters $S_0$ infinitely many times. 
Let $r_0 \ge 0$ be minimal so that $F^{r_0}(z_0)$ is contained in $S_0$. 
Let $k_0>r_0$ be minimal so that $F^{k_0}(z_0)\in S_i$ for some $i>0$, and let $R_{k_0}$ denote the square of $S_i$ that contains $F^{k_0}(z_0)$. Inductively, define $r_\ell>r_{\ell-1}$ minimal so that $F^{r_\ell}(z_0)\in S_0,$ and $k_\ell>r_\ell$, minimal so that $F^{k_\ell}(z_0)\in S_i, i>0,$ and let $R_{k_\ell}$ be the square so that $F^{k_\ell}(z_0)\in R_{k_\ell}$. By the choice of 
$K$ at the start of the proof, the distortion of the mapping $F^{k_\ell}|_{Q_\ell}\colon Q_\ell\to R_{k_\ell}$ where $Q_\ell  = \Comp_{z_0} F^{-k_\ell}(R_{k_\ell})$ is bounded by $K$.
We have already proved that in each rectangle $R\subset S_i$, $i\ge 1$ a definite proportion of points converge monotonically to the boundary point $(-3/2,0)$; hence at arbitrarily small scales around $z_0$, a definite mass of points escapes is mapped into the set $W$ (which converge to $(-3/2,0)\in \partial \V$), which implies that $z_0$ cannot be a Lebesgue density point of $X_0'$.  
\end{proof}

\section{Dynamically natural box mappings} 
\label{ASec:Subsets}

In this section we introduce the concept of a \emph{dynamically natural} complex box mapping. These are the maps for which various pathologies from Section~\ref{Sec:Pathologies} disappear, and which arise {naturally} in the study of rational maps on $\Cc$.

In order to define the concept, let us start by introducing two dynamically defined subsets of the non-escaping set. 

\subsection{Orbits that avoid critical neighborhoods}

The first subset consists of points whose orbits avoid a neighborhood of $\Crit(F)$.
Let $A \subset K(F)$ be a finite set and $W$ be a union of finitely many puzzle pieces. We say that $W$ is a \textit{puzzle neighborhood} of $A$ if $A \subset W$ and each component of $W$ intersects the set $A$.

If $\Crit(F) \neq \emptyset$, define 
\[
\Koc(F) := \left\{x \in K(F) \colon \exists W \text{ puzzle neighborhood of }\Crit(F) : \orb(x) \cap W = \emptyset\right\};
\]
otherwise, i.e.\ when $F$ has no critical points, we set $\Koc(F) \equiv K(F)$.

It is easy to see that the set $\Koc(F)$ is forward invariant with respect to $F$.

\subsection{Orbits that are well inside}
The second subset consists of any point in $K(F)$ whose orbit from time to time visits components of $\U$ that are well-inside of the corresponding components of $\V$. More precisely, let 
\[
m_F(x) := \modulus \left(\Comp_x \V \sm \overline{\Comp_x \U}\right),
\] 
and for a given $\delta > 0$, set 
\[
K_\delta (F) := \left\{y \in K(F) : \limsup_{k \ge 0} m_F(F^k(y)) > \delta \right\}.
\]
Define 
\[ 
\Kwi(F) := \bigcup_{\delta>0} K_\delta(F).
\] 
A component $U$ of $\U$ is said to be \emph{$\delta$-well-inside $\V$} if $m_F(x) > \delta$ for some (and hence for all) $x \in U$. Thus the set $\Kwi(F)$  consists of the points $x\in K(F)$  whose orbit visits infinitely often components of $\U$ that are $\delta$-well inside for some $\delta = \delta(x)>0$. By definition, the set $\Kwi(F)$ is forward invariant with respect to $F$.

\subsection{Dynamically natural box mappings}

\begin{definition}[No permutation condition and dynamical naturality]
\label{Def:NaturalBoxMapping}
A complex box mapping $F \colon \U \to \V$ satisfies the 
{\em no permutation condition} if  
\begin{enumerate}
\item
\label{It:Nat1}
for each component $U$ of $\U$ there exists $n \ge 0$ so that $F^n(U) \sm \U \neq \emptyset$.
\end{enumerate}
The mapping $F  \colon \U \to \V$ is called 
\emph{dynamically natural} if it \emph{additionally} satisfies the following assumptions:
\begin{enumerate}
\setcounter{enumi}{1}
\item
\label{It:Nat2}
the Lebesgue measure of the set $\Koc(F)$ is zero;
\item
\label{It:Nat3}
$K(F) = \Kwi(F)$.
\end{enumerate}
\end{definition}

If $F$ does not satisfy the no permutation condition, then there exists a component
$U$ of $\U$  and an integer $k>0$ so that $F^k(U)=U$. Notice that 
this implies that $U,\dots,F^{k-1}(U)$ are all components of $\V$ and that $F$ cyclically permutes these components.

Let us return to the pathologies described in Theorem~\ref{Thm:Path}. 
Each of the box mappings $F_i$, $i \in \{1, 2, 3\}$, is not dynamically natural in the sense of Definition~\ref{Def:NaturalBoxMapping}: for the map $F_i$ the respective condition $(i)$ in that definition is violated. We see that 
\begin{itemize}
\item
the box mapping $F_1$ has no escaping points in $\U_1$; 
\item
the non-escaping set of $F_2$ is equal to $\Koc(F_2)$, and hence $F_2$ provides an example of a box mapping with $\Koc(F_2)$ of non-zero Lebesgue measure;
\item
the wandering disk constructed for $F_3$ does not belong to $\Kwi(F_3)$, 
and hence $K(F_3) \sm \Kwi(F_3)$ is non-empty.  
\end{itemize}

Moreover, the example constructed in Proposition~\ref{Prop:TEx} is also not a dynamically natural box mapping as it violates condition \eqref{It:Nat2} of the definition of naturality.

The following lemma implies that no dynamically natural box mapping has the pathology described in Theorem~\ref{Thm:Path}~\eqref{It:PathAs1}. 

\begin{lemma}[Absence of components with no escaping points]  
\label{Lem:NE}
If $F \colon \U \to \V$ is a complex box mapping that satisfies the no permutation condition (Definition~\ref{Def:NaturalBoxMapping} \eqref{It:Nat1}), then 
\begin{itemize}
\item
$J_K(F)\ne \emptyset$;
\item 
each component of $K(F)$ is compact;
\item 
if $\U$ has finitely many components, then $K(F)$ and $J_\U(F) = J_K(F)$ are compact.
\end{itemize}
\end{lemma}

\begin{proof} 
Take a nested sequence $P_1\supset P_2\supset \dots$ of puzzle pieces. Then for each $i\ge 0$ either $P_{i+1}=P_i$, or $P_{i+1}$ is compactly contained in $P_i$. If $F$ satisfies the no permutation condition, then necessarily $\bigcap P_i$ is compactly contained in $\V$ and so $J_K(F) \ne \emptyset$. The second and third assertions also follow.
\end{proof}

\subsection{Motivating the notion of {\lq}dynamically natural{\rq} from Definition~\ref{Def:NaturalBoxMapping}} 
First of all, condition \eqref{It:Nat1} prohibits  $F$ to simply permute components of $\U$. Equivalently, it guarantees that each component of $\U$ has escaping points under iteration of $F$. This is clearly something one should expect from a mapping induced by rational maps: for example, under this simple condition, as we saw in Lemma \ref{Lem:NE}, each component of $K(F)$ is a compact set. 

Furthermore, under this condition we can further motivate assumption \eqref{It:Nat2}, see the remark after Corollary~\ref{Cor:Typical}: for such complex box mappings a.e. point either converges to the boundary of $\V$
or accumulates to the set of critical fibers. For any known complex box mapping which is  induced by a rational map,  the boundary of $\V$ does not attract a set of  positive Lebesgue measure, and therefore automatically
  \eqref{It:Nat2} holds.   Thus it makes sense to   assume \eqref{It:Nat2}. 

Assumption \eqref{It:Nat3} is imposed because one is usually only interested in points
that visit the {\lq}bounded part of the dynamics{\rq} infinitely  often. 

In fact, as we will show in Proposition~\ref{Prop:Inducing}, under the no permutation condition it is always possible to improve any box mapping to be dynamically natural by taking some further first return maps; this way of ``fixing'' general box mappings is enough in many applications.

\begin{remark} In \cite{ALdM} a different condition was used to rule out the pathologies we discussed in Section~\ref{Sec:Pathologies}. That paper is concerned with unicritical complex box mappings $F\colon\U\to\V$, where $\V$ consists of a single domain. In \cite{ALdM} it is assumed that $\U$  is {\em thin} in $\V$ where $\U$ is called {\em thin} in $\V$ if there exist $L,\epsilon>0$, such that  for any point $z\in \U$ there is a open topological disk $D\subset \V$ of $z$ with $L$-bounded geometry at $z$ such that 
$\modulus(\V\setminus \ovl D)>\epsilon,$ and $\area(D\setminus \U)/\area(D) > \epsilon$.
See Section 5 and Appendix A of \cite{ALdM} for results concerning this class of mappings. In particular, when $\U$ is thin in $\V$ we have that $\Koc(F)$ has measure zero. On the other hand, the requirement that $\U$ is thin is $\V$ is
a stronger geometric requirement than $K(F) = \Kwi(F)$ (compare Figure~\ref{Fig:X}). 
\end{remark}

\subsection{An ergodicity property of box mappings} 

In this subsection, we study an ergodicity property in the sense of typical behavior of orbits of complex box mappings that are not necessarily dynamically natural, but satisfy condition the no permutation condition from Definition~\ref{Def:NaturalBoxMapping}. The results of this subsection will be used later in Section~\ref{SSec:Inducing}, where we will show how to induce a dynamically natural box mapping starting from an arbitrary one, and in Section~\ref{Sec:ILF},  we will strengthen the results
in this subsection and study invariant line fields of box mappings. 

Let us say that $\mathcal F$ is a \emph{critical fiber} of $F$ if it is 
the intersection of  all the puzzle pieces containing a critical point. 
We say that this set is a \emph{recurrent critical fiber} if there exist iterates 
$n_i\to \infty$ so that some (and therefore all) limit points of  $F^{n_i}(\mathcal F)$ are contained in $\mathcal F$ (we refer the reader to Section~\ref{Sec:CombRen} for a detailed discussion on fibers and types of recurrence).

\begin{lemma}[Ergodic property of general box mappings] 
\label{Lem:AlmostAll}
Let $F \colon \U \to \V$ be a complex box mapping that satisfies the no permutation condition. Let $\mathcal C$ be the union of recurrent critical fibers of $F$. Define 
\begin{equation*}
\begin{aligned}
X_1 &=\{z \in K(F) \colon F^i(z) \to \partial \V  \mbox{ as }i\to \infty \},\\
X_2 &=\{z \in K(F) \colon  F^{i_j}(z) \to \mathcal C \mbox{ for some sequence }i_j\to \infty\}
\end{aligned}
\end{equation*}
and $$X=K(F)\setminus (X_1\cup X_2).$$ 
Then $X$ is forward invariant. Moreover, if $X'\subseteq X$ is a forward invariant set of positive Lebesgue measure, then 
there exists a puzzle piece $\J$ of $F$ so that $\area(\J \cap X') = \area(\J)$.
\end{lemma}

\begin{proof} 
That $X$ is forward invariant is obvious. 
For any integer $n \ge 0$, let $Y_n$ be the union of all critical puzzle pieces of depth $n$ for $F$ containing a recurrent critical fiber. 
Define $E_n \subset K(F)$ to be the set of points whose forward orbits are disjoint from $Y_n$. By construction, $(E_n)$ is a growing sequence of forward invariant sets
and $X'\subset \bigcup_{n \ge 0} E_n$. If $X'$ has positive Lebesgue measure, then there exists $n_0$ so that $E_{n_0}\cap X'$ has positive Lebesgue measure. Let $z_0 \in E_{n_0}\cap X'$ be a Lebesgue density point of this set.

Starting with $\V^0 := \V$, $\U^0 := \U$ and $F_0:=F$, for each $n \ge 1$ inductively define $F_n \colon \U^n \to \V^n$ to be the first return map under $F_{n-1}$ to the union $\V^n$ of all critical components of $\U^{n-1}$. Note that $\Crit(F_n) \subset \Crit(F)$ and $K(F_n)$ contains all recurrent critical fibers of $F$ (it is straightforward to see that fibers of $F_n$ are also fibers of $F$ for each $n \ge 0$).   

We may assume that there exists $k \ge 0$ so that the $F_k$-orbit of $z_0$ visits non-critical components of $\U^k$ infinitely many times. Indeed, otherwise the $F$-orbit of $z_0$ accumulates at a recurrent critical fiber contrary to the definition of $E_{n_0}$. Fix this mapping $F_k$.

{\bf Case 1:}  the orbit $z_n := F_k^n (z_0)$, $n=0,1,\ldots$ visits non-critical components of $\U^k$ infinitely many times, but only finitely many different ones. Write $\U'$ for the union of these finitely many components and consider $K' := \{z \in E_{n_0} \colon F^n_k(z) \in \U' \,\,\, \forall n \ge 0\}$. By Lemma~\ref{Lem:NE}, the closure of $K'$ is a  subset of $E_{n_0}$, and hence of $K(F)$. Therefore, there exists a point $y \in \omega(z_0) \cap E_{n_0}$.
 Let $\J_0 \ni y$ be a puzzle piece of depth $n_0$, and let $D \Subset \J_0$ be another puzzle piece containing $y$ of depth larger than $n_0$; such a compactly contained puzzle piece exists again because of Lemma~\ref{Lem:NE}. Take the moments $n_i$ so that $F_k^{n_i}(z_0) \in \J_0$ and let $A_i := \Comp_{z_0} F_k^{-n_i}(\J_0)$. Then $F_k^{n_i} \colon A_i \to \J_0$ is a sequence of univalent maps with uniformly bounded distortion when restricted to $B_i:=\Comp_{z_0} F_k^{-n_i}(D)$. 
 Therefore $\modulus(A_i\setminus \ovl B_i)$ is uniformly bounded from below. 
We also have that $A_{i'}$ is contained in $B_i$ for $i'>i$ sufficiently large. 
It follows that $\diam B_i \to 0$ as $i \to \infty$.  
Thus, by the Lebesgue Density Theorem, $\area(D \cap E_{n_0}) = \area(D)$, and hence $\area(D \cap X') = \area(D)$ as $X'$ is forward invariant. The claim of the lemma follows with $\J = D$.

{\bf Case 2:} the orbit $z_n$ visits infinitely many different non-critical components of $F_k$. If $n_s$ is the first moment the orbit $z_n$ visits a new non-critical component $U_s$ of $\U^k$, then the pullbacks of $U_s$ under $F_k$ back to $z_0$ along the orbit hit each critical fiber at most once. Let $n_s' > n_s$ be the first subsequent time that the orbit $z_n$ enters
a critical component of $\U^k$.  If such an integer $n_s'$ exists, then we have that 
$F_k^{n_s'} \colon \Comp_{z_0} F_k^{-n'_s}(\V^k)\to \V^k$ has degree bounded independently of $s$. 
We can proceed again as in Case 1 (there are only finitely many critical components of $\U^k$). 
If $n_s'$ does not exist, then we have that $F_k^n\colon  \Comp_{z_0} F_k^{-n}(\V^k) \to \V^k$ has uniformly bounded degree for all $n\ge n_s$. Since we assumed that $z_0$ does not converge to the boundary, 
there exists an accumulation point $y\in \V^k \cap \omega(z_0)$ with the property that 
$\J_0 := \Comp_y \V^k$ is not a component of $\U^k$. 
Note that $y \in \ovl{\U^k}$, but $y$ may or may not lie in $\U^k$. Each component of $\U^k \cap \J_0$ is compactly contained in $\J_0$. This allows us to choose an open topological disk $D \Subset \J_0$ so that $y \in D$ and such that $D$ contains at least one puzzle piece $\J$. We conclude similarly to the previous case that $\area(D \cap X') = \area(D)$. Thus $\area(\J \cap X') = \area(\J)$ as $\J \subset D$. 
\end{proof} 

\begin{corollary}[Typical behavior of orbits]
\label{Cor:Typical}
Let $F \colon \U \to \V$ be a complex box mapping that satisfies the no permutation condition and such that 
each puzzle piece of $F$ either contains a point $w$ so that $F^n(w)$ is a critical point for some $n\ge 0$, or it contains an open set disjoint from $K(F)$ (a \emph{gap}).
Then for a.e. $z \in K(F)$ either 
\begin{enumerate}
\item $F^i(z)\to \partial \V$  as $i\to \infty$, or
\item the forward orbit of $z$ accumulates to the fiber of a critical point. 
\end{enumerate} 
\end{corollary} 
\begin{proof} Let $X \subset K(F)$ be as in the previous lemma. For each $n\ge 1$, let $X'_n \subset X$
be the set of points $z\in X$ so that the orbit of $z$ remains outside critical 
puzzle pieces of depth $n$.
By the previous lemma if $X'_n$ has positive Lebesgue measure, then there exists a puzzle piece $\J$ so that  $\area(\J \cap X_n') = \area(\J)$. But, under the assumption of the corollary, either $\J$ contains a gap, which is clearly impossible since $\area(\J \cap X_n') = \area(\J)$, or there exists a  puzzle piece $\J'_n\subset \J$ which is mapped under some iterate of $F$ to a critical piece
of depth $n$, which contradicts the definition of the set $X'_n$ since also $\area(\J'_n\cap X_n') = \area(\J_n')$. Hence the set $X'_n$
has zero Lebesgue measure. It follows that the set $X'=\cup X'_n$
of points in $X$ whose forward orbits stay outside {\em some} critical puzzle piece also has measure zero.
The corollary follows. 
\end{proof}

\begin{remark}
For a polynomial or rational map, preimages of critical points are dense in the Julia set
(except if the map is of the form $z\mapsto z^{n}$ or $z\mapsto z^{-n}$).
This means that if a box mapping is associated to a polynomial or rational map, 
then the assumption of Corollary~\ref{Cor:Typical} is satisfied. The conclusion of that corollary motivates the assumption $\area(\Koc(F)) = 0$ in the definition of dynamical naturality (Definition~\ref{Def:NaturalBoxMapping}). 
\end{remark}

\subsection{Improving complex box mappings by inducing.}
\label{SSec:Inducing}

We end this section by showing how to ``fix'' a general complex box mapping for it to become dynamically natural. In fact, every box mapping satisfying the no permutation condition from the definition of naturality, subject to some mild assumptions, induces a dynamically natural box mapping that captures all the interesting critical dynamics; working with such induced mappings is enough in many applications. The following lemma illustrates how it can be done in detail.

\begin{proposition}[Inducing dynamically natural box mapping] 
\label{Prop:Inducing}
Let $F \colon \U \to \V$ be a box mapping that satisfies the no permutation condition. 
Assume that each component of $\V$ contains a critical point, or 
contains an open set which is not in $\U$. Then we can induce a dynamically natural box mapping 
from $F$ in the following sense. 

Take $\hat \U$ to be a finite union of puzzle pieces of $F$ so that $\hat \U$ is a nice set compactly contained in $\V$,
and let $\Crit_{\hat \U}(F) \subset \Crit(F)$ be a subset of critical points whose orbits visit $\hat \U$ infinitely many times. 
Then there exists an induced, dynamically natural box mapping $F_1 \colon \U_1 \to \V_1$ such that $ \Crit_{\hat \U}(F)\subset \Crit(F_1)\subset \Crit(F)$. 
\end{proposition} 

\begin{remark}
A typical choice for $\hat \U$ would be puzzle pieces of $\U$ that contain critical 
points and critical values of $F$. 
\end{remark}

\begin{proof}[Proof of Proposition~\ref{Prop:Inducing}]  
Let us define two disjoint nice sets $\V_1', \V_1'' \subset \V$ as follows. For the first set, we let 
\[
\V_1' := \bigcup_{c \in \Crit(F)} \DomL_c \hat \U, 
\]
where the union is taken over all points in $\Crit(F)$ whose orbits intersect $\hat \U$. For the second set, we define $\V_1''$ to be the union  of puzzle pieces of depth $n \ge 0$ containing all critical points in $\Crit(F) \sm \V_1'$ (if this set is empty, we put $\V_1'' = \emptyset$). Moreover, we can choose $n$ sufficiently large so that $\V_1'' \cap \V_1 = \emptyset$ and so that $\V_1''$ is compactly contained in $\V$; the latter can be achieved due to Lemma~\ref{Lem:NE}.

Now let $\V_1 := \V_1' \cup \V_1''$, and let $F_1 \colon \U_1 \to \V_1$ be the first return map to $\V_1$ under $F$. It easy to see that $F_1$ is a complex box mapping in the sense of Definition~\ref{Def:BM}; it also satisfies property \eqref{It:Nat1} of Definition~\ref{Def:NaturalBoxMapping} (notice that $K(F_1) \subset K(F)$). Moreover, from the first return construction it follows that $\Crit_{\hat \U}(F)\subset \Crit(F_1)\subset \Crit(F)$. 

Let us now show that $F_1$ is dynamically natural. For this we need to check properties \eqref{It:Nat2} and \eqref{It:Nat3} of Definition~\ref{Def:NaturalBoxMapping}.

\smallskip
\noindent
\textit{Property \eqref{It:Nat2} of Definition~\ref{Def:NaturalBoxMapping}:} Let $X_O \subset \Koc(F_1)$ be the set of points in $K(F_1)$ whose forward iterates under $F_1$ avoid a puzzle neighborhood $O$ of $\Crit(F_1)$; this is a forward invariant set. Hence, if $X_O$ has positive Lebesgue measure, then 
by Lemma~\ref{Lem:AlmostAll} either a set of positive Lebesgue measure of points in $X_O$ converge to the boundary of  $\V_1$, or  there exists a puzzle piece $\J$ of $F_1$ 
so that $\area(\J \cap X_O) = \area(\J)$. The first situation cannot arise due to the choice
of $\hat \U$ and $\V_1''$ (both are compactly contained in $\V$), and the second situation would imply that some forward iterate of $\J$ will cover 
a component of $\V_1$ and therefore a component $V$ of $\V$, and thus
we would obtain $\area(X_O\cap V)=\area(V)$, contradicting the assumption we made on $F$: the puzzle piece $V$ must contain either a critical point, or an open set in the complement of $\U$.

\smallskip
\noindent
\textit{Property \eqref{It:Nat3} of Definition~\ref{Def:NaturalBoxMapping}:} If there are only finitely many components in $\U_1$, this property is obviously satisfied. Therefore, we can assume that $\U_1$ has infinitely many components. Let $U$ be a component of $\U_1$ such that $U \Subset \Comp_U \V_1$, and let $F^s \colon U \to V$ be the corresponding branch of $F_1$; here $V$ is a component of $\V_1$. By the first return construction, the degree of this branch is bounded independently of the choice of $U$. By construction of $\V_1$, there exists $t \ge 0$ and a component $W$ of $\V_1$ such that $W \Subset \V$ and $F^{s+t}(U) = W$. Put $s' := s+t$. Again, the degree of the map $F^{s'} \colon U \to W$ is bounded independently of $U$. Let $W' \Supset W$ be a puzzle piece of $F$; it exists since $W \Subset \V$. We claim that the degree of the map $F^{s'} \colon U' \to W'$, where $U' = \Comp_{U} F^{-s'}(W')$, is independent of the choice of $U$.

Indeed, this is equivalent to saying that in the sequence $F(U'), \ldots, F^{s'-1}(U')$ the number of critical puzzle pieces is independent of the choice of $U$. Let $P$ be a critical puzzle piece in this sequence, and $c \in \Crit(F) \cap P$ be a critical point. Then $P \supsetneq \Comp_c \V_1$, because otherwise the $F$-orbit of $U$ would intersects $\V_1$ before reaching $V$, which is a contradiction to the fact that $F^s \colon U \to V$ is the branch of the first return map to $\V_1$. But there are only finitely many critical puzzle pieces of $F$ of depths larger than the depths of the components of $\V_1$, and this number does not depend on $U$. This yields the claim.

Pulling back the annulus $W' \sm W$ by the map $F^{s'} \colon U' \to W'$, we conclude that there exists $\delta>0$, independent of $U$, such that $\modulus(U' \sm \ovl U) \ge \delta$. Finally, $U' \subset \V_1$ because the depths of $U$ and $U'$, viewed as puzzle piece of $F$, differ by one. Hence $(\Comp_U \V_1) \sm U \supset U' \sm U$. This and the moduli bound imply $K(F_1) = K_\delta(F_1)$.  
\end{proof}

\section{Combinatorics and renormalization of box mappings}
\label{Sec:CombRen}
The combinatorics of critical points of complex box mappings are similar to those of complex polynomials. As is often the case in holomorphic dynamics, the tools used to study the orbit of a critical point depend on the combinatorial type of the critical point. In this section, we start by recalling the definitions of a fiber, of recurrent, persistently, and reluctantly recurrent critical points, and of a combinatorial equivalence of box mappings. We end this section by defining the notion of renormalization for complex box mappings.

In this section, $F \colon \U \to \V$ is a complex box mapping, not necessarily dynamically natural, but that satisfies the no permutation condition from Definition~\ref{Def:NaturalBoxMapping}.

\subsection{Fibers}
\label{SSec:Fibers}

While working with dynamically defined partitions, for example given by puzzles, it is convenient to use the language of fibers, introduced by Schleicher in \cite{Fibers, Fibers2}. 

In our setting, for $x \in K(F)$, let $P_n(x)$ be the puzzle piece of depth $n$ containing $x$. The \emph{fiber}\footnote{The term ``fiber'' was adopted from the theory of Douady's \emph{pinched disk models} for quadratic Julia sets and the Mandelbrot set: in this context, there is a natural continuous surjection from a Julia set to its pinched disk model, and the preimage of a point under this surjection --- i.e.\ the fiber in the classical sense --- is exactly a single combinatorial fiber in the sense of \cite{Fibers, Fibers2}. For details, see \cite{Do, SApp}.} of $x$ is the set 
\[
\fib(x) := \bigcap_{n \ge 0} P_n(x).
\]

By Lemma~\ref{Lem:NE}, $\fib(x)$ is a compact connected set. From the dynamical point of view, $\fib(x)$ consists of points that are ``indistinguishable'' by the puzzle partition and hence constitute a single combinatorial class. A fiber containing a critical point is called \emph{critical}.

\subsection{Persistently and reluctantly recurrent critical points}
\label{SSec:Rec}

A critical point $c \in \Crit(F)$ is said to be \emph{(combinatorially) non-recurrent} if the orbit of $F(c)$ is disjoint from some puzzle piece around $c$. 
If this is not the case, then $c$ is called \emph{(combinatorially) recurrent}. For such points, $\orb(F(c))$ intersects every puzzle neighborhood of $c$\footnote{In \cite{KSS}, the notion of $Z$-recurrent points appears. There, $Z$ refers an initial choice of puzzle partition and recurrence is defined with respect to that initial choice. In the setting of a given complex box mapping, one works with a single puzzle partition that is predefined by the map.}. In the language of fibers, a critical point $c$ is combinatorially recurrent if and only if the orbit of $F(c)$ accumulates at $\fib(c)$, i.e.\ $\ovl{\orb(F(c))} \cap \fib(c) \neq \emptyset$. 

It will be useful to define
\begin{equation*}
\begin{aligned}
\Back(c) &:= \{c' \in \Crit(F) \colon \ovl{\orb(c')} \cap \fib(c) \neq \emptyset\}, \\
\Forw(c) &:= \{c' \in \Crit(F) \colon \ovl{\orb(c)} \cap \fib(c') \neq \emptyset\},
\end{aligned}
\end{equation*}
to be the sets of critical points ``on the backward, resp.\ forward orbit of the critical point $c$'', and 
\[
[c] := \Back(c) \cap \Forw(c)
\]
to be the set of critical points that ``accumulate at each other's fibers''. With this definition, $c \in [c]$. Note that $\Back(c) = \Back(c')$, $\Forw(c) = \Forw(c')$,  and $[c] = [c']$ for every $c' \in \fib(c)$. 

A critical puzzle piece $P$ is a \emph{child} of a critical puzzle piece $Q$ if there exists $n \ge 1$ such that $F^n(P) = Q$ and $F^{n-1} \colon F(P) \to Q$ is a univalent map.

Now we want to distinguish different types of recurrence as follows.  Let $c$ be a combinatorially recurrent critical point of $F$. We say that $c$ is \emph{persistently recurrent} if for every critical point $c'\in [c]$,  every puzzle piece $P_n(c')$ has only finitely many children containing critical points in $[c]$.  Otherwise, $c$ is called \emph{reluctantly recurrent}.

One should observe that all critical points in a critical fiber are either simultaneously non-recurrent, or all persistently recurrent, or all reluctantly recurrent, and hence we can speak about non-, reluctantly, or persistently recurrent critical fibers\footnote{This observation demonstrates the idea that a given puzzle partition cannot distinguish points from their fibers. In this vein,  following Douady (with an interpretation by Schleicher), various rigidity results can be phrased as ``fibers are points'', i.e.\ each fiber is trivial and hence equal to a point.}.

\subsection{Renormalization}
\label{subsec:renorm}
Following~\cite[Definition 1.3]{KvS}, a complex box mapping $F\colon \U \to \V$ is called \textit{(box) renormalizable} if there exists $s \ge  1$\footnote{Usually we assume that 
$s$ is minimal with this property. Such minimal $s$ is called \emph{the period of renormalization}.}  and a puzzle piece $W$ of some depth containing a critical point $c \in \Crit(F)$ such that $F^{ks}(c) \in W$ for all $k \ge 0$. If this is not the case, then $F$ is called \emph{non-renormalizable}.

\begin{lemma}[Douady--Hubbard equivalent to box renormalization]
\label{Lem:DHB}
If $F\colon\U\to\V$ is a complex box mapping that satisfies the no permutation condition, then $F$ is box renormalizable if and only if it is renormalizable in the classical Douady--Hubbard sense~\cite{PolyLike}.
\end{lemma}

Recall that one says that a box mapping $F\colon\U\to\V$ is {\em Douady--Hubbard renormalizable} if there exist $s\in\mathbb N$, $c \in \Crit(F)$ and open topological disks $c \in U\Subset V \subset \U$ so that $F^s\colon U\to V$ is a polynomial-like map with connected Julia set.

\begin{proof}[Proof of Lemma~\ref{Lem:DHB}]
First suppose that $F$ is box renormalizable, and let $c \in \Crit(F)$, $W \ni c$ and $s$ be as in the definition of box renormalizable. By Lemma~\ref{Lem:NE}, the fiber of $c$ is a compact connected set. The condition $F^{ks}(c) \in W$ for all $k \ge 0$ can be now re-interpreted as $F^{ks}(c') \in \fib(c) \subset W$ for all $k \ge 0$ and all critical points $c' \in \fib(c)$. Now let $n$ be large enough so that the only critical points of $F$ in the set $V:=\Comp_c F^{-ns}(W)$ are the one contained in the fiber of $c$. If $U := \Comp_c F^{-s}(V)$, then $U \Subset V$ because $F$ satisfies the no permutation condition (see assumption \eqref{It:Nat1} in Definition~\ref{Def:NaturalBoxMapping}), and hence $F^s \colon U \to V$ is a polynomial-like map. This mapping has connected Julia set equal to $\fib(c)$ by our choice of $n$. We conclude that $F$ is renormalizable in the Douady--Hubbard sense. 

Conversely, suppose that $F\colon\U\to \V$ is Douady--Hubbard renormalizable. Then there exist $c \in \Crit(F)$, $s \in \mathbb N$ and a pair of open topological disks $U \Subset V$ so that  $\rho:=F^s\colon U\to V$ is a polynomial-like map with connected Julia set $K(\rho)$ with $c  \in K(\rho)$. Note that $U$ and $V$ are not necessarily puzzle pieces of $F$, however $K(\rho) \subset \fib(c)$. The last inclusion implies that for every puzzle piece $W$ of $F$ with $W \ni c$ it is also true that $W \supset K(\rho)$. Since $K(\rho)$ is forward invariant under $F^s$, we have that for all $k\in\mathbb N$, $F^{ks}(c)\in W$. Picking a smaller $W$ to assure minimality of $s$ in the definition of box renormalization, we conclude the claim. 
\end{proof}

We now want to relate the existence of non-repelling periodic points for a complex box mapping to its renormalization. 

\begin{lemma}[Non-repelling cycles imply renormalization]
\label{Lem:NonRepRen}
If $F\colon\U\to\V$ is a complex box mapping that satisfies the no permutation condition, then the orbit of the fiber of any attracting or neutral periodic point contains a critical point of $F$. Moreover, if $F$ is non-renormalizable, then all the periodic points of $F$ are repelling. 
\end{lemma}

\begin{proof} Suppose that $z_0$ is a periodic point with period $p$ for $F$. Let $U_0'$ denote the component of $\U$ that contains $z_0$, and let
$U_0=\Comp_{z_0}F^{-p}(U_0')$.  Since $F$ satisfies the no permutation condition (see assumption \eqref{It:Nat1} in Definition~\ref{Def:NaturalBoxMapping}), we have that $U_0\Subset U_0'$. 
We claim that the mapping $F^p\colon U_0\to U_0'$ has a critical point. 
If not, then consider the inverse of this map. By the Schwarz Lemma this  
inverse map will be strictly contracting at its fixed point.
\end{proof}

\subsection{Itineraries of puzzle pieces and combinatorial equivalence of box mappings}
\label{SSec:CombEquiv}

By Definition~\ref{Def:BM}, each component in the sets $\U$ and $\V$ for a box mapping $F\colon \U \to \V$ is a Jordan disk. Therefore, by the Carath\'eodory theorem for every branch $F \colon U \to V$ of this mapping there exists a well-defined continuous homeomorphic extension $\hat F \colon \ovl U \to \ovl V$, and by continuity this extension is unique. Let us denote by $\hat F \colon \cl \U \to \cl \V$ the total extended map. 

\begin{definition}[Itinerary of puzzle pieces relative to curve family]
\label{DefA:CurveFamily}
Let $F \colon \U \to \V$ be a box mapping satisfying the no permutation condition, and let $X\subset \partial \V$ be a finite set with one point on each component of $\partial \V$. Let $\Gamma$ be a collection of simple curves in $(\cl \V) \sm (\U \cup \PC(F))$, one for each $y \in \hat F^{-1}(X)$, that connects $y$ to a point in $X$ (see Figure~\ref{Fig:Itinerary}). Then for every $n \ge 0$ and for each component $U$ of $F^{-n}(\U)$ there exists a simple curve connecting $X$ to $\partial U$ of the form $\gamma_0\ldots\gamma_n$ where $\hat F^k(\gamma_k) \in \Gamma$. The word $(\gamma_0, \hat F(\gamma_1), \ldots, \hat F^n(\gamma_n))$ is called the \emph{$\Gamma$-itinerary} of $U$.
\end{definition}

If a component $V$ of $\V$ is also a component
of $\U$, then the corresponding curve will be contained in the Jordan curve $(\cl V) \sm (\U \cup \PC(F))
= \partial V$, because by definition $\PC(F)\subset \V$ and so no point in the boundary of $V$ lies
in $\PC(F)$. 

Note the $\Gamma$-itinerary of $U$ is not uniquely defined even though there is a unique finite word for every $y'\in \hat F^{-n}(x)\cap \partial U$. For example, in Figure~\ref{Fig:Itinerary} the critical component of $F^{-1}(\U)$ has four $\Gamma$-itineraries: $(a,a)$, $(a,b)$, $(b,a)$, and $(b,b)$. However, different components of $F^{-n}(\U)$ have different $\Gamma$-itineraries.

\begin{figure}[htbp]
\begin{center}
\includegraphics[scale=.5]{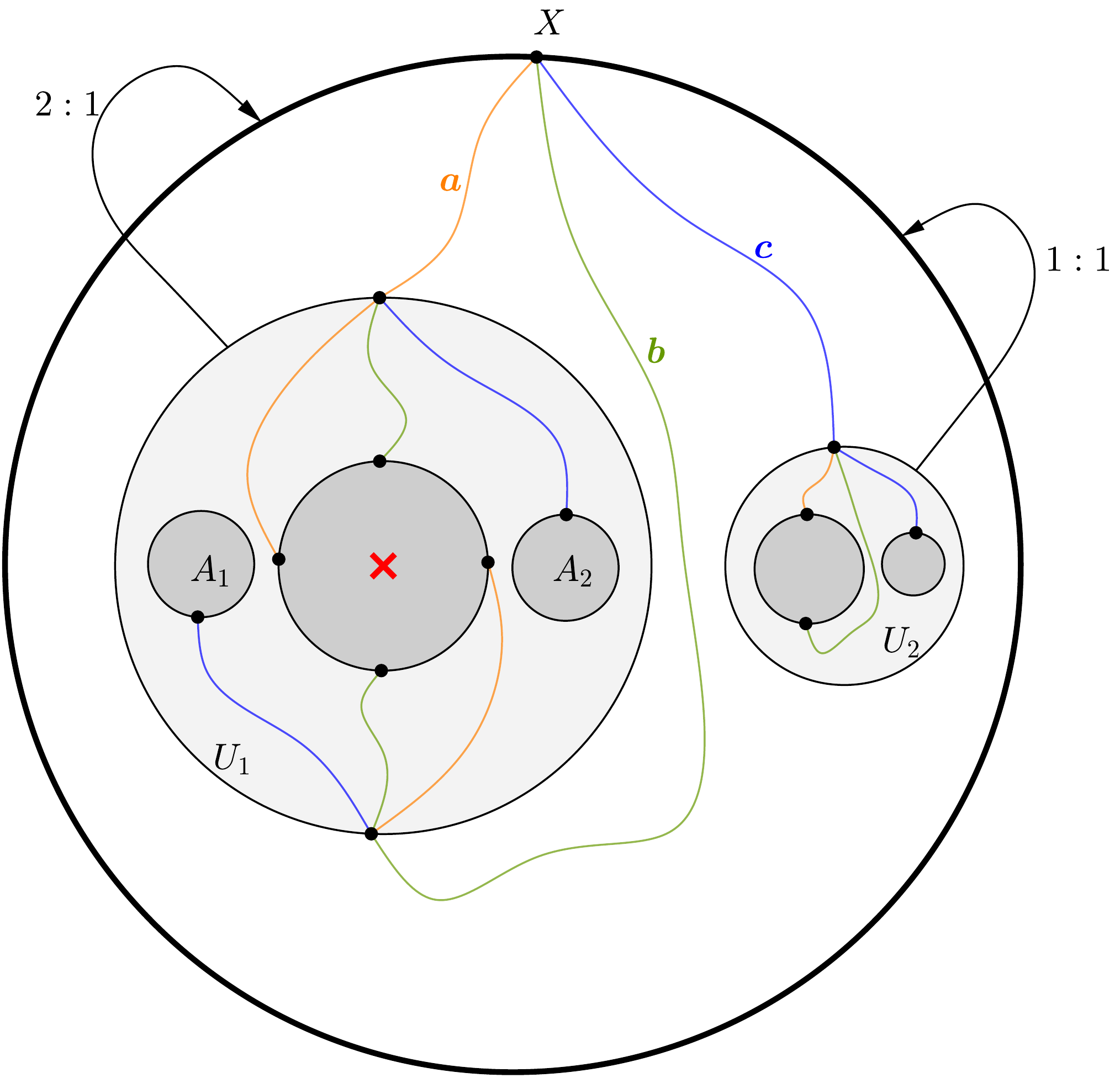}
\caption{An example of a complex box mapping $F \colon \U \to \V$ with $\U = U_1 \cup U_2$ (shaded in light grey), $\V$ having one component and a unique simple critical point (marked with the red cross) which mapped by $F$ to $U_1$. An example of a curve family $\Gamma = \{a,b,c\}$ is shown. The set $F^{-1}(\U)$ has $5$ components (shown in dark grey). Two of these, marked $A_1$ and $A_2$, are mapped over $U_2$ and hence have the same itineraries with respect to the components of $\U$. However, their $\Gamma$-itineraries are different: $A_1$ has $\Gamma$-itinerary $(b,c)$, while the $\Gamma$-itinerary of $A_2$ is $(a,c)$.}
\label{Fig:Itinerary}
\end{center}
\end{figure}

Definition~\ref{DefA:CurveFamily} might look a bit unnatural at first. A simpler approach might to define an itinerary of a component $U$ of $F^{-n}(\U)$ as a sequence of components of $\U$ through which the $F$-orbit of $U$ travels. However, such an itinerary would not distinguish components properly, an example is again shown in Figure~\ref{Fig:Itinerary}: both $A_1$ and $A_2$ have the same ``itineraries through the components of $\U$'', but clearly $A_1 \neq A_2$. This motivates using curves in order to properly distinguish components.

\begin{remark}[On the existence of $\Gamma$]
We observe that a collection $\Gamma$ in Definition~\ref{DefA:CurveFamily} always exists. Indeed, let $U$, $V$ be a pair of components of $\U$ resp.\ $\V$ with $\ovl U \subset V$, and $x \in \partial V$, $y \in \partial U$ be a pair of points we want to connect with a curve within $(\cl \V) \sm (\U \cup \PC(F))$. Clearly, the set $A:= \PC(F) \cap (\V \sm \ovl \U)$ is finite (the box mapping $F$ is not defined in $\V \sm \ovl \U$). Choose any simple curve $\gamma \colon [0,1] \to \ovl V \sm U$ so that $\gamma(0) = x$ and $\gamma(1) = y$. Such a curve exists because $\ovl V \sm U$ is a closed non-degenerate annulus bounded by two Jordan curves. We can choose $\gamma$ so to avoid $A$ as well as the closures of the components $W$ of $\U$ in $V$ with $\partial W \cap \PC(F) \not= \emptyset$. The set $A$ is finite, and there are only finitely many such $W$'s. Hence a curve $\gamma$ with such an additional property exists. This curve might intersect a component $U' \not= U$ of $\U$ within $V$. Let $t_1$ be the infimum over all $t \in [0,1]$ so that $\gamma(t) \in \partial U'$; similarly, let $t_2$ be the supremum over all such $t$'s (it follows that $0 < t_1 \le t_2 < 1$ as $\partial U' \cap \partial U = \emptyset$ and $\partial U' \cap \partial V = \emptyset$). Modify the curve $\gamma$ on $[t_1, t_2]$ by substituting it with the part of $\partial U'$ going from $\gamma(t_1)$ to $\gamma(t_2)$. Since $\partial U'$ is a Jordan curve, the new curve (we keep calling it $\gamma$) will be a Jordan curve. Moreover, $\gamma$ will avoid $U'$, will join $x$ to $y$ and will be disjoint from $A$. The desired curve in $\Gamma$ is obtained after resolving all intersections with all such $U'$.              
\end{remark}

Finally, let us give a definition of combinatorially equivalent box mappings (see~\cite[Definition 1.5]{KvS}). 

\begin{definition}[Combinatorial equivalence of box mappings]
\label{DefA:CombEquivBoxMappings}
Let $F \colon \U \to \V$ and $\tilde F \colon \tilde \U \to \tilde \V$ be two complex box mappings, both satisfying  the no permutation condition. Let $H \colon \V \to \tilde \V$ be a homeomorphism with the property that $H(\U)=\tilde \U$, $H(\PC(F)\sm\U) = \PC(\tilde F) \sm \tilde \U$ and such that it has a homeomorphic extension $\hat H \colon \cl \V \to \cl \tilde \V$ to the closures of $\V$ and $\tilde \V$.

The maps $F$ and $\tilde F$ are called \emph{combinatorially equivalent w.r.t.\ $H$} if:
\begin{itemize}
\item
$H$ is a bijection between $\Crit(F)$ and $\Crit(\tilde F)$; for $c \in \Crit(F)$, $\tilde c := H(c)$ is the \emph{corresponding} critical point;
\item
there exists a curve family $\Gamma$ as in Definition~\ref{DefA:CurveFamily} so that for every $n \ge 0$, and for each $k \ge 0$ such that both $F^k(c)$ and $\tilde F^k(\tilde c)$ are well-defined, the $\Gamma$-itineraries of $\Comp_{F^k(c)} F^{-n}(\V)$ coincide with the $\hat H(\Gamma)$-itineraries of $\Comp_{\tilde F^k(\tilde c)} \tilde F^{-n}(\tilde \V)$.
\end{itemize} 
\end{definition}

Using Definition~\ref{DefA:CombEquivBoxMappings}, we call a pair of puzzle pieces $P$, $\tilde P$ for $F$, $\tilde F$ \emph{corresponding} if all $\Gamma$-itineraries of $P$ coincide with the $\hat H(\Gamma)$-itineraries of $\tilde P$. The definition then requires that the orbits of the corresponding critical points $c$, $\tilde c$ travel through the corresponding puzzle pieces at all depths. This also gives rise to the notion of \emph{corresponding critical fibers} $\fib(c)$ and $\fib(\tilde c)$. Observe that the degrees of the corresponding critical fibers for combinatorially equivalent box mappings must be equal (here, the \emph{degree of a fiber} is the degree of $F$, respectively $\tilde F$, restricted to a sufficiently deep puzzle piece around the fiber).

\section{Statement of the QC-Rigidity Theorem and outline of the rest of the paper}
\label{Sec:QCStatement}

We are now ready to state the main result from~\cite{KvS}, namely, the rigidity theorem for complex box mappings, stated there as Theorem 1.4, but with the explicit assumption that the complex box mappings are dynamically natural. 

\begin{theorem}[Rigidity for complex box mappings]
\label{Thm:BoxMappingsMain2}
Let $F\colon \U\to \V$ be a dynamically natural complex box mapping that is non-renormalizable. Then 
\begin{enumerate}
\item 
\label{It:MainIt1}
each point $x \in K(F)$ is contained in arbitrarily small puzzle pieces.
\item 
\label{It:MainIt2}
If $F$ carries a measurable invariant line field  supported on a forward invariant set 
$X\subseteq K(F)$ of positive Lebesgue measure,
then 
\begin{enumerate}
\item\label{parta} there exists a puzzle piece $\J$ and a smooth foliation on $\J$ with the property that $\area(X\cap \J)=\area(\J)$ and the foliation is tangent to the invariant line field a.e.\ on $\J$;
\item\label{partb} the set $Y=\bigcup_{i\ge 0}F^i(\J)$ intersects the critical set of $F$ and each such critical point is strictly preperiodic.
\end{enumerate} 
\item 
\label{It:MainIt3}
Suppose $\tilde F \colon \tilde \U \to \tilde \V$ is another dynamically natural complex box mapping for which there exists a quasiconformal homeomorphism $H \colon \V \to \tilde\V$ so that
\begin{enumerate}
\item
\label{part3c}
$\tilde F$ is combinatorially equivalent to $F$ w.r.t.\ $H$, and so in particular $H(\U) = \tilde \U$,
\item
\label{part3b}
$\tilde F \circ H = H \circ F$ on $\partial \U$, i.e. $F$ is a conjugacy on $\partial U$, 
\end{enumerate}
Then $F$ and $\tilde F$ are quasiconformally conjugate, and this conjugation agrees with $H$ on $\V\setminus \U$.     
\end{enumerate}
\end{theorem}

In the following sections, we sketch the proof of Theorem~\ref{Thm:BoxMappingsMain2}. Our aim is to make it more accessible to a wider audience.

\begin{remark}
Let us make several remarks regarding the statement of Theorem~\ref{Thm:BoxMappingsMain2}:

\begin{enumerate}
\item
In \cite{KvS}, it was assumed that the box mapping  in the statement of Theorem~\ref{Thm:BoxMappingsMain2} is non-renormalizable \emph{and} that all its periodic points are repelling. However, due to Lemma~\ref{Lem:NonRepRen}, it is enough to assume that the box mappings in question are non-renormalizable as the property that all periodic points are repelling then follows from that lemma. 

\item
In the language of fibers, shrinking of puzzle pieces stated in Theorem~\ref{Thm:BoxMappingsMain2}~\eqref{It:MainIt1} is equivalent to \emph{triviality of fibers} for points in the non-escaping set, i.e.\ $\fib(x) = \{x\}$ for every $x \in K(F)$. Note that if all the fibers are trivial, then $c$ is combinatorially recurrent if and only if $c \in \omega(c)$.

\item There exist dynamically natural box mappings which have invariant line fields, see Proposition~\ref{prop:lattes}.
These box mappings are rather special (and are the analogues of Latt\`es rational maps in this setting),
see Proposition~\ref{prop:lattes-description}. However, if a dynamically natural box mapping $F$ is coming from a Yoccoz-type construction, namely, when the puzzles are constructed using rays and equipotentials (see Section~\ref{SSec:Examples}), then each puzzle piece of $F$ contains an open set which is not in $K(F)$ (these sets are in the Fatou set of the corresponding rational map). In this situation, conclusion \eqref{parta} of Theorem~\ref{Thm:BoxMappingsMain2} cannot be satisfied, and thus such $F$ carries no invariant line field provided $F$ is non-renormalizable.

\item
\label{Rem:ext}
Conditions \eqref{part3c} and \eqref{part3b} should be understood in the following sense. Since, by assumption, $H \colon \V \to \tilde \V$ is a quasiconformal homeomorphism, it has a continuous homeomorphic extension $\hat H \colon \cl \V \to \cl \tilde \V$ to the closures (see \cite[Corollary 5.9.2]{AIM}). This allows to use $H$ in the definition of combinatorial equivalence (Definition~\ref{DefA:CombEquivBoxMappings}), and condition \eqref{part3b} should be understood as
\[
\hat{\tilde F} \circ \hat H = \hat H \circ \hat F \text{ on }\partial \U,
\]
where $\hat F$, $\hat{\tilde F}$ are the continuous extensions of $F$, $\tilde F$ to the boundary. 

\item 
If $\mathcal U$ only has a finite number of components, and each component of $\U,\V,\tilde \U,\tilde \V$ is a quasi-disk  and a combinatorial equivalence  $H \colon \V \to \tilde \V$ as in 
Definition~\ref{DefA:CombEquivBoxMappings}, then one can 
construct a new quasiconformal map $H$ (by hand) which is a combinatorial equivalence and is a conjugacy on the boundary of $\U$. Even if components of $\U,\V,\tilde \U,\tilde \V$ are not quasi-disks, then one can slightly shrink the
domains and still obtain a qc conjugacy between $F,\tilde F$ on a neighborhood
of their filled Julia sets.  One also can often find a qc homeomorphism which is a conjugacy on the boundary using B\"ottcher coordinates when they exist. So in that case one essentially has that combinatorially equivalent box mappings $F, \tilde F$ are necessarily qc conjugate.

On the other hand, if $\mathcal U$ has infinitely many components, then it is easy to construct 
examples so that $F,G$ are topologically conjugate while they are not qc conjugate.
\end{enumerate}
\end{remark}

\subsection{Applications to box mappings which are not dynamically natural or which are renormalizable}
\label{SSec:RigidityGeneral}
In this subsection we discuss what can be said about rigidity of box mappings when we relax some of the assumptions in Theorem~\ref{Thm:BoxMappingsMain2}, that is if we do not assume that the box mapping is dynamically natural and/or non-renormalizable.

First of all, the proof of Theorem~\ref{Thm:BoxMappingsMain2} \eqref{It:MainIt1} (given in Section~\ref{Sec:LC}) yields the following, slightly more general result: \emph{if $F \colon \U \to \V$ is a non-renormalizable complex box mapping satisfying the no permutation condition, then each point in $\Kwi(F)$ has trivial fiber}.

Furthermore, in \cite{DS} it was shown that Theorem~\ref{Thm:BoxMappingsMain2} \eqref{It:MainIt1} can be extended to the following result \cite[Theorem C]{DS}: \emph{for a complex box mapping $F \colon \U \to \V$ and for every $x \in K(F)$ at least one of the following possibilities is true: (a) the fiber of $x$ is trivial; (b) the orbit of $x$ lands into a filled Julia set of some polynomial-like restriction of $F$; (c) the orbit of $x$ lands into a permutation component of $\V$ (i.e. there exists $k\ge 1$ and a component $V$
of $\V$ which is also a component of $\U$ so that $F^k(V)=V$); (d) the orbit of $x$ converges to the boundary of $\V$}. 

It is easy to see that for dynamically natural complex box mappings the possibility (c) cannot occur, while points satisfying (d) have trivial fibers, and hence \emph{for dynamically natural mappings every point in its non-escaping set has either trivial fiber, or the orbit of this point lands in the filled Julia set of a polynomial-like restriction}. These polynomial-like restrictions can be then further analyzed using the results on polynomial rigidity. For example, if the corresponding polynomials are at most finitely many times renormalizable and have no neutral periodic cycles, then the result in \cite{KvS} tells that each point in the filled Julia set has trivial fiber, and hence the points in the whole non-escaping set of the original box mapping have trivial fibers.

The discussion in the previous two paragraphs indicates the way how to extend Theorem~\ref{Thm:BoxMappingsMain2} \eqref{It:MainIt2}, \eqref{It:MainIt3} for more general box mappings. 

For example, using Theorem~\ref{Thm:BoxMappingsMain2} \eqref{It:MainIt2}, one can show that \emph{for every dynamically natural complex box mapping $F$, the support of any measurable invariant line field on $K(F)$ is contained in puzzle pieces with the properties described in \eqref{parta} and \eqref{partb} union the cycles of little filled Julia sets of possible polynomial-like renormalizations of $F$} (compare \cite{Dr}). 

Similarly, if for a pair of dynamically natural complex box mappings $F$ and $\tilde F$ that satisfy assumptions of Theorem~\ref{Thm:BoxMappingsMain2} \eqref{It:MainIt3} all polynomial-like renormalizations can be arranged in pairs $F^s \colon U \to V$, $\tilde F^s \colon \tilde U \to \tilde V$ so that $U, \tilde U$ and $V, \tilde V$ are the corresponding puzzle pieces and $F^s|_U$, $\tilde F^s|_{\tilde U}$ are quasiconformally conjugate, then Theorem~\ref{Thm:BoxMappingsMain2} \eqref{It:MainIt3} can be used to conclude that $F$ and $\tilde F$ are quasiconformally conjugate   \cite{DS}. Some conditions, for example in terms of rays for polynomial-like mappings, can be imposed in order to \emph{conclude} rather than \emph{assume} quasiconformal conjugation of the corresponding polynomial-like restrictions.   

The circle of ideas on how to ``embed'' results on polynomial dynamics into some ambient dynamics using complex box mappings, similar to the ones described in this subsection, is discussed in \cite{DS} under the name \emph{rational rigidity principle}.

\subsection{Box mappings for which the domains are not Jordan disks.}
\label{SSec:NotJordan}

Let us now discuss Theorem~\ref{Thm:BoxMappingsMain2} in the setting where 
we no longer assume that the components of $\U$, $\V$ are Jordan disks, but instead assume 
that  each component  component $U$ of $\U$ and $V$ of $\V$
\begin{enumerate}[label=(\roman*)]
\item 
\label{It:i}
is simply connected;
\item
\label{It:ii}
when $U\Subset V$ then $V\setminus \overline{U}$ is a topological annulus;
\item  
\label{It:iii}
has a locally connected boundary.
\end{enumerate}
This means that $U,V$ could for example be sets of the form $\mathbb D\setminus [0,1)$.
Such domains are often used for box mappings associated to real polynomial maps or to 
maps preserving the circle. 
In this setting 
\begin{itemize}
\item [-] Parts \eqref{It:MainIt1} and \eqref{It:MainIt2} of Theorem~\ref{Thm:BoxMappingsMain2}  hold, 
provided Assumptions \ref{It:i}, \ref{It:ii} are satisfied.  The  proof is identical as in the Jordan case, because 
these assumptions ensure that one  still has  annuli, $V\setminus U$, as before. 
\item[-] Part \eqref{It:MainIt3} of Theorem~\ref{Thm:BoxMappingsMain2} holds,  provided 
Assumptions \ref{It:i}, \ref{It:ii}, \ref{It:iii} are satisfied and, moreover,  the map $H\colon \V\to \tilde \V$ 
extends to a homeomorphism $\hat H\colon \cl{\V} \to \cl{\tilde \V}$. 
To see this, observe that the QC-Criterion, Theorem~\ref{Thm:QCCriterion}, also holds
for simply connected domains  $\Omega,\tilde \Omega$  provided that we now also assume 
that the homeomorphism   $\phi\colon \Omega\to \tilde \Omega$ extends 
to a homeomorphism $\hat \phi\colon \cl{\Omega}\to \cl{\tilde \Omega}$.  (The proof of this version
of the QC-criterion is the same as in \cite{KSS}, since Assumption \ref{It:iii} implies that we can apply
the Carath\'eodory Theorem to get that the 
Riemann mapping $\pi\colon \disk\to \Omega$ extends
to a continuous map $\hat \pi\colon \cl \disk \to \cl{\Omega}$ and thus the map 
 $\phi\colon \Omega\to \tilde \Omega$ induces a homeomorphism $\ovl \phi \colon \cl \disk \to \cl \disk$. Because $\phi$ extends continuously to $\partial \Omega$, the map $\ovl \phi$
 has the additional property that if $x,y\in \partial \disk$ and 
$\hat\pi(x)=\hat\pi(y)$, then $\ovl \phi(x)=\ovl \phi(y)$. The $K$-qc map 
$\hat \psi\colon \disk \to \disk$ so that $\hat \psi=\ovl \phi$   on $ \partial \disk$ 
which is produced in the proof the QC-criterion, 
will now again have the same property that if $x,y\in \partial \disk$ and 
$\hat\pi(x)=\hat\pi(y)$, then $\hat \psi(x)=\hat \psi(y)$.
Therefore, $\hat \psi$ produces a $K$-qc map $\psi\colon \Omega \to \tilde \Omega$ which agrees with $\phi$ on $\partial \Omega$.)  Moreover, since $H$ now is assumed to extend to a homeomorphism
on the closure of the domains, the assumption that $H$ is a conjugacy on $\partial U$ makes sense.  
Finally, the gluing theorems from Section~ \ref{SSec-QC-gluing} and the Spreading Principle 
go through without any further change. 
\end{itemize}

\subsection{Outline of the remainder of the paper}
\label{SSec:Outline}
Let us briefly describe the structure of the upcoming sections. Theorem~\ref{Thm:BoxMappingsMain2} contains three assertions about dynamically natural complex box mappings with all periodic points repelling: 
\begin{itemize}
\item[ \eqref{It:MainIt1}] on shrinking of puzzle pieces; 
\item[\eqref{It:MainIt2}] properties of box mappings with measurable invariant line fields;  and 
\item[\eqref{It:MainIt3}] on quasiconformal rigidity (\emph{qc rigidity} for short) of combinatorially equivalent, dynamically natural box mappings, $F\colon\U\to\V$ and $\tilde F\colon\tilde\U\to\tilde\V$.
\end{itemize}

We outline the main ingredients that go into the proof of these assertions in Sections~\ref{Sec:Ing} and \ref{Sec:Ing2}. For clarity and better referencing, these ingredients are shown in the diagram in Figure~\ref{Fig:Deps}, with arrows indicating dependencies between them. 

The main ingredients can be roughly divided into two groups: complex {\em a priori} bounds (discussed in Section~\ref{Sec:Ing}), and analytic tools (the Spreading Principle and the QC-Criterion, discussed in Section~\ref{Sec:Ing2}). Complex bounds, see Theorem~\ref{Prop:UC}, gives us uniform control of the geometry of arbitrarily deep, combinatorially defined puzzle pieces. Complex bounds is the chief technical result needed to prove both shrinking of puzzle pieces (assertion \eqref{It:MainIt1}) and ergodic properties of box mappings and properties of mappings with measurable invariant line fields (assertion \eqref{It:MainIt2}). The shrinking of puzzle pieces is proved, assuming complex bounds,  in Section \ref{Sec:LC}. That the existence of an invariant 
line field implies the existence of a smooth foliation on some puzzle piece $\J$ is shown in 
 Section~\ref{Sec:ILF} (again assuming complex bounds).  In Section~\ref{sec:lattes}
 we give an example of a dynamically natural box mapping with an invariant line field and 
 show that such mappings have a very particular structure.  We call these \emph{Latt\`es box mappings}.
We outline the proof of complex bounds in Section~\ref{Sec:UC}.

It is important that the complex bounds are obtained for combinatorially defined puzzle pieces, since this ensures that corresponding puzzle pieces for combinatorially equivalent complex box mappings have good geometric properties simultaneously. This is a starting point of the proof of qc rigidity (assertion \eqref{It:MainIt3}) and it allows us to use the QC-Criterion (Theorem~\ref{Thm:QCCriterion}). 

 In Section~\ref{Sec:Ing2},  we continue the buildup to the explanation of the proof of qc rigidity. In that section, we explain the Spreading Principle and the QC-Criterion. The Spreading Principle, Theorem~\ref{Thm:Spreading}, allows one to obtain a partial quasiconformal conjugacy between a pair of combinatorially equivalent complex box mappings $F$ and $\tilde F$ outside of some puzzle neighborhood $\Y$ of $\Crit(F)$, provided that one has a quasiconformal mapping from $\Y$ to the corresponding puzzle neighborhood $\tilde \Y$ for $\tilde F$ which respects the boundary marking, see Definition~\ref{Def:BdMark}. Using both the QC-Criterion and the Spreading Principle, one can construct quasiconformal mappings that respect boundary marking between \emph{arbitrarily deep} puzzle neighborhoods $\Y$ and $\tilde \Y$. 
 
 In Section~\ref{Sec:Sketch}, we combine all of the ingredients and explain the proof of qc rigidity. 

While we will assume that complex box mappings are dynamically natural, quite often we will only use part of the definition: to prove complex bounds around recurrent critical points, we will only need Definition~\ref{Def:NaturalBoxMapping}\eqref{It:Nat1}; we will use Definition~\ref{Def:NaturalBoxMapping}\eqref{It:Nat2} to prove the Spreading Principle; and we will use Definition~\ref{Def:NaturalBoxMapping}\eqref{It:Nat3} to obtain complex bounds around non-recurrent critical points.

In Section~\ref{Sec:Mane}, we prove a Ma\~n\'e-type result for dynamically natural complex box mappings that was mentioned earlier in the introduction.


\hskip -2.4cm
\begin{minipage}{\textwidth}
\hskip -2cm 
\definecolor{ColorKSS}{rgb}{0.89, 0.02, 0.17}
\definecolor{ColorKvS}{rgb}{0.13, 0.55, 0.13}
\definecolor{ColorKL}{rgb}{0.0, 0.45, 0.73}
\definecolor{CSingle}{rgb}{0.4, 0.26, 0.13}
\definecolor{CPair}{rgb}{0.73, 0.4, 0.16}


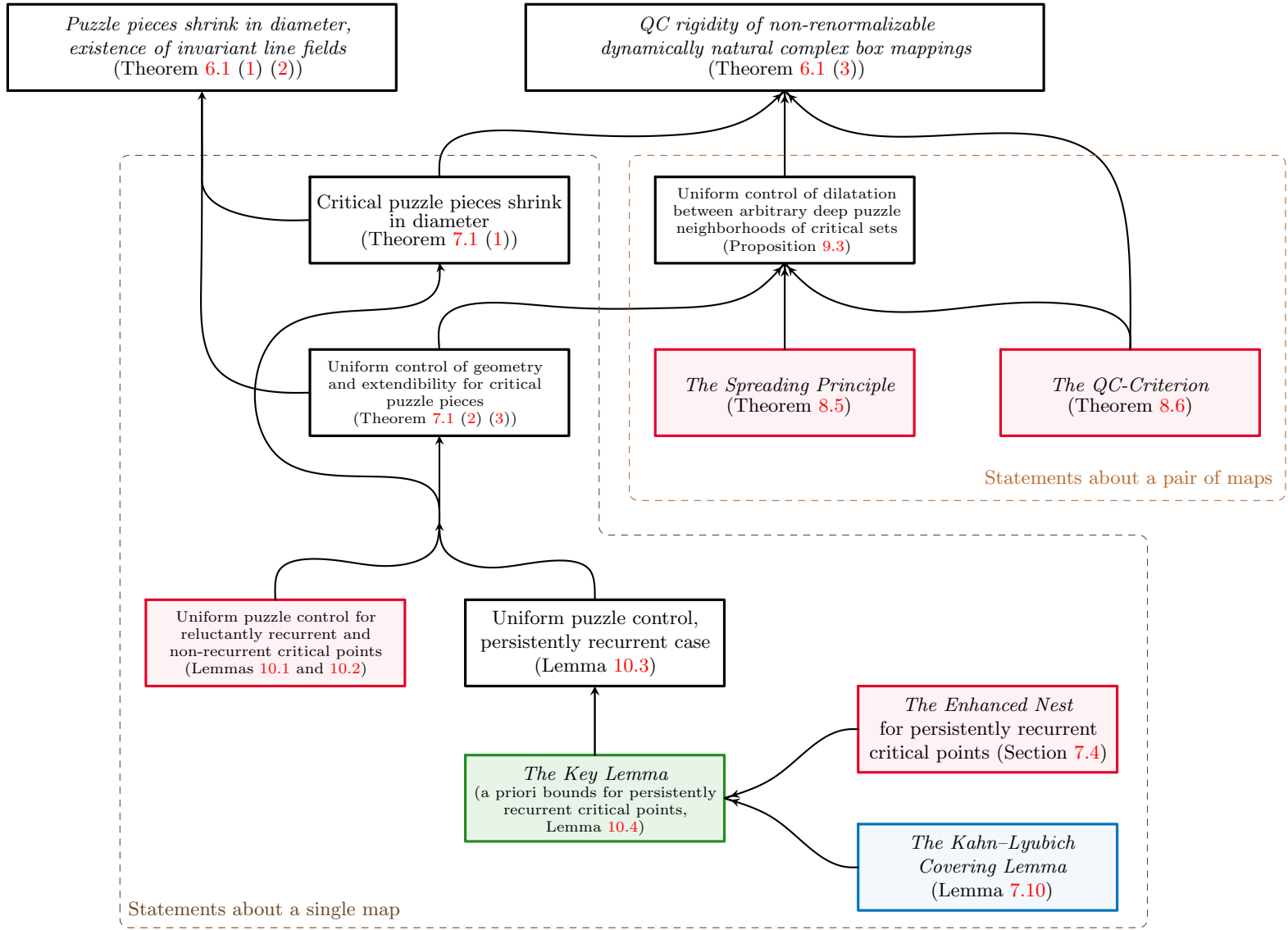
\captionof{figure}{The connection between the main ingredients that go into the proof of rigidity for non-renormalizable dynamically natural complex box mappings. The input ingredients are marked according to the following color scheme: the red boxes contain the ingredients coming from \cite{KSS}, the green box contains the ingredient coming from \cite{KvS}, and finally, the blue box contains the ingredient proven in \cite{KL09} (and used in \cite{KvS}).}
\label{Fig:Deps}
\end{minipage}



\section{Complex bounds, the Enhanced Nest and the Covering Lemma}
\label{Sec:Ing}

Complex bounds, also known as {\em a priori} bounds for a complex box mapping $F\colon\U\to\V$ concern uniform geometric bounds on arbitrarily small, combinatorially defined puzzle pieces. We will focus on two types of geometric control: \emph{moduli} bounds and \emph{bounded geometry}. This is the geometric control needed to apply the QC-Criterion in the proof of quasiconformal rigidity.

It is not possible to control the geometry of every puzzle piece, and to obtain complex bounds, one must restrict ones attention to a chosen, combinatorially defined, nest of critical puzzle pieces, and show that one has complex bounds for the puzzle pieces in this nest. One example of a nest of puzzle pieces is the principal nest. In this section, we will discuss some of the geometric properties of the principal nest. We will also describe the Enhanced Nest, which was introduced in \cite{KSS}. It is combinatorially more difficult to define than the principal nest and it is only defined around persistently recurrent critical points, but it also provides much better geometric control on the puzzle pieces than the principal nest, and it is especially useful for mappings with more than one critical point.

\subsection{Complex bounds for critical neighborhoods}
\label{SSec:GeomControl}

Define $\CritP(F) \subset \Crit(F)$ to be the set of critical points whose orbits \textbf{a}ccumulate on at least one \textbf{c}ritical fiber. Since it is a combinatorially defined set, it follows that for a combinatorially equivalent box mapping $\tilde F$ the set $\CritP(\tilde F)$ consists of the corresponding critical points. Note that in real one-dimensional dynamics if $\CritP(F) = \emptyset$, then $F$ is usually called a \emph{Misiurewicz map}. 

The following theorem proven in \cite[Proposition 6.2]{KSS} in the real case and extended to non-renormalizable complex maps in \cite{KvS}.

\begin{theorem}[Shrinking neighborhoods with uniform control]
\label{Prop:UC}
If $F \colon \U \to \V$ is a non-renormalizable dynamically natural complex box mapping, then there exist a $\delta > 0$ and an integer $N$ such that the following holds. For every $\epsilon >0$, there is a combinatorially defined nice puzzle neighborhood $W$ of $\Crit(F)$ with the following properties:
\begin{enumerate}
\item
\label{It:Shrink}
Every component of $W$ has diameter $< \epsilon$;
\item
\label{It:BddGeom}
Every component of $W$ intersecting $\CritP(F)$ has $\delta$-bounded geometry;
\item
\label{It:Ext}
The first landing map under $F$ to $W$ is $(\delta, N)$-extendible with respect to $\Crit(F)$.
\end{enumerate}
Moreover, if $\tilde F$ is another non-renormalizable dynamically natural box mapping combinatorially equivalent to $F$, then the same claims as above hold for $\tilde F$ and the corresponding objects with tilde.
\end{theorem}

\begin{remark}
In fact, for persistently recurrent critical points, we have better control.
There exists a {\em beau} constant $\delta'>0$ such that if a component $W_c$ of $W$ is a neighborhood of a persistently recurrent critical point, then $W_c$ has $\delta'$-bounded geometry and the first landing map to $W_c$ is $(\delta', N)$-extendible with respect to $\Crit(F)$. The constant $N$ is bounded by some universal function depending only on the number and degrees of the critical points of $F$. 
\end{remark}

Let us recall a notion of $(\delta, N)$-extendibility defined in \cite[page 775]{KSS}. Let $\delta>0$ and $N \in \mathbb N$. Let $V' \supset V$ be a pair of nice open puzzle neighborhoods of some set $A \subset \Crit(F)$, and $B$ be any subset of $\Crit(F)$. For $a \in A$, denote $V_a = \Comp_a V$; similarly for $V'_a$. We say that the first landing map $L_V$ to $V$ is \emph{$N$-extendible to $V'$ w.r.t.\ $B$} if the following hold (see Figure~\ref{Fig:DeltaNExt}): if $F^s \colon U \to V_a$ is a branch of $L_V$, and if $U' = \Comp_U(F^{-s}(V'_a))$, then
\[
\# \left\{0 \le j \le s-1 \colon (F^j(U') \sm F^j(U))\cap B \neq \emptyset\right\} \le N.
\]
Moreover, $L_V$ is \emph{$(\delta,N)$-extendible w.r.t.\ $B$} if there exists a puzzle piece $V_a' \supset V_a$ for every $a \in A$ such that $\modulus(V_a' \sm \ovl V_a) \ge \delta$ and $L_V$ is $N$-extendible to $\bigcup_{a \in A}V_a'$ w.r.t.\ $B$.

\begin{remark} Note that this definition is slightly more demanding than the one from
\cite{KSS}  because here we require that the set $V_a'$ to be a puzzle piece
rather than a topological disk.  The reason this can be  done is because 
we are able to prove this sharper statement in Lemma~\ref{Lem:PR}
as now we can rely on the estimate provided in 
Proposition~\ref{prop:KvSnice} from \cite{KvS}.
In \cite{KSS} such an estimate is not available.  
\end{remark}

\begin{figure}[htbp]
\begin{center}
\includegraphics[scale=.9]{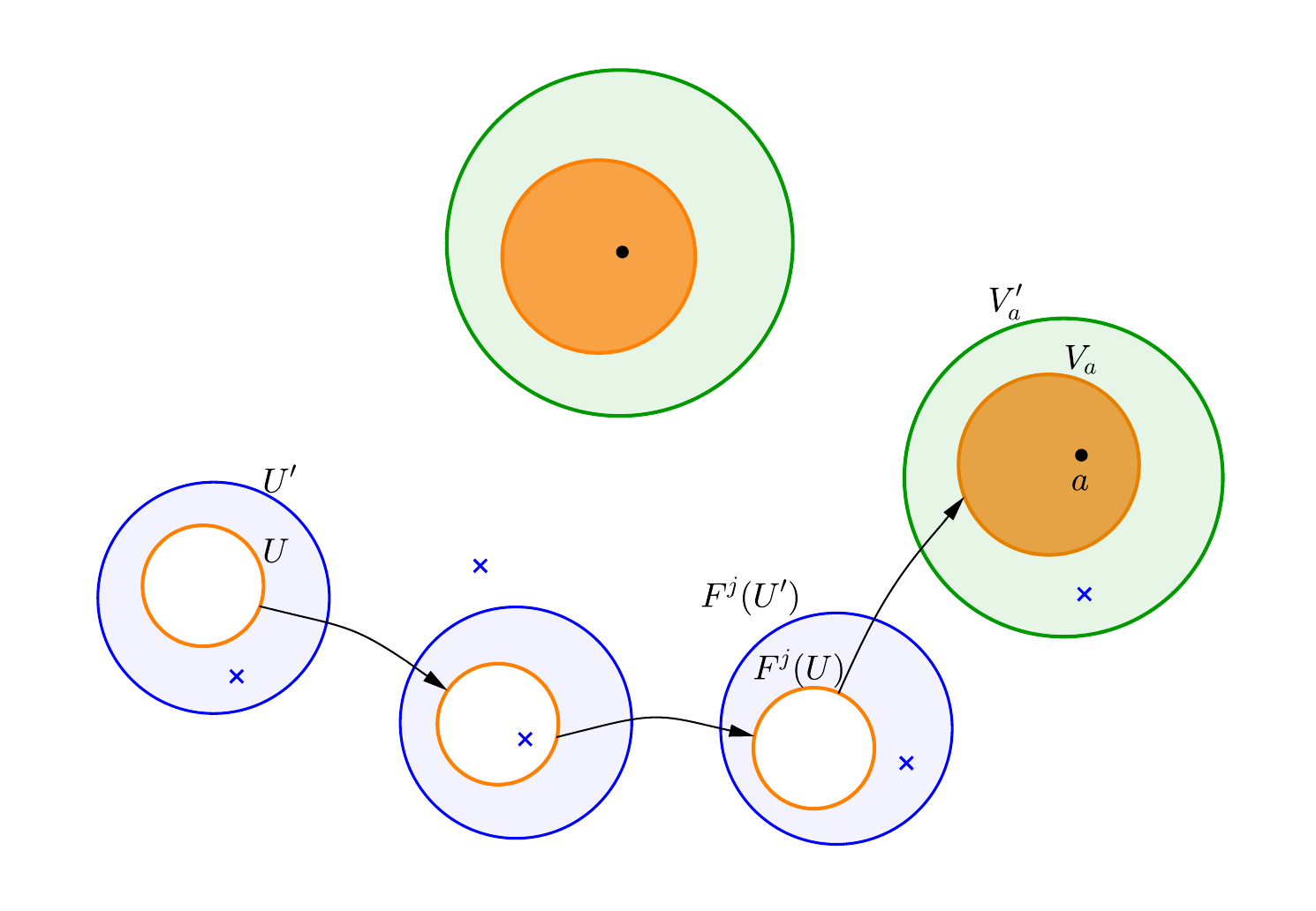}
\caption{A schematic drawing of relevant annuli in the definition of $(\delta,N)$-extendibility: $\delta$ controls the moduli of green annuli, $N$ controls the number of critical points in the blue annuli for each branch $F^s \colon U \to V_a$ of the first landing map $L_V$. The set $V$ is the union of solid orange disks; the set $A$ is indicated with black dots, the set $B$ is indicated with blue crosses.}
\label{Fig:DeltaNExt}
\end{center}
\end{figure}

We will sketch the proof of Theorem~\ref{Prop:UC} in Section~\ref{Sec:UC}. There are two main ingredients to this proof. The first one is the construction from \cite{KSS} of the \emph{Enhanced Nest} around persistently recurrent critical points. The second one is the \emph{Covering Lemma} of Kahn and Lyubich~\cite{KL09}, see \cite{KvS}. In \cite{KSS} other methods were used in the real case because \cite{KL09} was not 
available, and also to deal with the infinitely renormalizable case. Before going on to discuss these ingredients in greater detail, we will describe the geometry of the principal nest.

\subsection{The principal nest}

Let $c \in \Crit(F)$ be a critical point. The \emph{principal nest} (around $c$) is a nested collection of puzzle pieces $\bV^0 \supset \bV^1 \supset \bV^2 \supset \ldots$ such that 
\[
c \in \bV^0, \text{ and } \bV^{i+1} = \DomE_{c} \bV^i \text{ for every integer } i \ge 0.  
\]

Note that we can start the principal nest from an arbitrary strictly nice puzzle piece $\bV^0$ containing $c$. Clearly, if $c$ is combinatorially non-recurrent, then the principal nest is finite. Otherwise, the nest consists of infinitely many puzzle pieces. Moreover, each  $\bV^n$ contains $\fib(c)$, and hence all the critical points in $\fib(c)$. 

Associated to such the principal nest we have a sequence of return times $(p_n)$ defined as $\bV^{n+1} = \Comp_c F^{-p_n}(\bV^n)$; $p_n$ is the first time the orbit of $c$ returns back to $\bV^n$. Since the mapping from $\bV^{n+1}$ to $\bV^n$ is a first return mapping, and puzzle pieces are nice sets, the sets
$F^j(\bV^{n+1}), F^{j'}(\bV^{n+1})$ are pairwise disjoint for $0<j<j'\le p_n$. Thus all the maps $F^{p_n} \colon \bV^{n+1} \to \bV^n$ have uniformly bounded degrees.

\subsubsection{Geometry of the principal nest}

Let us describe some of the results concerning the geometry of the principal nest for unicritical mappings. 
Suppose that $F\colon\U\to\V$ is a dynamically natural complex box mapping with the properties that 
\begin{itemize}
\item $\V$ consists of exactly one component, 
\item there exists exactly one component $U_0$ of $\U$ that is mapped as a degree $d$ branched covering onto $\V$ which is ramified at exactly one point, and 
\item every component $U$ of $\U$, $U\neq U_0$ is mapped diffeomorphically onto $\V$.
\end{itemize}
Let $V^0=\V$ and $$V^0\supset V^1\supset V^2\supset\dots$$
be the principal nest about the critical point $c$ of $F$. 

There are two combinatorially distinct types of returns $F^{p_n}\colon V^{n+1}\to V^n$: {\em central} and {\em non-central}. We say that the return mapping from $V^{n+1}$ to $V^n$ is {\em central} if $F^{p_n}(c)\in V^{n+1}$. Otherwise the return is called {\em non-central}. 

For quadratic mappings, the following result is fundamental.
\begin{theorem}\cite[Decay of Geometry]{GS,Lyu} 
Let $n_k$ denote the subsequence of levels in the principal nest so that $F^{p_{n_k-1}}\colon V^{n_k}\to V^{n_k-1}$ is a non-central return. There exists a constant $C>0$ so that 
\[
\modulus \left(V^{n_k}\setminus \ovl V^{n_k+1}\right)>Ck.
\]\qed
\end{theorem}
For example, by the Koebe Distortion Theorem this result implies that the landing maps to $V^{n_k+1}$ become almost linear as $k\to\infty$. This result played a crucial role in the proofs of quasiconformal rigidity and density of hyperbolicity in the quadratic family.

For higher degree maps, even cubic mappings with two non-degenerate critical points and unicritical cubic mappings, Decay of Geometry does not hold; however, the geometry is bounded from below. The following result required new analytic tools --- the Quasiadditivity Law and Covering Lemma --- of \cite{KL09}.

\begin{theorem}[\cite{AKLS, ALS}]
Let $n_k$ denote the subsequence of levels in the principal nest so that $F^{p_{n_k-1}}\colon V^{n_k}\to V^{n_k-1}$ is a non-central return. There exists a constant $\delta>0$ so that 
\[
\modulus \left(V^{n_k}\setminus \ovl V^{n_k+1}\right)>\delta.
\]
\qed
\end{theorem}
Even for unicritical mappings, this difference makes the study of higher degree mappings quite a bit different from the study of unicritical mappings with a critical point of degree 2.

In general, there is no control of the moduli of the annuli $V^n\setminus \ovl V^{n+1}$ when the  returns are central.
Observe that when the return mapping from $V^{n+1}$ to $V^n$ is central, we have that the return time of $c$ to $V^{n+1}$ agrees with the return time of $c$ to $V^n$, and hence the annulus $V^{n+1}\setminus V^{n+2}$ is mapped by a degree $d$ covering map onto the annulus $V^{n}\setminus V^{n+1}$. Hence 
\[
\modulus \left(V^{n+1}\setminus \ovl V^{n+2}\right) = \frac{1}{d}\modulus\left(V^n\setminus \ovl V^{n+1}\right).
\]
Consequently, when there is a long cascade of central returns; that is, when $F^{p_n}(c)\in V^{n+k}$ for $k$ large, we have that $\modulus(V^{n+k}\setminus \ovl V^{n+1+k}) = \frac{1}{d^k}\modulus(V^n\setminus \ovl V^{n+1})$ may be arbitrarily small. In general, it is impossible to bound the length of long central cascades, so these moduli can degenerate. 

For unicritical mappings, the complex {\em a priori} bounds imply that if $\{n_{k}\}_{k=0}^\infty$ is the sequence of levels in the principal nest so that  $F^{p_{n_k-1}}\colon V^{n_k}\to V^{n_k-1}$ is non-central, then the mapping $ F^{p_{n_k}}\colon V^{n_k+1}\to V^{n_k}$ is $(\delta(k), 1)$-extendible, where $\delta(k)\to\infty$ at least linearly as $k\to\infty$ when the critical point has degree $2$, and $\delta(k)$ is bounded from below for higher degree mappings.

\subsubsection*{Example: Puzzle pieces in the principal nest do not have bounded geometry in general.}
Despite having moduli bounds, one cannot expect to have bounded geometry for the principal nest, even at levels following non-central returns. For example, suppose that $F\colon \U\to \V$ is real-symmetric 
box mapping, and that  $F^{p_{n}}\colon V^{n+1}\to V^n$ has a long central cascade, so that $F^{p_n}(0)\in V^{n+k}$, and that $F^{p_n}(V^{n+1})\cap\R\owns 0$.\footnote{Such a return is called high.} Then for any $\epsilon>0$, as the length of the cascade tends to infinity the puzzle piece $V^{n+k}$ is contained in an $\epsilon \cdot\mathrm{diam}(V^{n+k})$-neighborhood of its real trace. The reason for this is that as $k\to \infty,$ the  puzzle pieces in the principal nest converge to the Julia set of the return mapping, which in this case, is close to the real trace of $V^{n+k}$ as the return mapping, up to rescaling, is close to $z\mapsto z^2-2$. Hence in general one cannot recover $\delta$-bounded geometry for subsequent puzzle pieces, even if the return mappings are non-central.

\subsubsection*{Example: the Fibonacci map.}
 Let us assume that $F\colon\U\to\V$ is a real-symmetric, unicritical complex box mapping with quadratic critical point at 0 and {\em Fibonacci combinatorics} (see Figure~\ref{Fig:Fib}, left); that is, for each $V^n$ we have that there are exactly two components $V^{n+1}$ and $V^n_1$ of the domain of the first return map to $V^n$ that intersect the postcritical set of $F$. Moreover, we assume that for all $n$, $F^{p_n}(0)\in V^n_1,$ so that all of the first return mappings to the principal nest are non-central, that the return mapping from $V^n_1$ to $V^n$ is given by $F^{p_{n-1}}|_{V^n_1},$ and that restricted to the real trace $F^{p_{n}}(V^{n+1}\cap\R)\owns 0$; that is, these return mappings are all {\em high}.

\begin{figure}[h]
\begin{center}
\includegraphics[width=\textwidth, trim=60 20 50 20, clip]{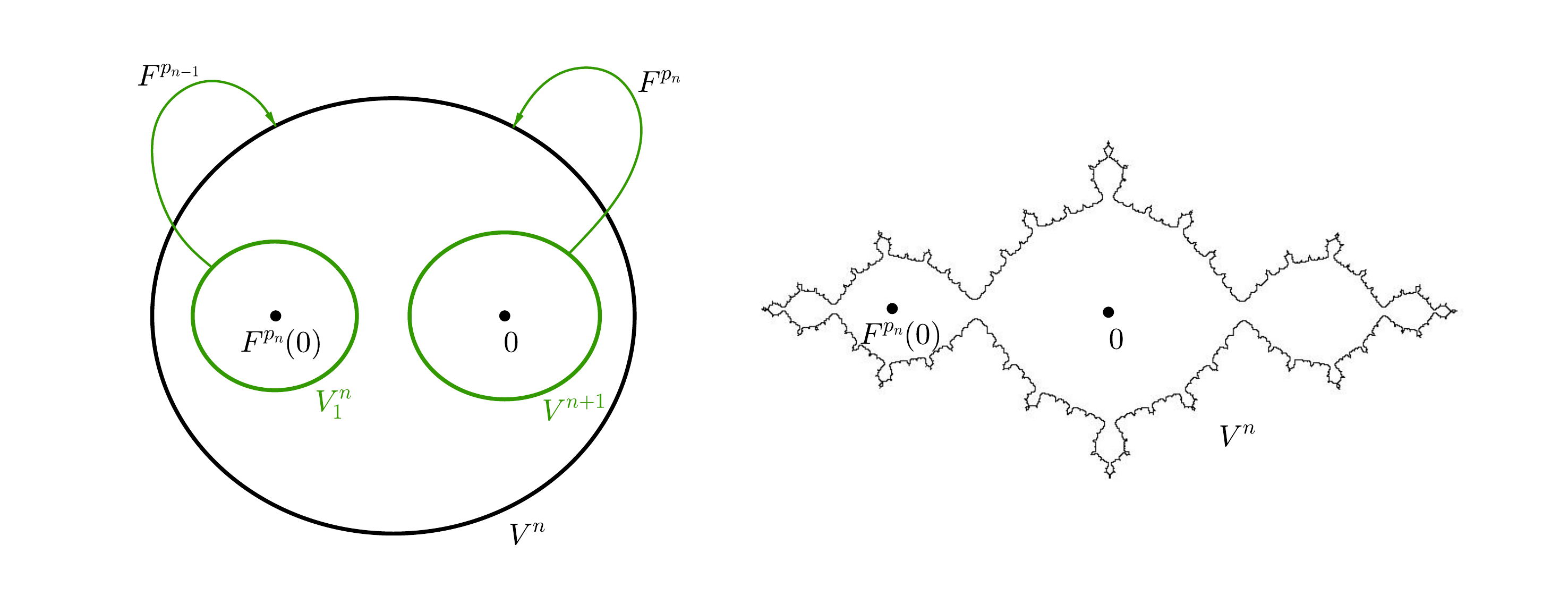}
\caption{(Left) The Fibonacci combinatorics. (Right) The actual shape of the $n$-th puzzle piece $V^n$ in the principal nest of the Fibonacci mapping of degree two.}
\label{Fig:Fib}
\end{center}
\end{figure}

For (quadratic) Fibonacci mappings, Decay of Geometry together with the precise description of the combinatorics implies that the puzzle pieces $V^n$ converge to the shape of the filled Julia set of $z\mapsto z^2-1$, {\em the basilica} \cite{Lyu-Fib}. In particular, these puzzle pieces become badly pinched and the hyperbolic distance between 0 and $F^{p_n}(0)$ in $V^{n}$ tends to infinity, see Figure~\ref{Fig:Fib}, right. Consequently, these puzzle pieces are not $\delta$-free (see Section~\ref{SSec:Defns} to recall the definition). However, the return mappings are $(K,1)$-extendible for $K$-arbitrarily large at sufficiently deep levels, and there exists $\delta>0$, independent of $n$, so that the puzzle pieces have $\delta$-bounded geometry.

\subsection{Inducing persistently recurrent box mappings}

Recall that a box mapping $F' \colon \U' \to \V'$ is \emph{induced} from $F \colon \U \to V$ if $\U'$ and $\V'$ consist of puzzle pieces of $F$ and the branches of $F'$ are the iterates of some branches of $F$. 

\begin{definition}[Persistently recurrent box mapping]
\label{Def:BMPerRec}
We say that a non-renormalizable dynamically natural box mapping $F$ is \emph{persistently recurrent} if every critical point of $F$ is persistently recurrent and its orbit accumulates at every critical fiber of $F$. The latter is equivalent to the property that $\Crit(F) = [c]$ for some, and hence for every $c \in \Crit(F)$ (see notation in Section~\ref{SSec:Rec}). 
\end{definition}

The next lemma shows that starting from an arbitrary box mapping $F$ (non-renormalizable and dynamically natural) and a persistently recurrent critical point $c \in \Crit(F)$ we can extract an induced box mapping $F'$ that is persistently recurrent and contains $c$, and hence $\fib(c)$, in its non-escaping set. More precisely:

\begin{lemma}[Inducing persistently recurrent box mappings]
\label{Lem:PRBM}
Let $F$ be a dynamically natural box mapping that is non-renormalizable, and $c \in \Crit(F)$ be a persistently recurrent critical point of $F$. Then there exists an induced persistently recurrent box mapping $F'$ such that $c \in \Crit(F') \subset \Crit(F)$.
\end{lemma}
 
\begin{proof}
By definition of persistent recurrence, the set $[c]$ consists only of persistently recurrent critical points. Choose an integer $m$ large enough so that each critical puzzle piece of depth at least $m$ contains a single critical fiber. Let $\V'_n$ be the union of all puzzle pieces of $F$ of depth $n \ge m$ intersecting $[c]$; this is a finite union of critical puzzle pieces containing only persistently recurrent critical fibers. Let $\U'_n$ be the union of the components of $\Dom(\V'_n)$ intersecting $[c]$, and define a box mapping $F'_n \colon \U'_n \to \V'_n$ to be the restriction to $\U_n'$ of the first return map to $\V_n'$ under $F$. 

Let us show that there exists an $n (\ge m)$ such that the only critical points of $F'_n$ are the ones in $[c]$. This will be our desired box mapping $F'$.

Assume the contrary, and let $U_n \subset \U'_n$ and $V_n \subset \V'_n$ be a pair of components containing, respectively, $c' \in [c]$ and $c'' \in [c]$, and such that $\Crit(F') \cap U_n \not\subset [c]$. This means that if $F^s \colon U_n \to V_n$, $s = s(n)$ is the corresponding branch of the first return map, then one of the puzzle pieces in the sequence $F(U_n), F^2(U_n), \ld, F^{s-1}(U_n)$ contains $c''' \in \Crit(F)$ which is either a non-persistently recurrent critical point or a persistently recurrent critical point with $[c]\cap[c''']=\emptyset$. If this is true for every $n$, then the orbits of $c'$, $c''$ and $c'''$ must all visit arbitrarily deep puzzle pieces around each other. In particular, $c'''$ must be combinatorially recurrent and $[c''']$ must intersect $[c']$. On the other hand, $c'''$ cannot be reluctantly recurrent: the critical point $c'$ will then have only finitely many children, contrary to the definition of reluctant recurrence. We arrived at a contradiction. 
\end{proof}

\subsection{The Enhanced Nest}
\label{SSec:Enhanced}

In this subsection, we review the construction and properties of the Enhanced Nest  of puzzle pieces defined around a chosen persistently recurrent critical point.

Throughout Section~\ref{SSec:Enhanced}, unless otherwise stated we will assume that \emph{all complex box mappings are dynamically natural, non-renormalizable, and persistently recurrent}. By Lemma~\ref{Lem:PRBM} we know that we can always induce such a mapping starting from a given (dynamically natural non-renormalizable) box mapping.   

The construction of the Enhanced Nest  is based on the following lemma \cite[Lemma 8.2]{KSS} (see Figure~\ref{Fig:AB}, left):

\begin{lemma}
\label{Lem:ENnu}
Let $F$ be a persistently recurrent box mapping, $c \in \Crit(F)$ be a (persistently recurrent) critical point, and $\bI \ni c$ be a puzzle piece. Then there exists a combinatorially defined positive integer $\nu$ with $F^\nu(c) \in \bI$ such that the following holds. If $\bU_0 := \Comp_c \left(F^{-\nu}(\bI)\right)$, and $\bU_j := F^j(\bU_0)$ for $0 < j \le \nu$, then
\begin{enumerate}
\item
\label{It:EN1}
$\# \left\{0 \le j \le \nu-1 \colon \bU_j \cap \Crit(F) \neq \emptyset\right\} \le b^2$, where $b$ is the number of critical fibers of $F$ (which all intersect $\ovl{\orb(c)}$);
\item
\label{It:EN2}
$\bU_0 \cap \PC(F) \subset \Comp_c\left(F^{-\nu}\left(\DomE_{F^\nu(c)}(\bI)\right)\right)$. \qed
\end{enumerate}
\end{lemma}
 
This lemma allows us to define the following pullback operators $\mathcal A$ and $\mathcal B$. Let $c \in \Crit(F)$ be a persistently recurrent critical point, and for a critical puzzle piece $\bI \ni c$, let $\nu = \nu(\bI)$ be the smallest possible integer with the properties specified by Lemma~\ref{Lem:ENnu}. Define
\begin{equation*}
\begin{aligned}
\mathcal A(\bI) &:= \Comp_c\left( F^{-\nu} \left(\DomE_{F^\nu(c)}(\bI)\right) \right),\\
\mathcal B(\bI) &:= \Comp_c \left(F^{-\nu}(\bI)\right).
\end{aligned}
\end{equation*}

Let us comment on the meaning of these operators. First of all, these operators depend on persistently recurrent critical point (or to be more precise, on the corresponding fiber, see Section~\ref{SSec:Fibers}). Once the critical point is fixed, say $c$, the operators $\mathcal A \colon \bI \mapsto \mathcal A (\bI)$ and $\mathcal B \colon \bI \mapsto \mathcal B (\bI)$ produce for any given puzzle piece $\bI \ni c$ its pullbacks $\mathcal A(\bI) \ni c$, respectively $\mathcal B(\bI) \ni c$ so that:
\begin{itemize}
\item
the degrees of the corresponding maps $F^\nu \colon \mathcal A(\bI) \to \bI$ and $F^\nu \colon \mathcal B(\bI) \to \bI$ are bounded only in terms of the local degrees of the critical points of $F$, and hence are independent of $\bI$ (we can thus speak of the bounds on degrees assigned to the operators); this property follows by Lemma~\ref{Lem:ENnu} \eqref{It:EN1} (note that $\bU_0 = \mathcal B (\bI)$ in that lemma);
\item
the puzzle piece $\mathcal A(\bI)$ has some ``external space'' free of the critical and postcritical set of $F$ (denoted earlier as $\PC(F)$), while the puzzle piece $\mathcal B(\bI)$ has some ``internal space'' disjoint from $\PC(F)$\footnote{One might interpret the notation as follows: the operator $\mathcal A$ gives puzzle pieces with some space free of $\PC(F)$ \textbf{A}bove the boundary of the piece, while the operator $\mathcal B$ gives pieces with some space free of $\PC(F)$ \textbf{B}elow the boundary of the piece.}, see Figure~\ref{Fig:AB}, left. This property is Lemma~\ref{Lem:ENnu} \eqref{It:EN2}.
\end{itemize}

It is immediate from the last property that if we apply $\mathcal B$ after $\mathcal A$, then for every $c \in \bI$ the puzzle piece $\mathcal B \mathcal A (\bI)$ will be the pullback of $\bI$ of uniformly bounded degree and will have an annular ``buffer'' around its boundary free of $\PC(F)$, see Figure~\ref{Fig:AB} right. This is a key step in the construction of the Enhanced Nest  below. In fact, these buffers can be written explicitly in terms of $\mathcal A$ and $\mathcal B$. If we set 
\[
P = \mathcal{B}\mathcal{A}(\bI),\quad
P^+=\mathcal{B}\mathcal{B}(\bI)\; \mbox{ and }\;
P^-=\mathcal{A}\mathcal{A}(\bI),
\]
then the annuli $P^+\sm \cl(P)$ and $P\setminus \cl(P^-)$ are disjoint from $\PC(F)$. In Lemma~\ref{Lem:KeyLemma}, we will show that there is a uniform bound on the moduli of such ``buffers'' for the elements in the Enhanced Nest.

\begin{figure}[h]
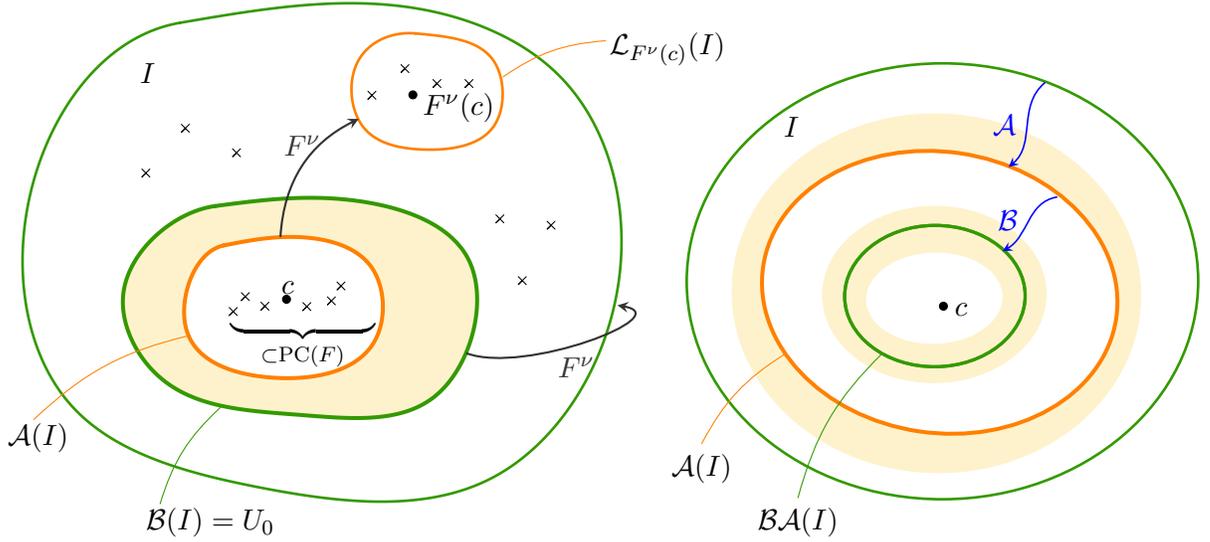

\begin{center}

\end{center}
\caption{The role of the pullback operators $\mathcal A$ and $\mathcal B$: the operator $\mathcal A$ creates some ``external'' space free of $\PC(F)$, while $\mathcal B$ creates some ``internal'' space free of $\PC(F)$. The annuli disjoint from $\PC(F)$ are shaded in yellow. The figure on the left illustrates the construction in Lemma~\ref{Lem:ENnu}; here, elements of the set $\PC(F)$ are marked with crosses and dots. The figure on the right illustrates the result of successive applications of $\mathcal A$ and $\mathcal B$: the puzzle piece $\mathcal B \mathcal A (\bI)$ has an annular ``buffer'' free of $\PC(F)$ around its boundary.}
\label{Fig:AB}
\end{figure}

For technical reasons that are apparent in the proof of Lemma~\ref{Lem:ENProps}, it is not enough to produce a nest only by applying $\mathcal A$ and $\mathcal B$ starting with some puzzle piece $\bI$. In order to construct an effective nest that goes fast enough in depth, we have to use several times the following \emph{smallest successor construction}.

Given a puzzle piece $P \ni c$, by a \emph{successor} of $P$ we mean a puzzle piece containing $c$ of the form $\Comp_c \DomL(Q)$, where $Q$ is a child of $\Comp_{c'} \DomL (P)$ for some $c' \in \Crit(F)$. If $c$ is persistently recurrent, then by definition each critical puzzle piece $\bI \ni c$ has the smallest (i.e.\ the deepest) successor, which we denote by $\Gamma(\bI)$. By construction, the corresponding return map $F^s \colon \Gamma(\bI) \to \bI$ has degree bounded only in terms of the local degrees of the critical points of $F$, and hence independent of $\bI$. 

\begin{definition}[The Enhanced Nest]
The \emph{Enhanced Nest} (around a persistently recurrent critical point $c$, or, equivalently, around a persistently recurrent critical fiber $\fib(c)$ of a persistently recurrent box mapping $F$) is a nested sequence $\bI_0 \supset \bI_1 \supset \bI_2 \ld$ of puzzle pieces such that 
\[
c \in \bI_0, \text{ and }\bI_{i+1} = \Gamma^{5b} \mathcal B \mathcal A (\bI_i) \text{ for every integer }i \ge 0, 
\]
where $b$ is the number of critical fibers intersecting $\ovl{\orb(c)}$.
\end{definition}

By the discussion above, $\bI_{n+1}$ is a pullback of $\bI_n$ and the return map 
\[
F^{p_n} \colon \bI_{n+1} \to \bI_n
\]
has degree bounded independently of $n$. Moreover, each element $\bI_n$ in the Enhanced Nest  comes with a pair of puzzle pieces $\bI^+_n, \bI^-_n$ nested as $\bI_n^- \subset \bI_n \subset \bI_n^+$ so that the annuli $\bI^+_n \sm \cl(\bI_n)$ and $\bI_n \sm \cl(\bI^-_n)$ are disjoint from $\PC(F)$.

This particular construction of the Enhanced Nest  is chosen because of the following lemma (\cite[Lemma~{8.3}]{KSS}):

\begin{lemma}[Transition and return times for the Enhanced Nest]
\label{Lem:ENProps}
Let $c \in \Crit(F)$ be a persistently recurrent critical point, and $(\bI_n)_{n \ge 0}$ is the Enhanced Nest  around $c$. If $p_n$ is the transition time from $\bI_{n+1}$ to $\bI_n$, and $r(\bI_{n+1})$ is the \emph{return time} of $\bI_{n+1}$ to itself, i.e.\ $r(\bI_{n+1})$ is the minimal positive integer $r$ such that there exists a point $x \in \bI_{n+1}$ with $F^r(x) \in \bI_{n+1}$, then
\begin{enumerate}
\item
\label{It:Return}
$3 r(\bI_{n+1}) \ge p_n$;
\item
\label{It:Tran}
$p_{n+1} \ge 2 p_n$. \qed
\end{enumerate}
\end{lemma}

This lemma is one of the main ingredients in the proof of the Key Lemma (Lemma~\ref{Lem:KeyLemma}). Let us comment on the meaning of conditions \eqref{It:Return} and \eqref{It:Tran} in Lemma~\ref{Lem:ENProps}.

Condition \eqref{It:Return} implies the following: if $F^s \colon A \to \bI_{n+1}$ is an arbitrary branch of the first return map to $\bI_{n+1}$ under $F$, then $s \ge r(\bI_{n+1}) \ge p_n /3$.

As for condition \eqref{It:Tran}, first observe that $(p_n)_{n \ge 0}$ is a monotonically increasing sequence of integers. Hence, by \eqref{It:Tran}, $p_{n+1} \ge 2p_n \ge p_n + p_{n-1}$. This means that this sequence grows faster than the sequence of Fibonacci numbers (with the same starting pair of values). For the classical Fibonacci numbers $f_{n}$, that is when $f_{n+1} = f_n + f_{n-1}$, $f_0=f_1=1$, the following property is well-known:
\[
f_{n+2} = f_1 + \sum_{i=0}^n f_i = 1 + \sum_{i=0}^n f_i.
\]
Along the same lines one can show that for our sequence $(p_i)_{i \ge 0}$ for every $n \ge 0$ we have
\[
p_{n+2} \ge p_1 + \sum_{i=0}^n p_i,
\]
or more generally, for every $n \ge m \ge 0$,
\begin{equation}
\label{Eq:Tran}
p_{n+2} \ge p_{m+1} + \sum_{i=m}^n p_i > \sum_{i=m}^n p_i.
\end{equation}

If we write $t := \sum_{i=m}^n p_i$, then $F^t(\bI_{n+1}) = \bI_m$ and the degree of the map $F^t \colon \bI_{n+1} \to \bI_m$ can be arbitrarily big. However, restricted to $\bI_{n+3}$ this map has degree bounded independently of $n$ and $m$. This follows from \eqref{Eq:Tran}: it takes additional $p_{n+2} - t > 0$ iterates of $F$ to map $F^t(\bI_{n+3})$ over $I_{n+2}$, and hence the degree of $F^t|_{\bI_{n+3}}$ is bounded above by the degree of $F^{p_{n+2}} \colon \bI_{n+3} \to \bI_{n+2}$, which is uniformly bounded (see Figure~\ref{Fig:Explanation}). This mechanism is in the core in the proof of Proposition~\ref{prop:KvSnice}, as it allows to apply the Covering Lemma, which we discuss next.

\begin{center}
\begin{figure}[h]
\includegraphics[width=1.\textwidth, trim=50 20 50 20, clip]{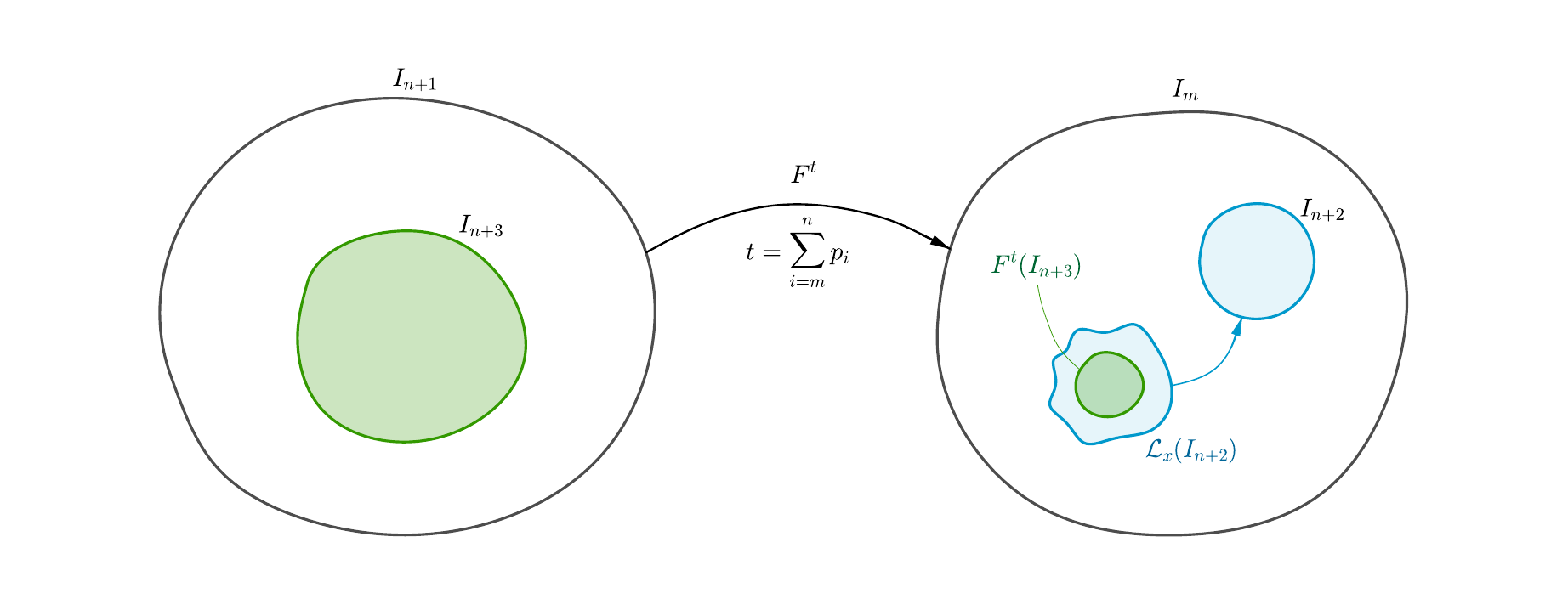}
\caption{The inequality for the transition times in the Enhanced Nest  implies that for every $n \ge m \ge 0$, the puzzle piece $F^t(\bI_{n+3})$, $t = \sum_{i=m}^n p_i$ is contained in some entry domain $\DomE_x(\bI_{n+2})$ to $\bI_{n+2}$ under $F$, and the degree of the restriction $F^t|_{\bI_{n+3}}$ is bounded independently of $n$ and $m$, while the degree of $F^t \colon \bI_{n+1} \to \bI_{m}$ does not.}
\label{Fig:Explanation}
\end{figure}
\end{center}


\subsection{The Covering Lemma} 

The second main ingredient that go into the proof of the Key Lemma (Lemma~\ref{Lem:KeyLemma}) is the Kahn--Lyubich Covering Lemma, which we discuss is this subsection.

Let $A \subset A' \subset U$, $B \subset B' \subset V$ be open topological disks in $\C$. Denote by $f \colon (A, A', U) \to (B, B', V)$ a holomorphic branched covering map $f \colon U \to V$ that maps $A'$ onto $B'$ and $A$ onto $B$.

One of the standard results on distortion of annuli under branched coverings is the following lemma; its proof is based on the Koebe Distortion Theorem.

\begin{lemma}[Standard distortion of annuli]
\label{Lem:StdControl}
Let $A \subset U$ and $B \subset V$ be open topological disks, and $f \colon (A, U) \to (B, V)$ be a holomorphic branched covering map. Suppose that the degree of $f$ is bounded by $D>0$. Then
\[
\modulus (U \sm \ovl A) \ge D^{-1} \cdot \modulus(V \sm \ovl B). \qed
\]
\end{lemma}

In this lemma, once the degree $D$ starts to grow, we lose control on the modulus of the pullback annulus. However, in the situation when the annulus $U \sm \ovl A$ is \emph{nearly degenerate} some portion of the modulus under pullback can still be recovered, provided that the restriction of $f$ to some smaller disk inside $U$ is of much smaller degree compared to $D$. This situation is often the case in complex dynamics, for example, when one considers a pair $P \subset Q$ of critical puzzle pieces such that $F^s(P) = Q$ for some large $s$: the degree of $F^s|_{P}$ can be arbitrary high, but restricted to some small neighborhood of the critical point in $P$ the degree is much smaller. This was made precise by Kahn and Lyubich~\cite{KL09} in the following lemma:

\begin{lemma}[Covering Lemma]
\label{Lem:CoveringLemma}
For any $\eta \ge 0$ and $D > 0$ there is $\epsilon = \epsilon (\eta, D) > 0$ such that the following holds: 

Let $A \subset A' \subset U$ and $B \subset B' \subset V$ be a triple of nested and open topological disks, and let $f \colon (A, A', U) \to (B, B', V)$ be a holomorphic branched covering map. Suppose the degree of $f$ is bounded by $D$, and the degree of $f|_{A'}$ is bounded by $d$. Then
\[
\modulus(U \sm \ovl A) \ge \min(\epsilon, \eta^{-1}\cdot \modulus(B'\sm \ovl B), C \eta d^{-2} \cdot \modulus(V \sm \ovl B)),
\]
where $C > 0$ is some universal constant.\qed
\end{lemma}

Let us make some comments on the statement of the Covering Lemma. The parameter $\epsilon$ in this lemma controls the degeneration of the annulus $U \sm \ovl A$: either $U \sm \ovl A$ has modulus at least $\epsilon$, or $\modulus(U \sm \ovl A) < \epsilon$ and we are in a \emph{nearly degenerate regime}. Note that the extend to which we have to degenerate to ``enter'' this regime depends on $\eta$ and the large degree $D$. Once we are in the nearly degenerate regime, then the Covering Lemma says that under the pullback we either recover at least $\eta^{-1}$-portion of the modulus of $B'\sm \ovl B$ (called the \emph{collar} in the terminology of \cite{KL09}), or $\modulus(U \sm \ovl A)$ is even smaller than $\eta^{-1} \cdot \modulus(B'\sm \ovl B)$ and in this even more degenerate regime we recover at least $C \eta d^{-2}$-portion of $\modulus(V \sm \ovl B)$.

\begin{figure}[h]
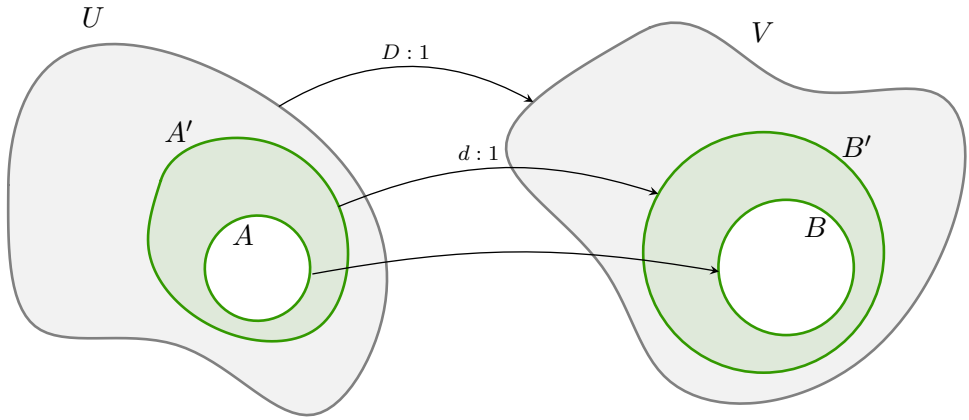

\begin{center}
\definecolor{grey1}{rgb}{0.5,0.5,0.5}
\definecolor{grey2}{rgb}{0.25,0.25,0.25}

\end{center}
\caption{The annuli in the Covering Lemma.}
\end{figure}

In the real-symmetric case, the Covering Lemma was proven in a sharper form in \cite[Lemma~9.1]{KvS}.

\section{The Spreading Principle and the QC-Criterion}
\label{Sec:Ing2}

In this section we discuss two quite general tools that are used to prove qc rigidity: the Spreading Principle and the QC-Criterion. We start with some auxiliary results on quasiconformal mappings and their gluing properties.

\subsection{Quasiconformal gluing}
\label{SSec-QC-gluing}

The following lemma is well-known (see for example \cite[Lemma 2]{PolyLike}), and is due to Rickman \cite{Ri} (sometimes it is also called Bers' Sewing Lemma and attributed to Bers).

\begin{lemma}[Sewing Lemma]
\label{Lem:Sewing}
Let $U \subset \C$ be an open topological disk and $E \subset U$ be a relatively compact set. Let $\phi_i \colon U \to \phi_i(U)$, $i \in \{1,2\}$  be a pair of homeomorphisms such that 
\begin{enumerate}
\item
$\phi_1$ is $K_1$-quasiconformal,
\item
$\phi_2|_{U \sm E}$ is $K_2$-quasiconformal,
\item
$\phi_1|_{\partial E} = \phi_2|_{\partial E}$.
\end{enumerate}
Then the map $\Phi \colon U \to \Phi(U)$ defined as
\begin{equation*}
\Phi(z) = \left\{
\begin{aligned}
&\phi_1(z) \quad\text{ if }z \in E\\
&\phi_2(z) \quad\text{ if }z \in U \sm E\\ 
\end{aligned}
\right.
\end{equation*}
is a $K$-quasiconformal homeomorphism with $K = \max(K_1, K_2)$. Moreover, $\ovl \partial \Phi = \ovl \partial \phi_1$ almost everywhere on $E$.\qed
\end{lemma}

We will need the following corollary to the Sewing Lemma, which is immediate if in the statement below the number of disks is finite; if this number is not finite, then the argument is slightly more delicate (see also the remark below).

\begin{lemma}[Countable gluing lemma]
\label{Lem:Gluing}
Let $V \subset \C$ be an open topological disk and $\psi \colon V \to \psi(V)$ be a $K$-quasiconformal homeomorphism. Let $B_i \subset V$, $i \in I$ be an at most countable collection of closed Jordan disks with pairwise disjoint interiors. Suppose there exist homeomorphisms $\phi_i \colon B_i \to \phi_i(B_i)$ that are uniformly $K'$-quasiconformal in the interior of $B_i$ and match on the boundary with $\psi$, i.e.\ for each $i \in I$,
\[
\psi|_{\partial B_i} = \phi_i|_{\partial B_i}.
\]
Then the map $\Phi \colon V \to \Phi(V)$ defined as
\begin{equation*}
\Phi(z) = \left\{
\begin{aligned}
&\phi_i(z) &\text{if } &z \in B_i\\
&\psi(z) &\text{if } &z \in V \sm \cup_{i \in I} B_i\\ 
\end{aligned}
\right.
\end{equation*}
is a $\max(K, K')$-quasiconformal homeomorphism.   
\end{lemma}

\begin{remark}
In general, if $I$ is infinite and one assumes that $\psi$ and all $\phi_i$ are just homeomorphisms, then the map $\Phi$ need not be even a homeomorphism. Lemma~\ref{Lem:Gluing} allows us in the statement of qc rigidity for box mappings to avoid the assumption  that the domains of $\U$ shrink to points, which was imposed in \cite{Lyu} and \cite{ALdM}. Under that assumption in those papers \cite[Lemma 11.1]{Lyu} was used instead of the Gluing Lemma above.
\end{remark}

\begin{proof}[Proof of Lemma~\ref{Lem:Gluing}]
Exhaust $I$ by finite sets $I_n$. For each $n$, define
\begin{equation*}
\Phi_n(z) := \left\{
\begin{aligned}
&\phi_i(z) &\text{if } &z \in B_i \text{ for some }i \in I_n\\
&\psi(z) &\text{if } &z \in V \sm \cup_{i \in I_n} B_i\\ 
\end{aligned}
\right.
\end{equation*}

Since $I_n$ is finite and each $\phi_i$, $i \in I_n$ matches with $\psi$ on the boundary of $B_i$, the map $\Phi_n$ is a homeomorphism of $V$ onto its image. By the Sewing Lemma (Lemma~\ref{Lem:Sewing}) applied to $E = \cup_{i \in I_n} \partial B_i$, $\phi_1 =\psi$, and $\phi_2 = \Phi_n$, it follows that 
\begin{equation*}
z \mapsto \left\{
\begin{aligned}
&\psi(z) &\text{if } &z \in \cup_{i \in I_n} \partial B_i\\
&\Phi_n(z) &\text{if } &z \in V \sm \cup_{i \in I_n} \partial B_i\\ 
\end{aligned}
\right.
\end{equation*}
is a $\max(K,K')$-quasiconformal homeomorphism. The latter map is clearly equal to $\Phi_n$.

In this way, we obtained a sequence $(\Phi_n)_{n \ge 0}$ of $\max(K,K')$-quasiconformal maps on $V$. By compactness of quasiconformal maps of uniformly bounded dilatation, this sequence has a sub-sequential limit. Moreover, since the maps in the sequence agree on more and more disks $B_i$, all sub-sequential limits are the same and equal to the map $\Phi$. The conclusion of the lemma follows.    
\end{proof}

\subsection{Lifts of combinatorial equivalence and the boundary marking}

Let $F\colon\U\to\V$ and $\tilde F\colon\tilde\U\to\tilde\V$ be a pair of complex box mappings that are combinatorially equivalent in the sense of Definition~\ref{DefA:CombEquivBoxMappings} w.r.t.\ to some \emph{$K$-quasiconformal} homeomorphism $H \colon \V \to \tilde \V$ such that $H(\U) = \tilde{\U}$ and $\tilde F \circ H = H \circ F$ on $\partial U$ for each component $U$ of $\U$. By Remark~\ref{Rem:ext} after Theorem~\ref{Thm:BoxMappingsMain2}, the latter condition means that if $\hat H \colon \cl \V \to \cl \tilde \V$ is a continuous homeomorphic extension of $H$ to the closures, then 
\[
\hat{\tilde F} \circ \hat H|_{\partial U} = \hat H \circ \hat F|_{\partial U}
\]
for each component $U$ of $\U$; here $\hat F$ and $\hat{\tilde F}$ are the corresponding homeomorphic extensions of $F$ and $\tilde F$ to the closures. Let $P$ and $\tilde P$ be a pair of corresponding puzzle pieces for $F$ and $\tilde F$ in the sense of Section~\ref{SSec:CombEquiv}.

\begin{definition}[Boundary marking]
\label{Def:BdMark}
We say that a homeomorphism $h \colon P \to \tilde P$ \emph{respects the boundary marking (induced by $H$)} if $h$ extends to a homeomorphism $\hat h$ between $\cl P$ and $\cl \tilde P$ and $\hat h|_{\partial P}$ agrees with the corresponding lift of $\hat H$; that is, if $k \ge 0$ is the depth of $P$, then 
\[
\hat{\tilde{F}}{}^k \circ \hat h|_{\partial P} = \hat H \circ \hat F^k|_{\partial P}.
\]
\end{definition}

Let $K$ be the quasiconformal dilatation of $H$. The definition of combinatorial equivalence then gives us a $K$-qc homeomorphism that respects the boundary marking between any two corresponding puzzle pieces of depths $0$ and $1$. The next lemma is a simple observation that shows that such $H$ gives rise to a $K'$-qc homeomorphism that respects the boundary marking for a pair of corresponding arbitrary deep puzzle pieces; of course, $K' \ge K$ can be arbitrarily large.

\begin{lemma}[Starting qc maps respecting boundary]
\label{Lem:StartingQC}
Let $P$, $\tilde P$ be a pair of corresponding puzzle pieces of depth $k \ge 0$ for $F$, $\tilde F$. Then there exists a $K'$-quasiconformal homeomorphism $h \colon P \to \tilde P$ that respects the boundary marking induced by $H$. 
\end{lemma}

\begin{proof}
Let $V$ be the component of $\V$ so that $F^k(P) = V$. If $F^k \colon P \to V$ is a conformal isomorphism, then, by combinatorial equivalence (Definition~\ref{DefA:CombEquivBoxMappings}), $\tilde F^k \colon \tilde P \to \tilde V$ is also a conformal isomorphism. In this situation, we can define $h := \tilde F^{-k}\circ H \circ F^k$ for the branch $\tilde F^{-k}$ that maps $\tilde V$ over $\tilde P$. Clearly, $h$ is a $K$-qc map that respects the boundary marking.

Suppose $v_1, \ldots, v_{\ell} \in V$ are the critical values of the map $F^k \colon P \to V$. By combinatorial equivalence, the map $\tilde F^k \colon \tilde P \to \tilde V$ must have the same number of critical values with the same branching properties (i.e.\ if $c \in P$ is a critical point of local degree $s \ge 2$ and so that $F^k(c) = v_i$, then the corresponding critical point $\tilde c$ must have local degree $s$ and is mapped by $\tilde F^k$ to the corresponding critical value $\tilde v_i$).

Choose an arbitrary smooth diffeomorphism $L \colon V \to V$ so that $L(v_i) = H^{-1}(\tilde v_i)$ and so that $L = \text{id}$ in some small neighborhood of $\partial V$. Then $H \circ L \colon V \to \tilde V$ is a $K'$-qc homeomorphism that maps the critical values of $F^k|_P$ to the critical values of $\tilde F^k|_{\tilde P}$ and respects the boundary marking. Such homeomorphism can be now lifted as above: define $h = \tilde F^{-k} \circ H \circ L \circ F^k$ (for appropriate choices of inverse branches). By construction, $h$ is $K'$-qc homeomorphism between $P$ and $\tilde P$ that respects the boundary marking, as required.
\end{proof}

\subsection{The Spreading Principle}
The Spreading Principle is a dynamical tool which makes it possible to construct quasiconformal conjugacies outside of some puzzle neighborhood $\Y$ of $\Crit(F)$ between two mappings $F\colon\U\to\V$ and $\tilde F\colon\tilde\U\to\tilde\V$ provided that one has an initial quasiconformal mapping $H\colon \V\to \tilde\V$ which is a combinatorial equivalence and a conjugacy on $\partial \U$ between $F$ and $\tilde F$, and a quasiconformal mapping on each component $P$ of $\Y$ to the corresponding component $\tilde P$ of $\tilde \Y$ that respects the boundary marking (in the sense of Definition~\ref{Def:BdMark}).

In the context of polynomials, the Spreading Principle was proven in \cite[Section 5.3, page 769]{KSS} and \cite{KvS}. Below we state and prove the corresponding version for dynamically natural box mappings. Our statement is slightly more general than the one in \cite{KSS, KvS} as it allows to take into account some pre-existing conjugacy on some part of the postcritical set. 

\begin{theorem}[Spreading Principle]
\label{Thm:Spreading}
\label{Thm:Spreading_bis}
Let $F \colon \U \to \V$ and $\tilde{F} \colon \tilde{\U}\to\tilde{\V}$ be a pair of dynamically natural complex box mappings and $H \colon \V \to \tilde \V$ be a $K$-quasiconformal homeomorphism that
\begin{itemize}
\item
provides a combinatorial equivalence between $F$ and $\tilde F$, is a conjugacy on the boundary of $\U$, and moreover, it
\item
conjugates $F$ to $\tilde F$ on $X_S := \{F^i(\fib(c)) \colon c \in S, i \ge 0\}$ for some subset $S \subset \Crit(F)$ (possibly empty).
\end{itemize}

Let $\Y$ be a nice puzzle neighborhood of $\Crit(F) \sm S$ such that $X_S \cap \Y = \emptyset$, and let $\tilde{\Y}$ be the corresponding neighborhood of $\Crit(\tilde F) \sm \tilde S$ for $\tilde{F}$. Further suppose that there exists a $K'$-quasiconformal homeomorphism $\phi \colon \Y \to \tilde \Y$ that respects the boundary marking induced by $H$. Then $\phi$ extends to a $\max(K,K')$-quasiconformal homeomorphism $\Phi \colon \V \to \tilde{\V}$ such that:
\begin{enumerate}
\item 
\label{It:SP1}
$\Phi=\phi$ on $\Y$;
\item 
\label{It:SP2}
for each $z\in \U \sm \Y$, $$\tilde{F}\circ\Phi(z)=\Phi\circ F(z);$$
\item 
\label{It:SP3}
$\Phi|_{\V \sm \DomL(\Y)}$ is $K$-quasiconformal;
\item 
\label{It:SP4}
$\Phi(P)=\tilde{P}$ for every puzzle piece $P$ that is not contained in $\DomL(\mathcal Y)$,
					and $\Phi \colon P\rightarrow\tilde{P}$ respects the boundary marking induced by $H$.
\end{enumerate}
\end{theorem}
\begin{proof}  
Denote by $\Pp_n$ the collection of puzzle pieces of depth $n \ge 0$ for $F$.

Since $X_S \cap \Y = \emptyset$ and $X_S$ is forward invariant under $F$, it follows that $X_S \cap \DomL(\Y) = \emptyset$; the same is true for the corresponding sets $\tilde X_{\tilde S}$ and $\tilde \Y$. Moreover, there exists a puzzle neighborhood $\mathcal X$ of $S$ consisting of puzzle pieces of the same depth, say $m$, with the property that $m$ is larger than the depths of the components of $\Y$ and $\mathcal X \cap \Y = \emptyset$. Note that $\Y \cup \X$ is a nice set by construction.

Let $\mathcal C$ be a collection of all critical puzzle pieces that are not contained in $\Y$ but do intersect $\Crit(F) \sm S$. This is a finite set. Using Lemma~\ref{Lem:StartingQC}, for each $C \in \mathcal C$ let us pick a quasiconformal map $h_C \colon C \to \tilde C$ that respects the boundary marking induced by $H$; suppose $K_0$ is the maximum over the quasiconformal dilatations of $h_C$ for $C \in \mathcal C$.

We now want to construct a quasiconformal homeomorphism $\psi \colon \X \to \tilde \X$ that respects the boundary marking. Let $Q$ be a component of $\X$, and $k \ge 0$ be minimal so that $F^k(Q)$ either is a component of $\V$, or is in $\mathcal C$. Now define a map $\psi_Q \colon Q \to \tilde Q$ by the formula $\tilde F^k \circ \psi_Q = f \circ F^k$, where 
\begin{equation}
\label{Eq:f}
f = \left\{
\begin{aligned}
&H &\text{ if } &F^k(Q) \text{ is a component of } \V,\\
&h_{F^k(Q)} &\text{ if } &F^k(Q) \in \mathcal C.\\
\end{aligned}
\right.
\end{equation}

Since none of the puzzle pieces in the sequence $Q, F(Q), \ldots, F^k(Q)$ intersects $\Y$ (as $X_S \cap \Y = \emptyset$), and $F$ is conjugate to $\tilde F$ on $X_S$, we conclude that the map $\psi_Q$ is a well-defined quasiconformal homeomorphism with dilatation $\max(K,K_0)$. By construction, $\psi_Q$ respects the boundary marking induced by $H$. We define $\psi \colon \X \to \tilde \X$ component-wise using the maps $\psi_Q$. Note that the dilatation of $\psi$ does not depend on the choice of the puzzle neighborhood $\X$, and as the depth of the neighborhood $\X$ tends to $\infty$, the set where the dilatation of $\psi$ is equal to $K_0$ shrinks to zero (in measure).    

Finally, let us adjust $\phi$, $\Y$ and $\Y'$ as follows. Set $\Y' := \Y \cup \X$ and $\tilde \Y' := \tilde \Y \cup \tilde \X$, and define
\begin{equation*}
\phi'(z) = \left\{
\begin{aligned}
&\phi(z) &\text{ if } &z \in \Y,\\
&\psi(z) &\text{ if } &z \in \X.\\
\end{aligned}
\right.
\end{equation*}
Defined this way, $\phi' \colon \Y' \to \tilde \Y'$ is a $K''$-quasiconformal homeomorphism between nice puzzle neighborhoods of $\Crit(F)$ and $\Crit(\tilde F)$, with $K'' := \max(K,K', K_0)$, and this homeomorphism respects the boundary marking.

Let $Q$ be a component of $\DomL(\Y')$, and $k \ge 0$ be the landing time of the orbit of $Q$ to $\Y'$. Since $\Crit(F) \subset \Y'$, the map $F^k \colon Q \to F^k(Q)$ is a conformal isomorphism. Therefore, we can pull back $\phi'$ and define a $K''$-quasiconformal homeomorphism $\phi'_Q \colon Q \to \tilde Q$ by the formula $\tilde F^k \circ \phi'_Q = \phi' \circ F^k$. Since $\phi'$ respects the boundary marking induced by $H$, so does $\phi'_Q$.

Let us now inductively define a nested sequence of sets $Y_0 \supset Y_1 \supset Y_2 \ldots$ so that:
\begin{itemize}
\item
$Y_0$ is the union of all puzzle pieces in $\Pp_0$;
\item
$Y_{n+1}$ is the subset of $Y_n$ consisting of puzzle pieces of depth $n+1$ that are not components of $\DomL(\Y')$.  
\end{itemize}

For each puzzle piece $Q \in \Pp_n$ that is contained in $Y_n$, it follows that none of the pieces in the orbit $Q, F(Q), \ldots$ can lie in $\Y'$ (because otherwise $Q$ would lie in $Y_{\ell-1} \sm Y_\ell$ for some $\ell \le n$). Therefore, there exists the minimal $k = k(Q) \ge 0$ so that $F^k(Q)$ either is in $\mathcal C$, or is a component of $\V$. By a similar reasoning as above, in both cases the map $F^k \colon Q \to F^k(Q)$ allows to pull back homeomorphisms consistently and to define a homeomorphism $H_Q \colon Q \to \tilde Q$ by the formula $\tilde F^k \circ H_Q = f \circ F^k$, where $f$ is as in \eqref{Eq:f}. Defined this way, the map $H_Q$ is $\max(K, K_0)$-quasiconformal map that respects the boundary marking induced by $H$.

We proceed by inductively defining a sequence of homeomorphisms $(\Phi^{\X}_n)_{n \ge 0}$, $\Phi^{\X}_n \colon \V \to \tilde \V$ as follows. Set $\Phi^{\X}_0 = H$, and for each $n \ge 0$ define
\begin{equation*}
\Phi^{\X}_{n+1}(z) = \left\{
\begin{aligned}
& \Phi^{\X}_n(z) \quad\text{ if }z \in \V \sm \Pp_{n+1},\\
&H_Q(z) \quad \text{ if }z \in Q \in \Pp_{n+1} \text{ and }Q \subset Y_{n+1},\\
&\phi'_Q(z) \quad \text{ if }z \in Q \in \Pp_{n+1}\text{ and }Q \not\subset Y_{n+1}.
\end{aligned}
\right.
\end{equation*}

By the Gluing Lemma (Lemma~\ref{Lem:Gluing}), the map $\Phi^{\X}_n \colon \V \to \tilde \V$ is a $K''$-quasiconformal homeomorphism for each $n \ge 0$. Note that the sequence $(\Phi^{\X}_n)$ eventually stabilizes on $\V \sm \bigcap_n Y_n$. The set $E(\Y') := \bigcap_n Y_n$ consists of points in $\V$ which are never mapped into $\Y'$, and hence $E(\Y') \subset \Koc(F)$. Since $F$ is dynamically natural, $E(\Y')$ has zero Lebesgue measure. Hence $\V \sm E(\Y')$ is a dense set. We conclude that the sequence $(\Phi^{\X}_n)$ converges the limiting $K''$-quasiconformal homeomorphism $\Phi^{\X} \colon \V \to \tilde \V$. This map depends on the depth of $\X$. Letting this depth go to $\infty$, in the limit we obtain the required quasiconformal map $\Phi \colon \V \to \tilde \V$ with quasiconformal dilatation equal to $\max(K,K')$. For this homeomorphism the properties \eqref{It:SP1}, \eqref{It:SP2} and \eqref{It:SP4} follow directly from the construction, while property \eqref{It:SP3} follows from the facts that $\area(E(\Y')) = 0$ and that $H$ conjugates $F$ and $\tilde F$ on $\bigcup_{c \in S} \fib(c)$. 
\end{proof}

\subsection{The QC-Criterion} 

The following criterion for the existence of quasiconformal extensions was proven in~\cite[Appendix~1]{KSS}. This result is inspired by the works of Heinonen--Koskela~\cite{HK} and Smania~\cite{Sm}.

Recall that a topological disk $P \subset \C$ has \emph{$\eta$-bounded geometry} if there is a point $x \in  P$ such that $P$ contains an open round disk of radius $\eta \cdot \diam(P)$ centered at $x$.

\begin{theorem}[QC-Criterion] 
\label{Thm:QCCriterion}
For any constants $0 \le  k < 1$ and $\eta  > 0$ there exists a constant K with the following property. Let $\phi \colon \Omega \to  \tilde \Omega$ be 
a quasiconformal homeomorphism between two Jordan domains. Let $X$ be a subset of $\Omega$ consisting of pairwise disjoint open topological disks. Assume that the following hold:
\begin{enumerate}
\item  If P is a component of $X$, then both of $P$ and $\phi(P)$ have $\eta$-bounded geometry, and moreover
\[
\modulus (\Omega \sm \cl(P))\ge \eta, \quad \modulus(\tilde \Omega \sm \cl(\phi(P)))\ge \eta;
\]
\item $|\bar \partial \phi | \le  k|\partial \phi|$ holds almost everywhere on $\Omega \sm X$. 
\end{enumerate}
Then there exists a $K$-quasiconformal map $\psi \colon \Omega\to \tilde \Omega$ such that 
$\psi=\phi$ on $\partial \Omega$.  \qed
\end{theorem}

The QC-Criterion will be applied in situations where $\Omega$ is a puzzle piece
and $X$ is a union of first return domains to $\Omega$. Note that \emph{a priori} the set $X$ can be pretty ``wild'', see Figure~\ref{Fig:X}.


\begin{figure}[h]
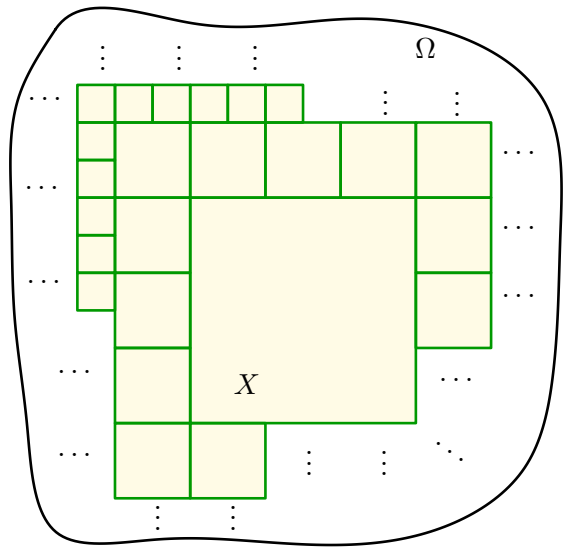

\begin{center}
\definecolor{qqzzqq}{rgb}{0.,0.6,0.}
\definecolor{ffdxqq}{rgb}{1.,0.8431372549019608,0.}

\end{center}
\caption{The set $X$ in the statement of the QC-Criterion can have a complicated shape. In particular, it is possible that $X$ tiles $\Omega$, as shown on the picture (components of $X$ are shown in yellow). In this example, the second condition in the QC-Criterion is automatically satisfied: the set $\Omega \sm X$, shown in green, has zero area.}
\label{Fig:X}
\end{figure}


\section{Sketch of the proof of the QC-Rigidity Theorem} 
\label{Sec:Sketch}
In this section we will explain how to combine shrinking of puzzle pieces (Theorem~\ref{Thm:BoxMappingsMain2} \eqref{It:MainIt1}), complex bounds (Theorem~\ref{Prop:UC}), the QC-Criterion (Theorem~\ref{Thm:QCCriterion}) and the Spreading Principle (Theorem~\ref{Thm:Spreading}) to prove qc rigidity for non-renormalizable dynamically natural complex box mappings. The argument follows that of \cite[Section 6]{KSS}.

Complex bounds from Theorem~\ref{Prop:UC} imply the following lemma, which gives us that for any point $z$, the orbit of which accumulates on a critical fiber, there are arbitrarily deep puzzle pieces containing $z$ that satisfy the geometric hypotheses of the QC-Criterion. The details are verbatim as in~\cite[Corollary 6.3]{KSS} (simplified in several places as we do not have to deal with (eventually) renormalizable critical points).

\begin{lemma}\cite[Corollary 6.3 (Geometric control for landing domains)]{KSS}
\label{Cor:UCcor}
Let $F$ be a non-renormalizable dynamically natural complex box mapping. Then for every integer $n \ge 0$ there exists a nice puzzle neighborhood $W$ of $\CritP(F)$ such that for every component $U$ of $\DomE(W)$ the following hold:
\begin{itemize}
\item
$U$ has $\eta$-bounded geometry;
\item
the depth of $U$ is larger than $n$, and if $P$ is the puzzle piece of depth $n$ containing $U$, then $\modulus(P \sm \ovl U) \ge \eta$, 
\end{itemize}
where $\eta>0$ is a constant independent of $n$.

Furthermore, if $\tilde F$ is another non-renormalizable dynamically natural box mapping combinatorially equivalent to $F$, then taking $\eta$ smaller if necessary, the same claims as above hold for $\tilde F$ and the corresponding objects with tilde.\qed 
\end{lemma}

The following lemma implies that we can modify a combinatorial equivalence $H$ from Theorem~\ref{Thm:BoxMappingsMain2}~\eqref{It:MainIt3} so that it becomes a conjugacy on the orbits of critical points in $\Crit(F) \sm \CritP(F)$ (if this set is not empty).

\begin{lemma}[Making $H$ partial conjugacy] 
\label{Lem:ModH}
Let $F \colon \U \to \V$ and $\tilde F \colon \tilde \U \to \tilde \V$ be non-renormalizable dynamically natural box mappings that are combinatorially equivalent with respect to some $K$-quasiconformal homeomorphism $H \colon \V \to \tilde \V$ that satisfies the assumptions of Theorem~\ref{Thm:BoxMappingsMain2}~\eqref{It:MainIt3}. Suppose $\Crit(F) \sm \CritP(F) \neq \emptyset$. Then there exists a $K'$-quasiconformal homeomorphism $H' \colon \V \to \tilde \V$ so that
\begin{itemize}
\item
$H'|_{\partial \U} = H|_{\partial \U}$;
\item
$H'$ is a combinatorial equivalence between $F$ and $\tilde F$;
\item
$\tilde F \circ H' = H' \circ F$ on $\partial \U$;
\item
$H'$ conjugates $F$ to $\tilde F$ on $\left\{F^i(c) \colon c \in \Crit(F) \sm \CritP(F), i \ge0\right\}$.
\end{itemize}
\end{lemma}

\begin{proof} 
Recall that $S:= \Crit(F) \sm \CritP(F)$ consists of non-recurrent critical points whose orbits do not accumulate on other critical points (we can talk about points rather than fibers since by Theorem~\ref{Prop:UC}~\eqref{It:Shrink} all critical fibers are trivial). Therefore, we can find a sufficiently deep puzzle neighborhood $\U_0$ of $\Crit(F)$ consisting of puzzle pieces of the same depth, say $n_0 \ge 0$, so that the orbit of $F(S)$ is disjoint from $\U_0$, and hence from $\DomL(\U_0)$. For simplicity we assume that $F(S)$ contains no critical points; otherwise, the modification of the argument below is straightforward. By increasing $n_0$ we can also assume that each component of $\U_0$ contains only one critical point.

Let $\phi \colon \U_0 \to \tilde \U_0$ be a $K'$-quasiconformal map that respects the boundary marking induced by $H$ constructed by Lemma~\ref{Lem:StartingQC}. Clearly, $K' \ge K$ and can be arbitrarily large. Let $H_0 \colon \V \to \tilde \V$ be the $K'$-qc extension of $\phi$ by the Spreading Principle (Theorem~\ref{Thm:Spreading}). If $P$ is a sufficiently small puzzle piece around a point $x \in \orb(F(S))$, then $P$ is not contained in $\DomL(\U_0)$ by our choices, and thus by property \eqref{It:SP4} of the Spreading Principle, $H_0(P) = \tilde P$. Since puzzle pieces shrink to points (Theorem~\ref{Thm:BoxMappingsMain2} \eqref{It:MainIt1}), by letting the depth of such $P$'s go to zero we conclude that $H_0(F^k(c))=\tilde F^k(\tilde c)$ for every $c \in S$ and $k\ge 1$.  

Now define $H' \colon \V \to \tilde \V$ by $H'=\tilde F^{-1} \circ H_0 \circ F$ on the components of $\U_0$ containing the points of $S$ and $H'=H_0$ elsewhere. As $H_0(F^k(c))=\tilde F^k(\tilde c)$  for all $k\ge 1$ and $c \in S$, the map
 $H'$ is well-defined and it is continuous because $\tilde F \circ H_0 = H_0 \circ F$ on $\partial \U_0$. It follows that $H'$ is the required $K'$-qc map.
\end{proof}

Lemmas~\ref{Lem:ModH}, \ref{Cor:UCcor} and the QC-Criterion (Theorem~\ref{Thm:QCCriterion}) imply the following proposition (which is similar to the claim in \cite[Section~6.4]{KSS}). The conclusion of this proposition will then be an input into the Spreading Principle later in the proof of qc rigidity of complex box mappings.

\begin{proposition}[Uniform control of dilatation]
\label{Prop:KeyProp}
Let $F$ and $\tilde F$ be non-renormalizable dynamically natural box mappings that are combinatorially equivalent with respect to some homeomorphism $H$ that satisfies the assumptions of Theorem~\ref{Thm:BoxMappingsMain2}~\eqref{It:MainIt3}. 

Then there exists $K > 0$ such that for every critical point $c \in \Crit(F)$ and the corresponding critical point $\tilde c \in \Crit(\tilde F)$, and every $n \ge 0$ there exists a $K$-quasiconformal homeomorphism respecting the boundary marking between the corresponding critical puzzle pieces of depth $n$ containing $c$ and $\tilde c$.
\end{proposition} 

\begin{proof}
Suppose $F$ and $\tilde F$ are combinatorially equivalent with respect to a $k$-quasiconformal homeomorphism $H$. By Lemma~\ref{Lem:ModH}, we can assume that $H$ is a conjugacy between $F$ and $\tilde F$ on the orbits of points in $\Crit(F) \sm \CritP(F)$.

Let $Q$, $\tilde Q$ be the pair of corresponding puzzle pieces of depth $n$ containing $c$, respectively $\tilde c$. If $c \in \Crit(F) \sm \CritP(F)$, then the required homeomorphism between $Q$ and $\tilde Q$ is obtained just by pulling back $H$. So assume that $c \in \CritP(F)$.
 
For the chosen $n$, let $W$ and $\tilde W$ be a pair of neighborhoods of $\CritP(F)$ and $\CritP(\tilde F)$ given by Lemma~\ref{Cor:UCcor}, and let $\phi_0 \colon W \to \tilde W$ be a $K_0$-quasiconformal map that respects the boundary marking given by Lemma~\ref{Lem:StartingQC}. Note that $K_0$ can be arbitrarily large. 
We can spread this map around using the Spreading Principle (Theorem~\ref{Thm:Spreading_bis}). Restricting the resulting map to the depth $n$ puzzle piece $Q$ we obtain a $\max(k, K_0)$-quasiconformal map $\phi \colon Q \to \tilde Q$ that respects the boundary marking. 

Let $X$ be the union of the landing domains to $W$ that are contained in $Q$. By Theorem~\ref{Thm:Spreading_bis}, the map $\phi$ has the property that its dilatation is bounded by $k$ on $Q \sm X$. On the other hand, each component $P$ of $X$ satisfies the conclusion of Lemma~\ref{Cor:UCcor}, that is, $P$ has $\eta$-bounded geometry and $\modulus(Q \sm \ovl P) \ge \eta$ for some constant $\eta$ independent of $n$. The same is true for $\tilde Q$ and $\tilde P$ (with the same constants). Therefore, we can apply the QC-Criterion to $\Omega:=P \supset X$, $\tilde \Omega := \tilde P \supset \tilde X$ and the map $\phi$. This criterion will give us a desired $K$-quasiconformal homeomorphism between $Q$ and $\tilde Q$ that agrees with $\hat \phi$ on $\partial Q$, and hence respects the boundary marking. The proof is complete once we note that by the QC-Criterion, the dilatation $K$ depends only on $k$ and $\eta$, but not on $n$, or on $K_0$. 

Further details of the proof can be found in~\cite[Page 787]{KSS}.
\end{proof}

Finally, we are ready to present the proof of qc rigidity for non-renormalizable box mappings.

\begin{proof}[Proof of Theorem~\ref{Thm:BoxMappingsMain2}~\eqref{It:MainIt3}]
Suppose that the homeomorphism $H$ that provides the combinatorial equivalence between $F$ and $\tilde F$ is $k$-quasiconformal. 
For every $n \ge 0$, define $W_n := \bigcup_{c \in \Crit(F)} P_n(c)$, where $P_n(c)$ is the puzzle piece of depth $n$ containing $c$, and similarly define $\tilde W_n$ for $\tilde F$. By Proposition~\ref{Prop:KeyProp}, we can construct a sequence of $K$-quasiconformal maps $h_n \colon W_n \to \tilde W_n$ so that each of them respects the boundary marking induced by $H$. Since $\Crit(F) \subset W_n$ and $h_n$ respects the boundary marking, we can spread $h_n$ around using Theorem~\ref{Thm:Spreading}. In this way we obtain a sequence of $\max(k, K)$-quasiconformal maps $H_n \colon \V \to \tilde \V$ each conjugating $F$ and $\tilde F$ outside of $W_n$. Finally, by Theorem~\ref{Prop:UC}~\eqref{It:Shrink}, the sets $W_n$ shrink to $\Crit(F)$ as $n \to \infty$. Therefore, by passing to a subsequence, $H_n$ converges to a quasiconformal map $H_\infty \colon \V \to \tilde \V$ that conjugates $F$ and $\tilde F$. Using the properties of $H_n$ from Theorem~\ref{Thm:Spreading}, it is not hard to see that $H_{\infty}$ does not depend on the subsequence. This concludes the proof. 
\end{proof}

\section{Sketch of the proof of Complex Bounds}
\label{Sec:UC}

As was discussed in Section~\ref{SSec:Rec}, there are three combinatorially distinct types of critical points: non-combinatorially recurrent, reluctantly recurrent and persistently recurrent. The proofs of complex bounds around each of these types of critical points are different from one another.
For reluctantly recurrent and non-recurrent critical points the claim of Theorem~\ref{Prop:UC} follows from the fact that around each of these points one can find arbitrarily deep puzzle pieces which map with uniformly bounded degree to some piece with bounded depth. For persistently recurrent critical points the proof will employ the Enhanced Nest construction of Section~\ref{SSec:Enhanced}. These arguments will be made more precise in Lemmas~\ref{Lem:RR}, \ref{Lem:NonR}, and \ref{Lem:PR} below, dealing with reluctantly, non-, and persistently recurrent cases respectively.

To obtain Theorem~\ref{Prop:UC} from these lemmas, which give, in particular, complex bounds for induced complex box mappings around recurrent critical points, a little care is needed, since critical fibers can accumulate on other critical fibers. However, this step is straightforward and is done verbatim as in the proof of \cite[Proposition 6.2]{KSS}, and hence will be omitted in our exposition. Our focus will be only on those lemmas.

\subsection{Puzzle control around non- and reluctantly recurrent critical points}

First we discuss puzzle pieces around a reluctantly recurrent point $c$.
 
 \begin{lemma}[Reluctantly recurrent case]
 \label{Lem:RR}
  Let $c$ be a reluctantly recurrent critical point. Then there exists a positive constant $\eta>0$
 so that for any $\epsilon>0$ there are nice puzzle neighborhoods $W'\supset W$ of $\Back(c)$ such that:
 \begin{enumerate}
 \item Each component of $W$ has $\eta$-bounded geometry.  
 \item For each $c'\in \Back(c)$, we have $\diam(W'_{c'})\le \epsilon$ and $\modulus(W'_{c'}\setminus \ovl W_{c'})\ge \eta$;
 here $W_{c'} = \Comp_{c'}W$ and $W'_{c'} = \Comp_{c'}W'$. 
 \item $F^k(\partial W_{c'})\cap W'_{c}=\emptyset$ for each $c'\in \Back(c)$ and each $k\ge 1$. 
 \end{enumerate}
 \end{lemma} 
\begin{proof}
Lemma 6.5 in \cite{KSS} uses the definition of 
  reluctant recurrence to show that one can map deep critical pieces
 around a reluctantly recurrent critical point to one of arbitrary fixed depth with  uniformly bounded degree. Then Lemma 6.6 in that paper shows that 
one obtains puzzle pieces with bounded geometry and moduli bounds around $c' \in \Back(c)$. The proofs use only condition~\eqref{It:Nat1} of Definition~\ref{Def:NaturalBoxMapping} (in the form of Lemma~\ref{Lem:NE}) and goes verbatim as in the proof of \cite[Lemma 6.6]{KSS}.
\end{proof} 

Moving on to non-recurrent critical points, we have the following statement.

 \begin{lemma}[Non-recurrent case]
 \label{Lem:NonR}
Let $c$ be a non-recurrent critical point. Then there is a constant $\eta>0$ and for every $\epsilon > 0$ there exists a puzzle piece $W \ni c$ such that $\diam(W) \le \epsilon$ and, if $c \in \CritP(F)$, such that $W$ has $\eta$-bounded geometry. 
 \end{lemma}

 \begin{proof}
 As in \cite[page 782]{KSS}, we only need to consider  the case when the orbit of $c$ does not intersect or accumulate on a reluctantly recurrent critical point. We then distinguish three cases, depending whether $\Forw(c)$ is 1) empty; 2) not empty, and contains only combinatorially non-recurrent critical points; and 3) not empty and contains a persistently recurrent critical point. (Case 4 in \cite{KSS} concerns renormalizable critical points which we do not have in our setting by hypothesis.) Note that case 1) corresponds to the situation when $c \not \in \CritP(F)$.
 
 In Case 1, choose $n_0$ to be the depth of puzzle pieces such that $P_{n_0}(F^k(c))$ contains no critical points for every $k \ge 1$ (here $P_k(x)$ stands for the puzzle piece of depth $k$ containing the point $x$). Since $F$ is dynamically natural, condition~\eqref{It:Nat3} of Definition~\ref{Def:NaturalBoxMapping} guarantees $K(F) = \Kwi(F)$. Therefore, there exists $\delta>0$ such that the orbit of $c$ visits infinitely many times the components of $\U$ that are $\delta$-well inside the respective components of $\V$. Hence there exist infinitely many $n \ge 1$ such that $F^{n-1} \colon P_{n_0+n-1}(F(c)) \to P_{n_0}(F^n(c))$ is a conformal map of uniformly bounded distortion; this distortion depends on $\delta$. The conclusion in case 1) follows. (Note that in this case we do not claim that $W$ has bounded geometry.)

Case 2: Since $\Forw(c)$ does not contain combinatorially recurrent critical points,
we can fix $n_0$ so that for each $c'\in \Forw(c)$ the forward
orbit of $c'$ is disjoint from $P_{n_0}(c')$
and so that if $F^k(c)\in P_{n_0}(c'')$ for some $c''\in \Crit(F)$ and $k\ge 0$, then 
$c''\in \Forw(c)$.  Choose $c'\in \Forw(c)$. 
Fix $n'>n\ge n_0$ so that $P_{n'}(c')\Subset P_n(c')$.
Let $n_i$ be a sequence of times so that $F^{n_i}(c)\in P_{n'}(c')$
and let $P_{n_i+n}(c)$ be the pullback of $P_{n}(c')$ under the map 
$F^{n_i}$. We claim that $F^{n_i}\colon P_{n_i+n}(c)\to P_{n}(c')$
has bounded degree. Indeed, if $c_1\in F^{k}(P_{n_i+n}(c))$ for some $0\le k<n_i$
and for some $c_1\in \Crit(F)$, then $c_1\in \Forw(c)$ and, moreover,  
since $F^{k}(P_{n_i+n}(c))\subset P_{n_0}(c_1)$ no further iterate
of $P_{n_i+n}(c)$ can contain $c_1$. It follows that
each critical point in $\Crit(F)$ can be visited at most once by 
$P_{n_i+n}(c),F(P_{n_i+n}(c)),\dots,F^{n_i}(P_{n_i+n}(c))=P_n(c')$. 
Note that  $P_{n'}(c')$ has some $\eta'$-bounded geometry (where $\eta'$ depends
of course on the integer $n'$ which we fixed once and for all). 
Since $P_{n'}(c')\Subset P_n(c')$ and the degree of $F^{n_i}\colon P_{n_i+n}(c)\to P_{n}(c')$
is bounded in terms of the number and degrees of the critical points of $F$, it follows that $P_{n_i+n}(c)$ has $\eta$-bounded geometry, where
$\eta$ does not depend on $i$. By Montel (or Gr\"otzsch) it then also follows that 
the diameters of $P_{n_i+n}(c)$  must shrink to zero. 

 The proof in Case 3 goes verbatim as in the proof of \cite[Lemma 6.8]{KSS}, where in fact the 
 set $\Omega_i$ in that paper can be replaced by a puzzle piece in our setting.  (In fact, this would allow
 us to simplify the argument in \cite{KSS} for this case, because
  there the set $\Omega_i$ is obtained via the upper and lower bounds in \cite{KSS} 
which   are proved by hand, whereas here we can use the Enhanced Nest construction combined 
with the Covering Lemma.)
 \end{proof}

\subsection{Puzzle control around persistently recurrent critical points}

 \begin{lemma}[Persistently recurrent case]
 \label{Lem:PR}
 Let $c$ be a persistently recurrent critical point. 
 Then there exists $\delta>0$ so that for each $\epsilon>0$ there exists a nice puzzle neighborhood $W$ of $[c]$ with the following properties:
 \begin{enumerate}
  \item Each component of $W$ has $\delta$-bounded geometry.  
 \item For each $c'\in \Back(c)$, we have $\diam(W_{c'})\le \epsilon$, where $W_{c'} = \DomL_{c'}(W)$.
  \item The first landing map to $W$ is $(\delta,N)$-extendible w.r.t.\ $[c]$. 
 \end{enumerate} 
 \end{lemma}
\begin{proof}
The proof of this result the same as the proof of  \cite[Lemma 6.7]{KSS}. It follows by inducing a persistently recurrent complex box mapping from $F$, see Lemma~\ref{Lem:PRBM}, and the 
Key Lemma (Lemma~\ref{Lem:KeyLemma}) below.
\end{proof}

Finally, let us state and outline the proof of the Key Lemma. Because of Lemma~\ref{Lem:PRBM}, we can assume that our box mapping is persistently recurrent by considering an induced map.

\begin{lemma}[Key Lemma]
\label{Lem:KeyLemma}
Let $F \colon \U \to \V$ be a persistently recurrent non-renormalizable dynamically natural box mapping. 
Then there exists $\eta > 0$ so that for every $c \in \Crit(F)$ and every $\epsilon > 0$ there is a pair of puzzle pieces $Y' \supset Y \ni c$ such that
\begin{enumerate}
\item
\label{It:KL1}
$\diam Y < \epsilon$;
\item
\label{It:KL4}
$Y$ has $\eta$-bounded geometry at $c$;
\item
\label{It:KL2}
$(Y' \sm \ovl Y) \cap \PC(F) = \emptyset$, and $\modulus(Y' \sm \ovl Y) \ge \eta$.
\end{enumerate}  
\end{lemma}

\begin{proof}
This lemma was proven in Sections 10 and 11 of \cite{KvS} using the Enhanced Nest construction presented in Section~\ref{SSec:Enhanced} and the Covering Lemma of Kahn and Lyubich (Lemma~\ref{Lem:CoveringLemma}). Let us sketch the argument.

Let $c_0 \in \Crit(F)$ be a critical point, necessarily persistently recurrent, and let $(\bI_n)_{n \ge 0}$ be the Enhanced Nest around $c_0$. Recall that a puzzle piece $P$ around $c_0$ is called $\delta$-nice if $\modulus(P \sm \ovl{\Comp_{x} \Dom(P)}) \ge \delta$ for every $x \in P \cap \PC(F)$. Note that since $F$ is persistently recurrent, there are only finitely many components in $\Dom(P)$ intersecting $\PC(F)$. 

The key step in \cite{KvS} was to prove:
\begin{proposition}\cite[Proposition 10.1]{KvS}
\label{prop:KvSnice}
There exists a {\em beau}\footnote{That is, a constant depending only on the number and degrees of the critical points of $F$.} $\delta'>0$ and $N\in \mathbb N,$ (depending on $F$), so that for all $n\ge N$, $\bI_n$ is $\delta'$-nice. 

In particular,
for every $\epsilon >0$, there exists $\delta>0$, so that if $\bI_0$ is $\epsilon$-nice, then $\bI_n$ is $\delta$-nice for every $n \ge 1$.
\end{proposition}
We will prove this proposition after we show how it implies Lemma~\ref{Lem:KeyLemma}.
This proposition uses for its proof, in a non-trivial way, the results of Lemma~\ref{Lem:ENProps} on transition and return times for the Enhanced Nest and the Covering Lemma (Lemma~\ref{Lem:CoveringLemma}). We will sketch its proof below; for now, let us use this proposition and complete the proof of the Key Lemma. 

Since each $\bI_{n+1}$ is a pullback of $\bI_n$ containing $c_0$, Proposition~\ref{prop:KvSnice} implies that $\modulus(\bI_{n} \sm \ovl{\bI_{n+1}}) \ge \delta$. Hence, by the Gr\"otzsch inequality, the puzzle pieces $\bI_n$, and thus all puzzle pieces around $c_0$ shrink to zero in diameter.

By the discussion after Lemma~\ref{Lem:ENnu} , the degree of the map $F^\nu|_{\mathcal{B}(\bI_n)}$ does not depend on $n$, and so Proposition~\ref{prop:KvSnice} gives us that the moduli $\modulus(\mathcal{B}(\bI_n)\setminus\ovl{\mathcal{A}(\bI_n)})$ are uniformly bounded from below (see also Figure~\ref{Fig:AB}, left). 
Since the mapping from $\bI_{n+1}$ to $\bI_n$ has uniformly bounded degree,
one can construct puzzle pieces $\bI_{n+1}^- \subset \bI_{n+1} \subset \bI_{n+1}^+$ so that
$\modulus(\bI_{n+1}^+\sm\ovl{\bI_{n+1}})$ and
$\modulus(\bI_{n+1} \sm \ovl{\bI_{n+1}^-})$ are both uniformly bounded from below, and disjoint from $\PC(F)$.
In other words, the puzzle pieces $\bI_n$ are all $\delta$-free (see Section~\ref{SSSec:NiceFree} for the definition). This property together with the Koebe Distortion Theorem implies that if $\bI_n$ has $\delta$-bounded geometry at the critical point, then $F(\bI_{n+1})$ has $\delta'=\delta'(\delta)$ bounded geometry at the critical value, and pulling back $F(\bI_{n+1})$ by one iterate of $F$ to $\bI_{n+1}$ improves the bounded geometry constant. See \cite[Proposition 11.1]{KvS} for details.

Now the Key Lemma for arbitrary $c \in \Crit(F)$ follows by picking $Y$ to be $\hat{\mathcal{L}}_c(\bI_n)$ and $Y'$ to be $\hat{\mathcal{L}}_c(\bI_n^+)$ for $n$ sufficiently large.
\end{proof}

Let us finally sketch the proof of Proposition~\ref{prop:KvSnice}.

\begin{proof}[Proof of Proposition~\ref{prop:KvSnice}]

We will use the results and notation from Section~\ref{SSec:Enhanced}. Pick some big integer $M$ and assume that $n > M$ is sufficiently large. Let $A$ be a component of $\Dom(\bI_n)$ intersecting $\PC(F)$. Our goal is to estimate $\modulus(\bI_n \sm \ovl A)$ and show that for all sufficiently large $n$ such moduli are uniformly bounded from below. 

Define $t_{n, M} := p_{n-1}+p_{n-2}+\dots + p_{n-M}$ and notice that $F^{t_{n,M}}(\bI_n) = \bI_{n-M}$. It follows from Lemma~\ref{Lem:ENProps} (see also the discussion at the end of Section~\ref{SSec:Enhanced}) that $F^{t_{n,M}}(A)$ is contained in some landing domain to $\bI_{n-4}$.
 
Choose any $z\in F^{t_{n,M}}(A)\subset\mathcal{L}_z(\bI_{n-4})$, and any integer $k$ satisfying $n-4\le k \le n-M$. Let $\nu_k\ge 0$ be the entry time of $z$ into $I_k$, and set $A_k$ be the component of $F^{-\nu_k}(I^+_{k}\sm \ovl I_k)$ that surrounds $z$ (where the outer puzzle piece $\bI_n^+$ was defined before Lemma~\ref{Lem:ENProps}). Thus we obtain $M-6$ annuli $A_k$ that are contained in $\bI_{n-M}$, and all of them surround $F^{t_{n,M}}(A)$ (see Figure~\ref{Fig:ENProof}). The components of $F^{-t_{n,M}}(A_k)$ which surround $A$ give us $M-6$ annuli contained in $\bI_n$ which surround $A$. Using the Covering Lemma, we can show that by taking $M$ sufficiently large, the sum of their moduli is bounded. This yields a lower bound on $\modulus(\bI_n\sm \ovl A)$. Let us explain how this is done.

\begin{center}
\begin{figure}[h]
\includegraphics[width=1.\textwidth, trim=2 2 2 2, clip]{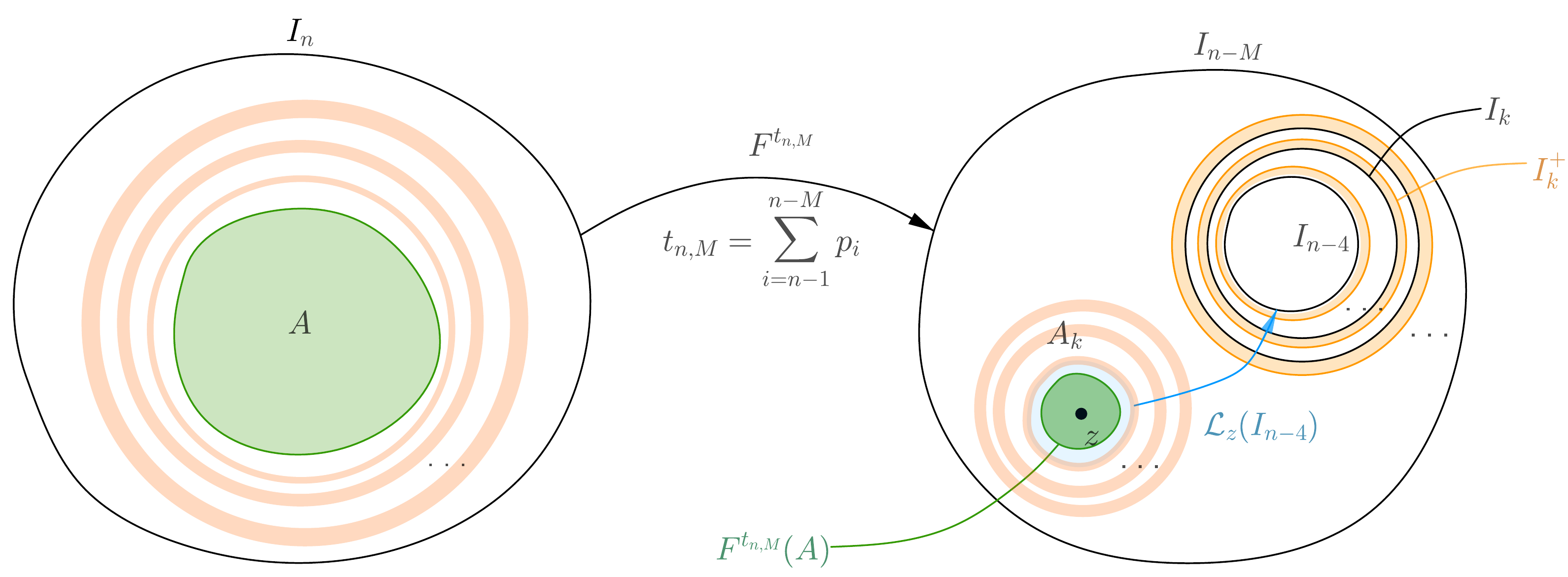}
\caption{The construction of pullback annuli in the proof of Proposition~\ref{prop:KvSnice}.}
\label{Fig:ENProof}
\end{figure}
\end{center}    

Let $\mu_n>0$ be so that $\bI_n$ is $\mu_n$-nice. When $F$ is persistently recurrent, the set $\PC(F)$ is compact, so $\mu_n$ is not zero. The construction of the Enhanced Nest gives us annuli $\bI_k^+\sm \ovl \bI_k$ that are disjoint from $\PC(F)$ and with moduli bounded from below by $K_1\mu_{k-1}$, where $K_1$ depends only on the degrees of the critical points of $F$. Since $F^{\nu_k}|_{\DomE_{z}(\bI_k^+)}$ is a landing map, its degree is bounded in terms of the degrees of the critical points of $F$, and so there is a constant $K_2$ with the property that $\mathrm{mod}(A_k)\ge\frac{K_1}{K_2}\mu_{k-1}$. This gives us a lower bound on the moduli of the annuli $A_k$ surrounding $F^{t_{n,M}}(A)$. Now we are going to use the  Covering Lemma (Lemma~\ref{Lem:CoveringLemma}) to transfer this bound to the annuli surrounding $A$. To do this, we need to bound the degree of $F^{t_{n,M}}|_A$ independently of $n$ and $M$, but this is not too hard, and is done in Step 5 of the proof in \cite[Proposition 10.1]{KvS}, with the key idea essentially presented at the end of Section~\ref{SSec:Enhanced}. Let $d$ denote the bound on the degree of this mapping. Finally we will explain how the Covering Lemma is applied. 

Let $\mu_{n-5,M-5} = \min(\mu_{n-5},\mu_{n-6}, \dots, \mu_{n-M})$. Taking $K_1$ smaller if necessary, we may assume that $K_1/K_2 < 1$. We let $C>0$ be the universal constant from the Covering Lemma. Fix $\eta = K_1 / 2K_2$ and $D$ to be degree of $F^{t_{n,M}}|_{\bI_n}$. Let $\epsilon>0$ be the constant associated to $\eta$ and $D$ by the Covering Lemma. For each annulus $A_k$, let $B$, respectively $B'$, be the regions bounded by the inner, respectively outer, boundaries of $A_k$. Fix $M=\left \lceil\frac{4d^2 K_2}{C K_1\eta}+6 \right\rceil$, and let $V=\bI_{n-M}$.  Then we have that
$$\modulus(V\sm \ovl B) > \frac{K_1}{K_2}(\mu_{n-M-1}+\dots + \mu_{n-5})\ge \frac{4d^2}{C\eta}\mu_{n-5,M-5}.$$ Thus the Covering Lemma implies that
\begin{equation} \modulus(I_n\sm \ovl A) > \min(\epsilon,\eta^{-1}\modulus(A_k), C\eta d^{-2}\modulus(V\sm \ovl B))\ge \min(\epsilon, 2\mu_{n-5, M-5}).
\label{eqn:clcb}
\end{equation} We always have that there exists a constant $K$ so that $\mu_n > K\mu_{n-1}$. Combining these estimates it is easy to argue that the $\mu_n$'s stay away from $0$.

Finally, to see that there exists $\delta' > 0$, beau, as in the statement of the proposition, observe that by \eqref{eqn:clcb} (because of the multiplication by 2 of $\mu_{n-5,M-5}$) this bound is beau provided that $\epsilon$ does not depend on $F$. The constant $\epsilon = \epsilon(\eta, D)$ is given by the Covering Lemma, and we  have that $\eta = \eta(K_1, K_2)$ and $D = D(M) = D(d, K_1, K_2, C)$. Notice that the constants $d, K_1, K_2$ are universal constants depending only on the number and degrees of the critical points of $F$ and $C$ is a universal constant given by the Covering Lemma.
\end{proof}

\section{All puzzle pieces shrink in diameter}
\label{Sec:LC}

In this section, we show how to conclude Theorem~\ref{Thm:BoxMappingsMain2}~\eqref{It:MainIt1} using Theorem~\ref{Prop:UC}. For critical points the claim is established in Theorem~\ref{Prop:UC}~\eqref{It:Shrink}. Our goal here is to show shrinking of puzzle pieces around all other points in the non-escaping set $K(F)$ of $F$.

Let $x \in K(F) \sm \Crit(F)$ be a point. We consider two cases: 
\begin{itemize}
\item[(1)] the orbit of $x$ accumulates at some critical point $c$; 
\item[(2)] the orbit of $x$ is disjoint from some puzzle neighborhood of $\Crit(F)$.
\end{itemize}

We start with the first case. Let $\epsilon>0$, and let $\mathcal W$ be the nice puzzle neighborhood of $\Crit(F)$ given by Theorem~\ref{Prop:UC} with the property that each of its components has  Euclidean diameter at most $\epsilon$. Since $\orb(x)$ accumulates at $c$,  there exists $s\in\mathbb N$ minimal so that $F^s(x)\in\mathcal W$. We let $W$ denote the component of $\mathcal W$ that contains $F^s(x)$, and we let $U=\mathcal{L}_x(W) = \Comp_x F^{-s}(W)$. By Theorem~\ref{Prop:UC} \eqref{It:Ext}, the corresponding branch of the first landing map is $(\delta, N)$-extendible with respect to $\Crit(F)$, where neither $\delta$, nor $N$ depends on $\epsilon$. So there exists an arbitrarily small puzzle piece $\tilde W\Supset W,$ such that $\modulus(\tilde W\sm \ovl W)>\delta$, and if $\tilde U = \Comp_x F^{-s}(\tilde W),$ we have that
$\modulus(\tilde U\sm \ovl U)>\delta'$, where $\delta' = \delta'(\delta, N).$ Thus by the Gr\"otzsch inequality we have that the puzzle pieces around $x$ shrink to points.

Let us treat the second case. Since $F$ is dynamically natural, $K(F) = \Kwi(F)$ (condition~\eqref{It:Nat3} of Definition~\ref{Def:NaturalBoxMapping}). Therefore there exist an infinite sequence of components $U^j$ of $\U$ and $V^j$ of $\V$ and a $\delta = \delta(x) > 0$ such that $\ovl U^j \subset V^j$, $\modulus(V^j \sm \ovl U^j) \ge \delta$, and the orbit of $x$ visits $X := \bigcup_j U^j$ infinitely often. Choose $m$ sufficiently large so that $\orb(x)$ does not intersect the puzzle neighborhood of $\Crit(F)$ of depth $m$, and suppose $m<m_1<m_2<\ldots$ is a sequence of iterates such that $F^{m_i}(x) \in X$, and let $U^{m_i} = \Comp_{F^{m_i}(x)} X$.  Let $U_i$, $V_i$ be the pair of puzzle piece, both containing $x$, obtained by pulling back $U^{m_i}$ and $V^{m_i}$ under $F^{m_i}$. By construction, each $V_i \sm \ovl U_i$ is a non-degenerate annulus. Moreover, the degree of the map $F^{m_i} \colon V_i \to V^{m_i}$ is bounded above by some constant $C \ge 1$ that depends only on $m$ and the local degrees of $F$ at the critical points, but not on $m_i$. Hence, by the Koebe Distortion Theorem, there exists $\delta' = \delta'(\delta,C) > 0$ such that $\modulus(V_i \sm \ovl U_i) \ge \delta'$ for every $i \ge 0$. Now the conclusion follows by the Gr\"otzsch inequality once we arrange the sequence $(m_i)$ so that $V_{i+1} \subset U_i$ for every $i \ge 0$.
\qed

\section{Ergodicity properties and invariant line fields}
\label{Sec:ILF}

In this section, we present the proof of Theorem~\ref{Thm:BoxMappingsMain2} (\ref{parta}). Since this theorem is automatically true if $\Crit(F) = \emptyset$ (in this case $\area(K(F)) = 0$ by definition of dynamical naturality), in this section we assume that $F$ has at least one critical point.

By Theorem~\ref{Thm:BoxMappingsMain2}~\eqref{It:MainIt1}, the puzzle pieces of a non-renormalizable dynamically natural box mapping $F$ shrink to points. Therefore, the orbit of a point $z \in K(F)$ accumulates on the fiber of a point $z' \in K(F)$ (i.e.\ $\orb(z)$ enters every small puzzle neighborhood of $z'$) if and only if $z' \in \omega(z)$. Furthermore, for a pair of distinct points $z, z' \in K(F)$, every two nests of puzzle pieces around them are eventually disjoint (starting from some large depth). We will frequently use these two observations below without further notice.

\subsection{Ergodicity properties}
\label{SSec:ErgodicityDynNat}

In this subsection, we study ergodic properties of non-renormalizable dynamically natural box mappings. In particular, in the next theorem we draw consequences for the map in presence of a forward invariant set of positive Lebesgue measure. The real analogue of this result was proved
in \cite{vSV}, see also \cite[Theorem 3.9]{McM} and \cite{Man}. Even though this result is interesting in its own right, our purpose is to use it in Section~\ref{SSec:Fields} in the situation when this forward invariant set contains the support of a measurable invariant line field.

\begin{theorem}[Ergodic properties of non-renormalizable box mappings] 
\label{Thm:ErgodicNatural}
Let $F\colon \U\to \V$ be a  non-renormalizable dynamically natural box mapping with $\Crit(F) \neq \emptyset$, and let $X \subseteq K(F)$ be a forward invariant set of positive Lebesgue measure. Let $z$ be a Lebesgue density point of $X$. Then the orbit of $z$ accumulates on a critical point $c$, and any such critical point $c$ is a point of Lebesgue density of $X$. 

Moreover, if $c$ is not persistently recurrent, then there exist puzzle pieces $\V \supset \K \Supset \J \ni c$ of $F$ with $\area(\J \cap X) = \area(\J)$ and an infinite sequence $(k_n)$ such that $F^{k_n}(z) \in \J$ and the maps $F^{k_n} \colon \Comp_z F^{-k_n}(\K) \to \K$ have uniformly bounded degrees (independent of $n$).  
\end{theorem} 

\begin{proof}
Since $F$ is dynamically natural, $\area(\Koc(F)) = 0$, and hence the orbit of Lebesgue almost every point in $X$ accumulates on a critical point. Pick $z \in X$ to be a Lebesgue density point of $X$. We can assume that $z$ is not a backward iterate of a critical point, and hence focus only on critical points in $\omega(z)$.

Define a partial ordering on the set of critical points of $F$ as follows. Given two critical points  $c,c' \in \Crit(F)$, 
we say  that $c\ge c'$ if $F^k(c)=c'$ for some $k \ge 0$, or $\omega(c)\ni c'$. Note that $c \ge c'$ implies $\omega(c)\supseteq \omega(c')$. Among all critical points in $\omega(z)$, let $c_0 \in \omega(z)$ be a maximal critical point with respect to this partial ordering. Note that there might be several such points, we pick one of them.

Recall from Section~\ref{SSec:Rec}, $[c_0]$ stands for the equivalence class of critical points $c$ such that $c \in \omega(c_0)$ and $c_0 \in \omega(c)$ (as usual, the discussion from Section~\ref{SSec:Rec} simplifies since we know that the fibers of $F$ are all trivial). Let $c_0 \ge c_1 \ge c_2 \ge \ldots \ge c_k$ be a chain in the partial ordering starting from $c_0$ and consisting of pairwise non-equivalent critical points (we pick a representative in each equivalence class\footnote{A more precise way would be to write $[c_0] \ge [c_1] \ge \ldots \ge [c_k]$.}, and the proof below repeats for each such choice).

It follows from the definition of persistent recurrence that the omega-limit set of a persistently recurrent critical point $c$ is minimal (see Section~\ref{SSec:Defns}), provided that puzzle pieces shrink to points. In particular, if $z \in \omega(c)$ and $c$ is persistently recurrent, then $\omega(z) = \omega(c)$. Therefore, $c_k$ is the only critical point that can be persistently recurrent in the chain above\footnote{This justifies the term ``minimality'' for the omega-limit set of persistently recurrent critical points: such points can be only at the end of any chain, i.e.\ persistently recurrent points are always minimal with respect to our order.}. 

\textbf{Case 1:} $c_0$ is persistently recurrent. In this case, $k=0$ and the chain consists of one element (the equivalence class $[c_0]$). It follows that there exists an integer $n_0$ so that 
if some iterate of $z$, say $F^{k}(z)$, is contained in a puzzle piece $P \ni c_0$ of depth at least $n_0$, then 
\begin{equation}\begin{array}{ll}  &\mbox{ any critical point $c \in \Crit(F)$
which is contained in  a pullback of $P$ along} \\ 
&\mbox{$\{z,\dots,F^{k}(z)\}$ satisfies $c \in [c_0]$. } \end{array} 
\label{eq:no-other}
\end{equation} 

Since $c_0$ is persistently recurrent, one can construct
an Enhanced Nest  $(\bI_n)$ around $c_0$, see Section~\ref{SSec:Enhanced} 
This nest has the property that $\diam \bI_n \to 0$ as $n \to \infty$ and that there exists $\delta>0$ so that for every $n\ge 0$ there exists a puzzle piece $\bI_n^-\subset I_n$ such that $\bI_n\cap \omega(c_0)\subset \bI_n^-$ and $\modulus(\bI_n\setminus \ovl \bI_n^-)\ge \delta$. The latter bound is Proposition~\ref{prop:KvSnice}. Let us further assume that this nest is built sufficiently deep so that (\ref{eq:no-other}) holds for each of its elements. Note here that a priori some critical orbits for critical points not in $[c_0]$ might intersect the annulus $\bI_n\setminus \bI_n^-$. However, due to \eqref{eq:no-other}, the inclusion $\bI_n\cap \omega(c_0)\subset \bI_n^-$ implies that if $F^k(z) \in \bI_n^-$ for some $k \ge 1$, then the map $\rho := F^k \colon \Comp_z F^{-k}(\bI^-_n) \to \bI_n^-$ has an extension to a branched covering onto $\bI_n$ of the same degree as $\rho$. Taking $k$ minimal and using this extension, by a standard argument we conclude that $c_0$ is a Lebesgue density point of $X$. 

\textbf{Case 2:} $c_0$ is reluctantly recurrent. From the definition of reluctant recurrence it follows that there exists a shrinking nest of puzzle pieces $(Q_i)_{i\ge0}$ around $c_0$ and an increasing sequence of integers $s_i \to \infty$ so that $F^{s_i}(Q_i) = Q_0$ and so that all the maps $F^{s_i} \colon Q_i \to Q_0$ have uniformly bounded degrees (see \cite[Lemma 6.5]{KSS}\footnote{The proof of this lemma goes through verbatim in our setting.}). Let $\K = Q_0$ and $\J = Q_{i'}$ with the smallest $i' > 0$ so that $Q_{i'} \Subset Q_0$ (which exists as the nest is shrinking). 

Finally, for $i > i'$, let $\ell_i$ be the first entry time of the orbit of $z$ to $Q_i$. Then $F^{\ell_i}(z) \in Q_i \subset \J$ and the degree of the map $F^{\ell_i} \colon \Comp_z F^{-\ell_i}(Q_i) \to Q_i$ is bounded independently of $i$. Hence, the maps $F^{\ell_i + s_i} \colon \Comp_z F^{-(\ell_i+s_i)}(Q_0) \to Q_0 = \K$ have uniformly bounded degrees for all $i > i'$. Using a standard argument we conclude that $c_0$ is a Lebesgue density point of $X$ and moreover $\area(\J \cap X) = \area(\J)$.

Now, for $j \in \{1, \ldots, k\}$, let $\J'$ be an arbitrary small puzzle piece around $c_j$ of depth larger than the depth of $\J$. Suppose $s \ge 0$ is the smallest iterate so that $F^s(c_0) \in \J'$. Then $\area(X \cap \Comp_{c_0} F^{-s}(\J')) = \area(\Comp_{c_0} F^{-s}(\J'))$. Since $F$ is a non-singular map and $X$ is forward invariant, we can iterate $X$ forward to obtain $\area(\J' \cap X) = \area(\J')$. As this is true for arbitrary small puzzle pieces $\J' \ni c_j$, we conclude that $c_j$ is a point of Lebesgue density for $X$.

\textbf{Case 3:} $c_0$ is non-recurrent. The proof in this case is similar to the proof in Case 2. First observe that $[c_0]=c_0$ and no other critical point in $\omega(z)$ can accumulate on $c_0$. Therefore, there exists a puzzle piece $\K \ni c_0$ of sufficiently large depth so that the orbits of points in $\omega(z) \cap \Crit(F)$ are disjoint from $\K$. Hence, if $\J \ni c_0$ is another puzzle piece so that $\J \Subset \K$ (again, such exists since $\fib(c_0) = \{c_0\}$) and $k$ is such that $F^k(z) \in \J$, then the map $F^k \colon \Comp F^{-k}(\J) \to \J$ is univalent and extends to a univalent map over $\K$. The proof now goes exactly as in Case 2.
\end{proof}

As usual, an \emph{ergodic component} of $F$ is a set $E \subset K(F)$ of positive Lebesgue measure so that $F^{-1}(E) = E$ up to a Lebesgue measure zero set. The corollary below shows that there are only finitely many of such sets that are disjoint up to a set of zero Lebesgue measure.

\begin{corollary}[Number of ergodic components]
\label{cor:num erg comps}
If $F \colon \U \to \V$ be a non-renormalizable dynamically natural complex box mapping, then for each ergodic component $E$ of $F$ there exist one or more critical points $c$ of $F$ so that $c$ is a Lebesgue density point of $E$.
\end{corollary}

\subsection{Line fields}
\label{SSec:Fields}

Recall that for a holomorphic map $f$ and a set $B$ in the domain of $f$, a \emph{line field} on $B$ is the assignment  $\mu(z)$ of a real line through each point $z$ in a  subset $E \subset B$ of positive Lebesgue measure 
so that the slope of $\mu(z)$ is a measurable function of $z$. A line field is \emph{invariant} if $f^{-1}(E) = E$ and the differential $\text{D}_z f$ transforms the line $\mu(z)$ at $z$ to the line $\mu(f(z))$ at $f(z)$.
The assumption that the line field is measurable, implies that Lebesgue almost every $x\in B$ is a point 
of almost continuity: for any $\epsilon>0$ 
$$\frac{\area(\{z\in D(x,r) : |\mu(z)-\mu(x)|\ge \epsilon\})}{\area(D(x,r))} \to 0$$
as $r\to 0$ (here $D(x,r)$ stands for the disk of radius $r$ centered at $x$). Alternatively one can define $\mu$ as a measurable 
Beltrami coefficient $\mu(z)d\bar z /dz$ with $|\mu|=1$. We say that  the line field $\mu$ is {\em holomorphic} on an open set $U$
 if $\mu=\bar \phi / |\phi|$ a.e. on $U$, where  $\phi=\phi(z)dz^2$ is a holomorphic quadratic differential on $U$.  We say that the line field is {\em univalent} if $\phi$ has no zeroes. In that case $\mu$
 induces a foliation on $U$ given by trajectories of $\phi$. 

The absence of measurable line fields was proved  by McMullen  for infinitely renormalizable quadratic polynomials with complex bounds 
\cite[Theorem 10.2]{McM}, in  \cite{LvS2} for covering maps of the circle 
and for real rational maps with all critical points real in  \cite{Shen}. Since then, for a number of other cases, including
for non-renormalizable polynomials \cite{KvS} and for certain real transcendental maps \cite{RvS1}.

In our exposition here we will use the following result due to Shen, \cite[Corollary 3.3]{Shen} the proof of which in our setting goes through without any changes.

\begin{proposition}[\cite{Shen}, Corollary 3.3]
\label{prop:shen}      
Let $F$ be a  dynamically natural complex box mapping.  
Let $z$ be a point in the filled Julia set $K(F)$.  Assume there exist a positive constant $\delta>0$, a positive integer $N \ge 2$, sequences $\{s_n\},\{p_n\},\{q_n\}$ of positive integers tending to infinity, sequences $\{A_n\},\{B_n\}$ of topological disks with the following properties (see Figure~\ref{Fig:ILF}):
\begin{enumerate}
\item[(i)] $F^{s_n} \colon A_n \to B_n$ is a proper map whose degree is between $2$ and $N$;
\item[(ii)] $\diam(B_n) \to 0$ as $n \to \infty$;
\item[(iii)] $F^{p_n}(z) \in A_n$, $F^{q_n}(z)\in B_n$ and there exists a topological disk
$B_n^-\subset B_n$ so that\footnote{A small typo in \cite{Shen} is corrected in here.}
$$
\left\{F^{s_n}(u) : u \in A_n,(F^{s_n})'(u) = 0\right\}\cup \left\{F^{q_n}(z), F^{p_n+s_n}(z)\right\} \subset B_n^-$$ 
and 
$$\modulus (B_n\setminus \cl B_n^-)\ge \delta.$$
\item[(iv)]  for any $n$ the following maps are univalent:  
$$F^{p_n}\colon \Comp_z F^{-p_n}(A_n)\to A_n \mbox{ and } 
F^{q_n} \colon \Comp_z F^{-q_n}(B_n)\to B_n.$$ 
\end{enumerate}
Then for any $F$-invariant line field $\mu$, either $z$ is not in the support of $\mu$,
or $\mu$ is not almost continuous at $z$.  \qed
\end{proposition} 

\begin{figure}[htbp]
\begin{center}
\includegraphics[scale=1., trim=30 20 30 20, clip]{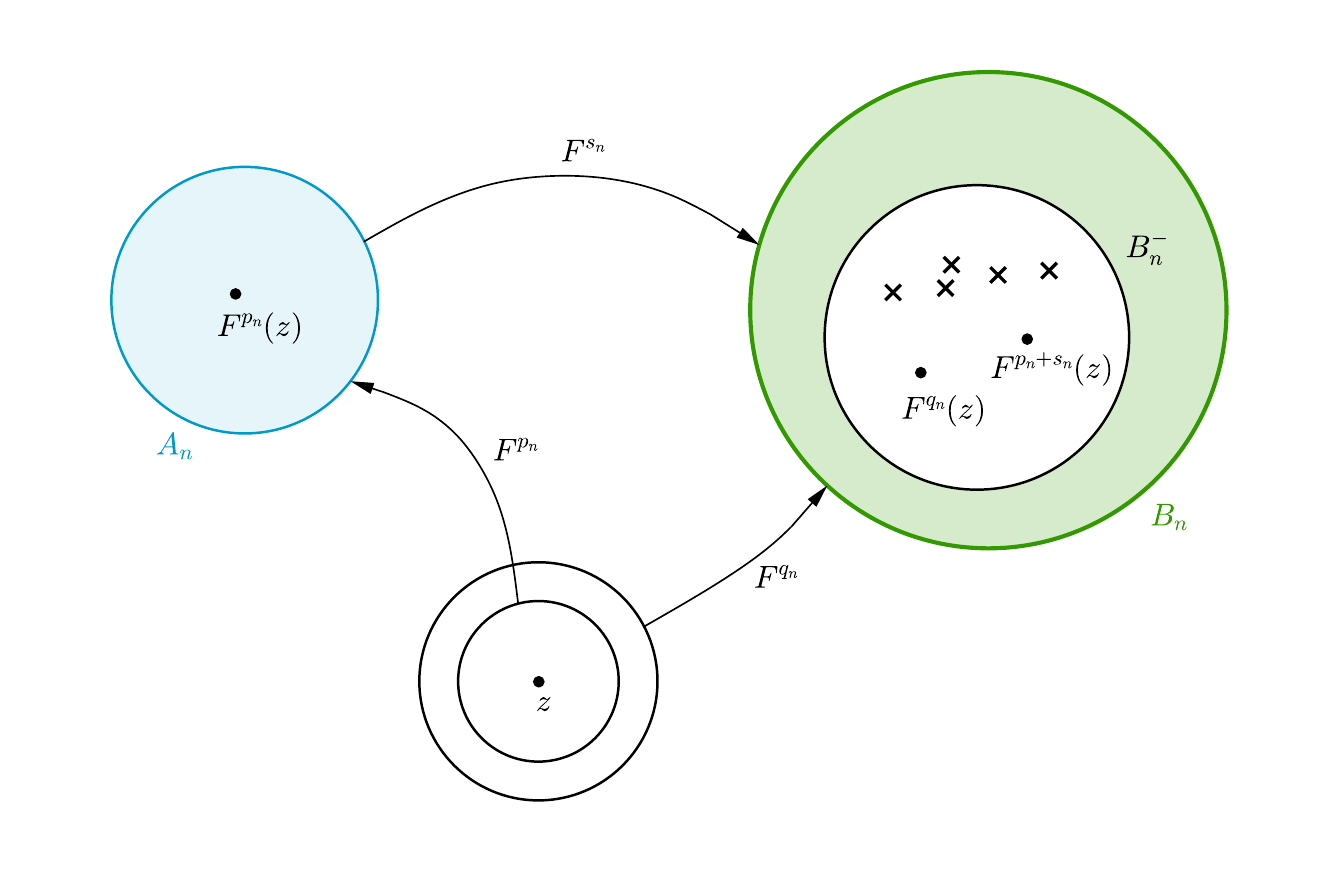}
\caption{An illustration to the statement of Proposition~\ref{prop:shen}. The critical values of the map $F^{s_n}$ are marked with crosses. The green annulus has modulus at least $\delta >0$, with $\delta$ independent of $n$. The branches $F^{p_n}$ and $F^{q_n}$ are univalent.
If $z$ is in the support of $\mu$ and $\mu$ is almost continuous at $z$, then by going to limits, one can extract univalent line fields on $A_n$ and also $B_n$, which is impossible because the map $F^{s_n}\colon A_n\to B_n$ has critical points. 
}
\label{Fig:ILF}
\end{center}
\end{figure}

The proof of the above proposition follows the same idea as the original idea by McMullen: the maps $F^{p_n}$, $F^{q_n}$ send a
neighborhood of $z$ diffeomorphically to $A_n$, $B_n$ and so since $\mu$ is increasingly constant near a point of almost continuity $z$, by property (iv)  we obtain in $A_n,B_n$ what increasingly looks like continuous foliations of trajectories of the corresponding quadratic differentials. However, the map 
$F^{s_n}\colon A_n\to B_n$ has critical points, so this impossible. To make this argument precise one needs in particular
assumption (iii). 

We are now in position to prove Theorem~\ref{Thm:BoxMappingsMain2} \eqref{parta} in the form of the following theorem.

\begin{theorem}[Invariant line fields] 
\label{Thm:ILF}
Let $F\colon \U\to \V$ be a  non-renomalizable dynamically natural box mapping with $\Crit(F) \neq \emptyset$. Assume that $F$ supports a measurable invariant line field on a forward 
invariant subset $X\subseteq K(F)$ of positive Lebesgue measure.
Then there exists a puzzle piece $\J$ of $F$ so that 
$\area(X\cap \J) = \area(\mathcal \J)$
and so that the invariant line field extends to an invariant univalent line field on $\J$. 
\end{theorem}

\begin{proof}
The proof uses the result of Theorem~\ref{Thm:ErgodicNatural}. Let $z \in X$ be a Lebesgue density point which is a point of almost continuity for $\mu$. Assume that $\omega(z)$ contains a reluctantly recurrent or non-recurrent critical point, say, $c'$. By Theorem~\ref{Thm:ErgodicNatural}, there exist puzzle pieces $\K \Supset \J \ni c'$ of $F$ with $\area(\J \cap X) = \area(\J)$ and an infinite sequence $(k_n)$ such that $F^{k_n}(z) \in \J$ and the maps $F^{k_n} \colon \Comp_z F^{-k_n}(\K) \to \K$ have uniformly bounded degrees, say, bounded by some $N$ independent of $n$. Write $U_n:=\Comp_z F^{-k_n}(\J)$. Since $z$ is a Lebesgue density point of $X$, it follows that $\area(U_n\cap X)/\area(U_n)\to 1$ as $n \to \infty$ (note that $(U_n)$ form a shrinking nest of puzzle pieces with uniformly bounded geometry). As $F^{k_n}|_{U_n}$ is a composition of univalent maps with bounded distortion 
and at most $N$ unicritical covering maps with a unique critical point of  the form $w\mapsto w^{d_i}$,
we obtain that $\mu$ is a holomorphic line field on $\J$. Since $\mu$ can have only finitely many singularities on $\J$, by, if necessary, shrinking $\J$ slightly we can obtain that the line field $\mu$ is univalent on $\J$. 

So assume now that $\omega(z)$ contains only persistently recurrent critical points, and let $c_0 \in \omega(z)$ be one of them. For $c_0$, we start by repeating the same argument as in the proof of Case 1 of Theorem~\ref{Thm:ErgodicNatural}. 
Based on that argument, let show how to construct the topological disks $A_n, B_n$ and the sequences $\{s_n\}, \{p_n\}, \{q_n\}$ with the properties required by Proposition~\ref{prop:shen}.

For each $c \in [c_0]$, define $W_{n,c}^- := \DomL_{c} (\bI_n^-)$ and let $F^{t_{n,c}} \colon W_{n,c}^- \to \bI_n^-$ be the corresponding branch of the first landing map to $\bI_n^-$. If now $W_{n,c} := \Comp_{c} F^{-t_{n,c}}(\bI_n)$, then $W_{n,c} \cap \omega(c_0) \subset W_{n,c}^-$. Define $W_n^- := \bigcup_{c \in [c_0]} W_{n,c}^-$.

Let $q_n$ be the first entry time of $\orb(z)$ to $W_n^-$, and let $B_n^-:=W_{n,c_1}^-$ be the component of $W_n^-$ containing $F^{q_n}(z)$. Put $B_n := W_{n,c_1}$. The map $F^{q_n} \colon \Comp_z F^{-q_n}(B_n^-) \to B_n^-$ is univalent (because of \eqref{eq:no-other} and since $W_n^- \supset [c_0]$). Moreover, the inclusion $W_{n,c_1} \cap \omega(c_0) \subset W_{n,c_1}^-$ implies that this map has a univalent extension $F^{q_n} \colon \Comp_z F^{-q_n}(B_n) \to B_n$ with $\modulus(B_n \sm \cl B_n^-) \ge \delta' >0$, where $\delta'$ depends only on $\delta$ and the number and local degrees of the critical points in $[c_0]$, in particular, $\delta'$  is independent of $n$. The latter moduli bound holds because the branch $F^{t_{n,c_1}} \colon B_n^- \to \bI_n^-$ has degree bounded independently of $n$, and same is true for its extension $F^{t_{n,c_1}} \colon B_n \to \bI_n$, again due to \eqref{eq:no-other}.

Now, let $V_n^- := \bigcup_{c \in [c_0]} \DomL_{c} (B_n^-)$, and let $p_n$ be the entry time of $\orb(z)$ to $V_n^-$. Let $A_n^-$ be the component of $V_n^-$ containing $F^{p_n}(z)$ and $F^{s_n} \colon A_n^- \to B_n^-$ be the corresponding branch of the first landing map to $B_n^-$; as such, the degree of this branch is bounded independently of $n$. Let $A_n := \Comp_{A_n^-} F^{-s_n} (B_n)$. By construction, $A_n \cap \omega(c_0) \subset A_n^-$, and hence by \eqref{eq:no-other} the map $F^{s_n} \colon A_n \to B_n$ has the same degree as $F^{s_n}|_{A_n^-}$, and thus it is bounded independently of $n$, while the map $F^{p_n} \colon \Comp_z F^{-p_n}(A_n) \to A_n$ is univalent. 

Using the disks and sequences of iterates constructed above, we are in position to apply Proposition~\ref{prop:shen}; this gives us the desired contradiction to our choice of $z$ to be a point of almost continuity for $\mu$.
\end{proof}

\section{Latt\`es box mappings}\label{sec:lattes} 

In this section, we construct and describe non-renormalizable dynamically natural box mappings $F\colon \U\to \V$ with non-escaping set of positive area and which support a measurable invariant line field. We will call these {\em Latt\`es box mappings}. 
Such a map is quite special, and  should be thought 
of as an analogue of a Latt\`es mapping on the Riemann sphere, see \cite{milnor-lattes}.

\subsection{An example of a dynamically natural Latt\`es box mapping}
\label{SSec:LattesExample}

\begin{proposition}[Latt\`es box mappings exist]
\label{prop:lattes} 
There exists a non-renormalizable dynamically natural box mapping $F\colon \U\to \V$ 
with $\area(K(F)) > 0$ and a measurable invariant line field supported on $K(F)$. For the example we give here
\begin{enumerate}
\item
$\Crit(F) = \{c\}$ and $F^2(c) = F(c)$, i.e.\ the unique critical point of $F$ is strictly pre-fixed;
\item $\V$ consists of two components $V$ and $V'$  such that $U:=\U\cap V=V$
contains $c$ and such that  $F(c)$ is contained in a 
component $U'$ of $\U\cap V'$;
\item $\U \cap V'$ contains infinitely many components and they tile $V'$;
\item $F(U') = V'$ and for each component $U''$ of $\U\cap V'$ distinct from $U'$
one has $F(U'')=V$. 
\end{enumerate} 
\end{proposition} 
\begin{proof}
This example starts with the box mapping  $F_2\colon \U_2\to \V_2:=(-1,1)\times (-1,1)$  from Theorem~\ref{Thm:Path} \eqref{It:PathAs2} which was based on the Sierpi\'nski carpet construction.
This box mapping leaves invariant the horizontal line field on $(-1,1)\times (-1,1)$. 
Denote the component of $\U_2$ containing $0$ by $U^0_2$. We can make sure that 
$\U_2$ is invariant under $z\mapsto -z$, i.e.\ if $U$ is a component of $\U_2$,
then $-U$ is also a component of $\U_2$. Moreover, it is straightforward to modify this map so that  
$F_2(-z)=-F_2(z)$ for $z\in U^0_2$ and $F_2(-z)= F_2(z)$ for $z\in \U_2\setminus U_2^0$, and so that $F_2$ still preserves horizontal  lines. 

By Theorem~\ref{Thm:Path} \eqref{It:PathAs2}, the 
set $K(F_2)$ has full Lebesgue measure in $\V_2$.  Moreover, 
we have that each puzzle piece of $F_2$ maps linearly over $\V_2$. 
It follows that the orbits of Lebesgue almost all points $z\in K(F_2)$ will enter an arbitrarily deep puzzle piece
around $0$. Indeed, assume by contradiction that there exists a puzzle piece $P \ni 0$
so that the set of points the orbits of which avoid $P$ has positive Lebesgue measure; call this set $Y$. 
Take a Lebesgue density point $y$ of $Y$ and a puzzle $P_n\ni y$ of depth $n \ge 0$; note that $\diam P_n \to 0$ as $n \to \infty$.
Since $F^n_2 \colon P_n\to \V_2$ is a linear map, we have $\area(P_n\cap Y) / \area(P_n) \le (\area(\V_2) - \area(P))/\area(\V_2) <1$, 
which gives a contradiction.

Let $\phi\colon (-1,1)\times (-1,1)
\to \disk$ be the Riemann mapping.   For simplicity choose $\phi$ so that it preserves the real line 
and so that $\phi(0)=0$. By symmetry it follows that $\phi(-z)=-\phi(z)$ for all $z\in \V_2$. 
Set $\hat \V_2 := \disk$, $\hat \U_2 :=\phi(\U_2)$ and define $\hat F_2\colon\hat \U_2 \to \hat \V_2$
by $\hat F_2=\phi\circ F_2\circ \phi^{-1}$. 
Then $\hat F_2$ also has a holomorphic invariant line field on $\disk$. Moreover, $\hat F_2$ preserves $0$, maps real points to real points, and is univalent on each component of $\hat \U_2$. Let $0\in \hat U^0_2 :=\phi(U^0_2)$ be the {\lq}central{\rq} component of $\hat F_2$. It follows that $\hat F_2(-z)=- \hat F_2(z)$ for $z\in \hat U^0_2$ and $\hat F_2(-z)= \hat F_2(z)$ for $z\in \hat \U_2\setminus \hat U_2^0$

Finally, using $\hat F_2$ we will now define a box mapping $F \colon \U \to \V$, as described in the statement of Proposition~\ref{prop:lattes}, by taking $\V$ to be a formal union of two copies of $\disk$, and $\U$ to be a formal union of $\disk$ and countably many open topological disks that tile some other copy of $\disk$. The details are as follows (note that this construction is taken for simplicity, and can be easily modified so that $\V$ would embed into $\C$). 

Let $Q\colon \disk\to \disk$ be the quadratic map $Q(z)=z^2$. Let 
 $U=V=V'=\disk$ and define
$F\colon U\to V'$ by $F(z)=Q(z)$ for $z\in U$.  Moreover, define $\U \cap V' := Q(\hat \U_2)$,  $U':=Q(\hat U_2^0)$ 
and   $F\colon  U'\to V'$ 
by $$F(z)=  Q\circ \hat F_2\circ Q^{-1}(z) .$$ 
This map is well defined because $\hat F_2(-z)=-\hat F_2(z)$ for all $z \in \hat U^0_2$ and so different choices of $Q^{-1}(z)$ give the same result. Since $\hat F_2(0)=0$, this map is also univalent. 

Finally, define 
$F\colon (\U\cap V')\setminus U' \to V$ by 
$$F(z)=  \hat F_2\circ Q^{-1}(z).$$
Again this map is univalent and well-defined because $\hat F_2(-z)=\hat F_2(z)$ for each $z$
in a non-central component of $\hat U_2$. 

The fact that $F$ preserves an invariant line field follows from the commuting diagram in Figure~\ref{Fig:Lattes}\footnote{The figure was drawn using T.A.\,Driscoll's Schwarz--Christoffel Toolbox for MATLAB, see \url{https://tobydriscoll.net/project/sc-toolbox/.}}. Indeed, there is a holomorphic line field  $\mu$ on $V=\phi(\V_2)$, 
which is the image under $\phi$ of the horizontal line field 
on $\V_2=(-1,1)\times (-1,1)$ and which is invariant under $\hat F_2$. 
The pushforward  $Q_*\mu$ of this line field $\mu$ on $V'=Q(V)$  by $Q$ is again a holomorphic line field
which, by definition, is invariant under $F\colon U'\to V'$.  The map $F$
sends each component of $(\mathcal U\cap V' )\setminus U'$ univalently onto $V$
and again the line field of $Q_*\mu$ is mapped by $F$ to the line field  $\mu$ 
by definition. 

By construction, $K(F) = \Kwi(F),$ and by the 2nd paragraph in the proof, almost every point in $K(F)$ accumulates onto the critical point, and hence $F$ is dynamically natural. Finally, $F$ is non-renormalizable since $c$ and $F(c)$ lie in different components of $\U$ and the only critical orbit is non-recurrent. 
\end{proof} 

\begin{figure}[hbt!]
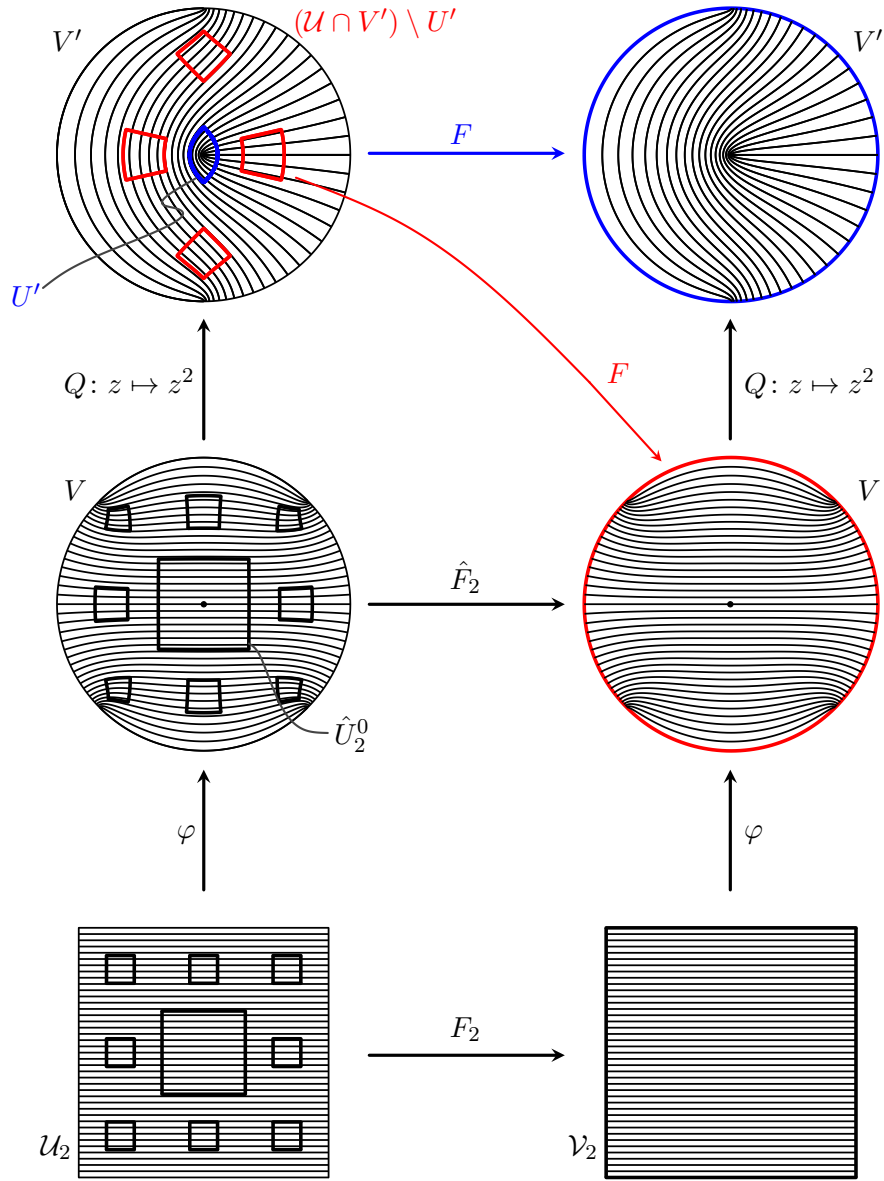

\definecolor{uuuuuu}{rgb}{0.26666666666666666,0.26666666666666666,0.26666666666666666}
\definecolor{ffffff}{rgb}{1.,1.,1.}
\definecolor{ffqqqq}{rgb}{1.,0.,0.}
\definecolor{qqqqff}{rgb}{0.,0.,1.}

\caption{The construction of a Latt\`es box mapping $F$. The red, respectively blue, domains are mapped by the correspondingly colored branches of the map $F$ over $V$, resp.\ $V'$. Only the first two steps in the Sierpi\'nski carpet construction for $F_2$ are shown. Some leaves of invariant foliations are shown with thin black lines.} 
\label{Fig:Lattes}
\end{figure}

\subsection{Properties of Latt\`es box mappings} 

Latt\`es box mappings are rather special. Some of their necessary features 
are described in the following result, which, in particular, provides a proof of Theorem~\ref{Thm:BoxMappingsMain2} \eqref{partb}.

\begin{proposition}[Properties of Latt\`es box mappings]  \label{prop:lattes-description} 
Assume $F \colon \U \to \V$ is a Latt\`es box mapping, i.e.\ $F$ is a non-renormalizable dynamically natural box mapping that has an invariant line field supported on a forward invariant measurable set $X \subseteq K(F)$ of positive Lebesgue measure. Let $\J \subset \U$ be the puzzle piece from Theorem~\ref{Thm:ILF}. Define $Y := \bigcup_{i\ge 0} F^i(\J)$ to be the union of forward orbits for all points in $\J$ and $\Crit'=\{c_0,\dots,c_k\}$ to be the set of critical points of $F$ which are contained in $Y$. Moreover, let  $\PC' := \left\{F^i(c_j) \colon i>0, j=0,\dots,k\right\}$. 
Then:
\begin{enumerate}
\item  
$\PC'$ is a finite set, i.e.\ each critical point in $\Crit'$ is eventually periodic;
\item no forward iterate of a critical point in $\Crit'$ is eventually mapped into another point 
in $\Crit'$;
\item each critical point in $\Crit'$ is quadratic; 
\item each critical point in $\Crit'$ has real multiplier; 
\item 
\label{It:PCF}
$F^{-1}\left(\PC'\right)\cap Y = \PC'\cup \Crit'$.
\end{enumerate} 
\end{proposition} 

In the proof of this proposition we will use the following  several times:

\medskip 

\noindent 
{\bf Observation:} Suppose that $F^t$ has a critical point at $c$ of order $d\ge 2$. Then there exists a  neighborhood 
$U$ of $c$ and  a conformal function $S\colon U\to \C$ with $S(c)=c$, 
$\Df S(c)=e^{2\pi i /d}$ and so that $F^t\circ S=F^t$ on $U$. 
Hence if there exists an $F$-invariant line field which is holomorphic on $U$,  then the line field
(and the  foliation defined  by this   line field) is also invariant under
the symmetry $S$. This implies that the invariant line field can only be univalent on $U$ if $d=2$. 
Moreover, the line field cannot be univalent both in $U$ and $F^t(U)$.

\begin{proof}[Proof of Proposition~\ref{prop:lattes-description}] 
By Theorem~\ref{Thm:ILF}, $\area(X \cap \J) = \area(\J)$ and the invariant line field in $\J$ is univalent and so describes a smooth invariant foliation on $\J$. 

{\bf Claim 1:} Let $t_0 > 0$ be minimal so that $F^{t_0}(\J)$ contains a critical point $c_0$.
Then  $c_0$ is eventually periodic.  

{\bf Proof of Claim 1:}  Assume  by contradiction 
that $F^i(c_0)$, $i\ge 0$,  consists of infinitely many distinct points.  Let $t>0$ be minimal so that $F^t(\J)$ is a 
connected component $V$ of $\V$. By choosing $\J$ smaller if necessary, we can assume that $t \ge t_0$. Let $\hat \V$ be the union of components of $\V$  visited by  $F^i(c_0)$, with 
$i> t-t_0.$ Since there are only finitely many components of $\V$ and the forward orbit of $c_0$ is infinite, at least one of these components, say $\hat V$, is visited infinitely many times and we can assume that this component is  tiled by infinitely many components of $\U$.  Let $t'$ be so that $F^{t'}(\J)=\hat V$.
Let $v_1,\dots,v_n$ be the critical values of $F^{t'}|_\J$. Then there exists a puzzle piece
$Q\subset \hat V$ containing some iterate of $c_0$ which is disjoint from $\{v_1,\dots,v_n\}$
(here we use that puzzle pieces shrink in diameter to zero). 
Then there exists a puzzle piece $P \subset \J$ so that $F^{t'}(P)=Q\ni F^i(c_0)$ and $F^{t'}|_P$ is univalent. 
Hence the line field is univalent in a neighborhood of $F^i(c_0)$. 
However, by the  assumption in the claim and since $F^i(c_0)\in Q$, there exists a puzzle piece 
 $P' \subset \J$ so that $F^{t''}(P') = Q$ for some $t''>0$ and so that the orbit $F(P'), \ldots, F^{t''}(P')$ passes through $c_0$. It follows that  there exists a point $z\in P'$  so that $\Df F^{t''}(z)=0$.
 Since the  line field is univalent on $\J\supset P'$ and on $Q$ this gives a contradiction to 
 the above observation. 

In fact, this argument is easy to see: if $\phi$ is the quadratic differential on $\J$ corresponding to the univalent line field $\mu$, then the trajectories of the quadratic differentials $(F^{t'}|_P)_* \phi$ and $(F^{t''}|_{P'})_* \phi$ on $Q$ will not match because the former will have no singularity, while the latter will have at least one singularity corresponding to an iterate of a critical value; this is a contradiction to invariance of $\mu$. Thus we can conclude that $c_0$ is eventually periodic, proving the claim. 

{\bf Claim  2:} The degree of the critical point $c_0$ is two, and 
no iterate of $c_0$ meets another critical point. Moreover, it is impossible
that there exists a univalent line field in a neighborhood of any forward iterate of $c_0$.

{\bf Proof of Claim 2:}   This immediately follows from the above Observation. 
Alternatively, let $q$ be the quadratic differential whose horizontal trajectories give a foliation on $\J$ and its forward iterates. We say that $q$ has \emph{zero of order $k \ge 0$} at some point $z_0$ if in some local coordinate system with the origin at $z_0$ the differential $q$ has the form $w^k dw^2$. Therefore, if the foliation given by $q$ is invariant under $F$, then the local degree of $z_0$ w.r.t.\ $F$ must divide $k+2$ \cite[Section 5.3]{HubBook1}. 

From the discussion in the previous paragraph it follows that since there exists a univalent line field near $c_0$, i.e.\ the quadratic differential has zero of order $0$ at $c_0$, the degree of $c_0$ must be two. Therefore, $q$ has a simple pole at $F(c_0)$, and thus the foliation has a $1$-prong singularity at $F(c_0)$ (see Figure~\ref{Fig:Lattes}, top left, for an example of a 1-prong singularity). Assume by contradiction that $F^i(c_0)=c_1$ and that none of the points $F(c_0),\ldots,F^{i-1}(c_0)$ are critical. This way we have an invariant foliation near $c_1$ with $1$-prong singularity at $c_1$. But then it is impossible to map this foliation to a foliation near $F(c_1)$ so as to preserve the foliation: a $1$-prong singularity has no local symmetries that are required if one wants to preserve the trajectories. The claim follows.

{\bf Claim  3:} Each critical point in $Y$ is eventually periodic. 

{\bf Proof of Claim 3:} 
Assume that some iterate of $\J$ intersects a critical point $\ne c_0$, then  take
 $t_1>t_0$ minimal so that $F^{t_1}(\J)$ contains a critical point $c_1\ne c_0$. 
Then, since $c_0$ is eventually periodic and $F^i(c_0)\ne c_1$ (by Claim 2), 
there exists a puzzle piece $\J_1\subset \J$ 
so that $F^{t_1}|_{\J_1}$ is univalent, and so that $F^{t_1}(\J_1)\ni c_1$. Using exactly 
the same argument as in Claim 1 it then follows that $c_1$ is also eventually periodic.

{\bf Claim 4:}  The degree of each critical point in $\Crit' = \{c_0, \ldots, c_k\}$ is equal to two. Moreover, there exists no univalent line field on 
a neighborhood of any point in $\PC'$.  

{\bf Proof of Claim 4:} As we saw in Claims 1 and 2, there exists a univalent line field
on a neighborhood $U_i$  of $c_i$, $i=0,\dots,k$. The proof now follows similarly as in Claim 2.

{\bf Claim 5:}  Let $p_0,p_1,\dots,p_k$ be the periodic points on which the critical points 
$c_0,\dots,c_k$ eventually land (some of these periodic points might 
coincide). Then the multipliers at these periodic points are real. 
Moreover, the line field is holomorphic  on the components $V_i$ of $\V$ containing $p_i$. 

{\bf Proof of Claim 5:} There exists a holomorphic line field (with a unique $1$-prong singularity) in a 
neighborhood of $p_i$. This line field can only be invariant near the periodic point $p_i$ if the multiplier is real. Moreover, by iterating $F$ one obtains that the line field near the periodic point 
$p_i$ extends to a holomorphic line field on $V_i$.  Let  $U_i\ni p_i$ be the component of $\U$ containing  $p_i$. Then the only  singularities of this line field on $V_i$ are at iterates  of a critical point $c_j\in U_j$ 
(assuming such a critical point $c_j$  exists).  

{\bf Claim 6:} The only singularities of the holomorphic line field on $Y$
are in $\PC'$. In particular, if $z\in Y$ and $F(z)\in \PC'$, then either $z\in \Crit'$,
or $z\in \PC'$. 

{\bf Proof of Claim 6:} The first statement follows from the fact that the line field 
on $\J$ is assumed to be univalent, and so the only way singularities 
of the line field on $Y$ can be created is by passing through some critical point. The 2nd statement follows:
the line field has a singularity at $z$  if and only if  it has a singularity at $F(z)$. 
\end{proof} 

\begin{remark}
This proposition does \emph{not} imply that such a box mapping $F$ is postcritically finite. 
Indeed, it is possible to take two box mappings $F_i\colon \U_i\to \V_i$
where  $F_1$ has an invariant line field (as described in Section~\ref{SSec:LattesExample}) 
and $F_2$ is one without an invariant line field and with an infinite postcritical 
set. Then taking the disjoint union of $\U_i$  and $\V_i$ we obtain a new 
map $F\colon \U\to \V$ with an invariant line field. 

However, if $F\colon\U\to\V$ is a Latt\`es complex box mapping, then there exist $\U'\subset\U$ and $\V'\subset\V$ so that $F|_{\U'}\colon\U'\to\V'$ is a postcritically finite complex box mapping.
\end{remark}

\subsection{Further remarks on Latt\`es box mappings}

We end this section with some discussion regarding Latt\`es box mappings, e.g.\ as constructed in Proposition~\ref{prop:lattes}, and their properties established in Proposition~\ref{prop:lattes-description}.

\subsubsection{M\"obius vs.\ quasiconformal conjugacy in presence of Latt\`es parts.}
\label{sssec:lattesqc}
The box mapping $F$  we constructed in Proposition~\ref{prop:lattes} is \emph{flexible}:
there exists a family of maps $F_t$ depending holomorphically on $t\in \disk$
so that $F_0=F$ and so that $F_t$ is qc-conjugate but not conformally
conjugate to $F$. An explicit way to construct such maps for $t\in (-1,1)$ is
to consider a family of maps exactly as before, but replacing
the map $F_2$ by the map $F_{2,t}:=h_t^{-1}\circ F_2 \circ h_t$, where 
\[
h_t(x,y)=\left(\frac{t}{1-t}x,y \right)
\]
(so replacing squares by rectangles). As $|t|\to 1$, the modulus of the central rectangle inside $\mathcal U_2$ tends monotonically to $\infty$ and thus one
can see that $F_t$ is not conformally conjugate to $F_0$ for $t>0$.

From Theorem~\ref{Thm:BoxMappingsMain2}\eqref{It:MainIt3} it follows that if $F, G$ are combinatorially equivalent non-re\-nor\-malizable dynamically natural box mappings so that there exists
a qc homeomorphism $H$ which is a conjugacy on the boundary (condition \eqref{part3c} of that theorem), then $F$ and $G$ are qc-conjugate. 
Suppose that, in addition, $F \colon \U \to \V$ is \emph{not} a Latt\`es box mapping, 
and so does not have an invariant line field. Then $F, G$ are hybrid conjugate, 
i.e.\ the qc-dilatation of the  qc-conjugacy between $F$ and $G$ 
 vanishes in the filled Julia set of $F$.   (In particular, if the filled Julia set has full measure in the $\V$, then the qc-conjugacy 
 is conformal.)  This is in some sense an  analogue of the Thurston Rigidity 
Theorem which states that two postcritically finite topologically conjugate maps
are M\"obius conjugate.

\subsubsection{Latt\`es rational maps vs.\ Latt\`es box mappings.}

Recall, that a rational map on the sphere is called \emph{Latt\`es} if it is a degree two quotient of an affine torus endomorphism. In \cite{milnor-lattes}, Milnor gave a complete description of all Latt\`es rational maps. For Latt\`es box mappings, a similar description may be impossible, but nevertheless it seems feasible to give a complete list of all  Latt\`es box mappings. 

Since Latt\`es rational maps are examples of Collet--Eckmann maps, by the result in \cite{PRL07, PRL} it follows that one can induce a Latt\`es box mapping from a Latt\`es rational map (see also the discussion in Section \ref{SSSec:NiceCouples}).  

Finally, property \eqref{It:PCF} in Proposition~\ref{prop:lattes-description} is the analogue
of the property of Latt\`es rational maps $f\colon \Cc \to \Cc$ that $f^{-1}(\PC(f))
= \PC(f)\cup \Crit(f)$, where $\PC(f)$ is the set of forward iterates of $\Crit(f)$ (the \emph{postcritical set}). Latt\`es rational maps play a special role in the study of postcritically finite\footnote{A map $f$ is called \emph{postcritically finite} if each critical point of $f$ is eventually periodic. In the setting of branched coverings, a \emph{critical point} is a point at which the map is not locally injective. Postcritically finite topological branched coverings of the sphere are sometimes called \emph{Thurston maps}.} branched coverings of the sphere. A celebrated result by W.~Thurston gives a criterion when a postcritically finite branched covering of the sphere is equivalent (in some precise sense) to a rational map, see \cite{DHThurston}. For this criterion to work, it is necessary to assume that the branched covering has so-called \emph{hyperbolic orbifold}, which roughly translates to the fact that the corresponding Thurston $\sigma$-map is contracting on the Teichm\"uller space of possible conformal structures of the sphere marked with the postcritical set. The orbifold is non-hyperbolic if and only if $f^{-1}(\PC(f)) = \PC(f)\cup \Crit(f)$ and $\# \PC(f) \le 4$ (see \cite[Lemma 2]{DHThurston}, and also \cite[Appendix C8]{HubBook2}); such orbifolds are called \emph{parabolic}. The maps with parabolic orbifold and $\# \PC(f) \le 3$ are always equivalent to a rational map and can be completely classified. The case when $\# \PC(f) = 4$ is more interesting. The only rational maps with this property are exactly flexible Latt\`es rational maps. In general, a Thurston map $f$ with $4$ postcritical points (with hyperbolic or parabolic orbifold) admits a flat structure with $4$ cone singularities and can be studied using so-called \emph{subdivision rules}, see \cite{BM}.


\section{A Ma\~n\'e Theorem for complex box mappings}
\label{Sec:Mane}

A classical theorem of Ma\~n\'e \cite{Man, ST} states that for a rational map, a forward invariant compact set
in the Julia set is either expanding, contains parabolic points or critical points, or intersects the $\omega$-limit set of a recurrent critical point (see also \cite[Section III.5]{RealBook} for the real version of Ma\~n\'e's theorem). In this section we will show that a similar, and in some sense stronger statement also holds for box mappings for which 
{\em each} domain
is {\lq}well-inside{\rq} its range.

More precisely, let $F\colon \U\to \V$ be a complex box mapping for which there exists $\delta>0$ so that for {\em each}
$x\in \U$ we have  
\begin{equation} \modulus(\V_x \setminus \ovl\U_x)\ge \delta \mbox{ or } \U_x=\V_x. \label{eq:maneassump} 
\end{equation} 
 Here $\U_x$  denotes
the component of $\U$ that contains a point $x\in \U$. The theorem below gives a lower bound 
for $|DF^k(x)|$ provided the distance $d(F^k(x),\partial \V)$ of $F^k(x)$ to $\partial \V$ is at least
$\kappa>0$ and the iterates $x,\dots,F^{k-1}(x)$ do not enter {\lq}deep{\rq} critical puzzle pieces.

Let $F_0=F$, $\U^0=\U$ and $\V^0=\V$ and  let $\V^1$ be the union of the components of $\U^0$ containing critical points. Next,   for $i\ge 0$  define 
inductively $F_{i+1}\colon \U^{i+1}\to \V^{i+1}$ to be  the first return map to $\V^{i+1}$ and denote by 
$\V^{i+2}$  the union of the components of $\U^{i+1}$ containing critical points.
If $y$ is a forward iterate of $x$ under various compositions of $F_0,\dots,F_i$ then we define
$t(y,x)$ so that $y=F^t(x)$.

\begin{theorem}[Ma\~n\'e-type theorem for box mappings]
\label{thm:mane}\label{Thm:Mane}
Let $F\colon \U\to \V$ be a complex box mapping so (\ref{eq:maneassump}) holds. 
For each $\nu\ge 1$ and each $\kappa>0$ there exists $\lambda>1$ and $C>0$ so that for all $k\ge 0$
and each $x$ so that $x,\dots,F^{k-1}(x)\in \U\setminus \V^\nu$, $d(F^k(x),\partial \V)\ge \kappa$ one has 
$$|\Df F^k(x)|\ge C \lambda^k.$$
\end{theorem} 

\begin{remark} Notice that this statement is more similar to the statement of the Ma\~n\'e theorem
for real maps than the one for rational maps. Indeed, in the above theorem it 
is not required that the distance of the orbit of 
$x$ to the $\omega$-limit set of any recurrent critical point is bounded away from zero. 
For general rational maps (rather than box mappings as above) the 
inequality $|\Df F^k(x)|\ge C \lambda^k$ fails: for example if $x$ is in the boundary 
of a Siegel disk and this boundary does not containing critical points. 
\end{remark} 
\begin{remark}  
$\lambda>1$ and $C>0$ also depend on $F$. 
\end{remark} 


To prove Theorem~\ref{Thm:Mane} we first need to introduce some notation, and prove some 
preparatory lemmas.  
Denote by $||v||_O$ be the Poincar\'e norm of a vector $v\in T_x\C$ 
on a simply connected  open domain $O\ni x$ and by   $||v ||$ the Euclidean norm of $v$.

\begin{lemma}\label{lem:poinc-exp/cont}  Let $F\colon \U\to \V$ be a complex box mapping so that 
(\ref{eq:maneassump}) holds and define $F_i$ as above. Then for each $i\ge 0$
 there exist $\delta_i>0$ and $\tau_i>1$ and for each 
$\kappa_0>0$  there exists $\rho_i\in (0,1)$ (which also depend on $F$) 
so that for each $x\in \U^i$ and each $v\in T_x \U^i$
\begin{enumerate}
\item 
we have  either 
$$\modulus(\V^i_x \setminus \ovl\U^i_x)\ge \delta_i \,\,  \mbox{ or }\,\,  \U^i_x=\V^i_x.$$
If $\U^i_x=\V^i_x$ then $\U^i_x$ contains a critical point. 
\item If $\U_x^i$ does not contain a critical point of $F_i$ then
$\U^i_x$ is compactly contained in $\V^i_x$ and 
$$||\Df F_i(x)v||_{\V^i_{F_i(x)}}  \ge \tau_i ||v||_{\V^i_x}.$$ 
\item If $\U^i_x$ contains a critical point  of $F_i$ 
and $d(x,\Crit(F))\ge \kappa_0$, then $\U^i_x= \V^{i+1}_x$ and 
 $$||\Df F_i(x)v||_{\V^i_{F_i(x)}} \ge   \rho_i ||v||_{\V^{i+1}_x}.$$
\item  
In fact,  if $i\ge 1$, $x\in \U^i$, $\U^i_x$ contains a critical point  of $F_i$, 
$v\in T_x \U^i$ and $d(x,\Crit(F))\ge \kappa_0$, then
$$ ||\Df F_i(x)v||_{\V^{i-1}_{F_i(x)}} \ge \rho_{i-1}\cdot \tau_{i-1}^{t(i-1)-1} \cdot ||v||_{\V^i_x}$$
where  $t(i-1)\ge 1$ is so that 
$F_i(x)=F_{i-1}^{t(i-1)}(x)$.
\item If $i\ge 1$, $y\in \V^i$ and $w\in T_y \V^i$, then 
$$||w||_{\V^{i-1}_y} \le (1/\rho_i)  ||w||_{\V^0_y} $$
(we will apply this assertion for the case when $x\in \U^i$ and $y=F_i(x)\in \V^i$). 
\end{enumerate} 
 \end{lemma} 
\begin{proof}  Part (1) follows inductively from the corresponding statement for $F_{i-1}$: 
since conformal maps preserve moduli and since one pulls back at most $N=\# \Crit(F)$ times by a non-conformal map, see  \cite[Lemma 8.3]{KvS}. Part (2) follows since 
$F_i\colon \U^i_x\to \V^i_{F_i(x)}$ is an isometry w.r.t.\ to the Poincar\'e metric
and the embedding $\U_{x}^i \hookrightarrow \V_{x}^i$ is a contraction (by a factor which 
only depends on $\delta_i$).
Part (3) holds because $\Df F_i$ only vanishes at critical points of $F$. 

To see Part (4), we need to consider two subcases. 
(i) $F_{i-1}(x)\in \V^{i}$ and  (ii) $F_{i-1}(x)\in \V^{i-1} \setminus \V^{i}$.
If (i) holds then  $F_i(x)=F_{i-1}(x)$  and so by Part \textcolor{red}{(3)} we have
$$||\Df F_i(x)v||_{\V^{i-1}_{F_i(x)}} =
||\Df F_{i-1}(x)v||_{\V^{i-1}_{F_{i-1}(x)}} 
\ge \rho_{i-1}  ||v||_{\V^i_x}.$$
On the other hand, if (ii) holds, i.e. if $F_{i-1}(x)\in \V^{i-1} \setminus \V^{i}$, then 
$F_i$ is of the form $F_{i-1}^{t(i-1)-1}\circ F_{i-1}$ where $t(i-1)\ge 1$
and where the first iterate is through a critical branch of $F_{i-1}\colon \U^{i-1}\to \V^{i-1}$ 
and the other iterates pass through diffeomorphic  branches. 
Therefore, using Parts (2) and (3) we obtain Part (4).

Part (5) holds because the diameter of each component of $\V^i$ is bounded from below (with a bound which depends 
on $F$ and $i$) and because one has a lower bound for the modulus of the component 
of $\V^{i-1}\setminus \V^i$ containing $y$. 
\end{proof} 

\begin{proposition}\label{prop:exp-realtime} 
Let $F\colon \U\to \V$ be a complex box mapping so that 
(\ref{eq:maneassump}) holds and define the maps $F_i$ as above. Then for each $n\ge 0$ and $\kappa_0>0$
there exist $\tau>1$ and $C>0$ so that the following holds. 
Assume that $0\le i\le n$, $x\in \U^i$, $d(x,\Crit(F))\ge \kappa_0$ 
and $y=F_i(x)$ and define $t=t(y,x)$, i.e. 
so that $F_i(x)=F^t(x)$. 
Then for each $v\in T_x \U^i$ one has 
$$||\Df F_i(x)v||_{\V^i_y} \ge C\tau^t \cdot ||v||_{\V^i_x},$$
and  if $\U^i_x$ does not contain a critical point 
$$||\Df F_i(x)v||_{\V^i_y} \ge \tau^t \cdot ||v||_{\V^i_x}.$$ 
\end{proposition} 
\begin{proof} Define $t(j) = \# \{x,\dots,F^{t-1}(x)\in \U^j\}$, let $\rho=\min_{i=0,\dots,n}\rho_i$
and $\tau=\min_{i=0,\dots,n}\tau_i$ where $\rho_i,\tau_i$ are as in the 
previous lemma, and choose $\eta>0$
so that $\tau (\rho/\tau)^{\eta}>\tau':=(1+\tau)/2$. 
Let $j$ be the smallest integer $<i$ so that  $t(j)< \eta t(j-1)$. \\
{\bf Case 1:} $j$ does not exist. Then $t(i-1)\ge \eta^{i-1} t$.
As before write $F_i(x)=F_{i-1}^{t(i-1)-1}\circ F_{i-1}(x)$. Then by Part (4)
of Lemma~\ref{lem:poinc-exp/cont}
$$ ||\Df F_i(x)v||_{\V^{i-1}_{F_i(x)}} \ge \rho_{i-1} \cdot \tau_{i-1}^{t(i-1)-1} \cdot ||v||_{\V^i_x}.$$
Since  $t(i-1)\ge \eta^{i-1} t$, by Schwarz inclusion 
and since $\V^i_y \subset \V^{i-1}_{F_i(x)}$,
this implies $$||\Df F_i(x)v||_{\V^i_y} \ge C\tau_*^t \cdot ||v||_{\V^i_x},$$
when taking $C=\rho/\tau$ and by defining $\tau_*= \tau^{\eta^{i-1}}$.
This gives the first inequality, when  renaming $\tau_*$ by $\tau$. 
Let $t_0$ be so large that $C\tau_*^t>1$ for $t\ge t_0$. 
Now take $\hat \tau >1$ so that $C\tau_*^t>\hat \tau^t$ for $t\ge t_0$.
This implies the 2nd inequality when $t\ge t_0$ (again renaming 
$\hat \tau$ by $\tau$). 
If $t<t_0$ then the 2nd inequality holds immediately from Part (2) of the previous lemma,
by choosing  $\tau>1$ sufficiently close to one. \\
%
%
{\bf Case 2:}  $j$ does exist. Then $t(j-1)\ge \eta^{j-1} t$ and 
$t(j)\le \eta \cdot  t(j-1)$. In that case iterate $x$ by $F_{j-1}$ through 
its critical branches until  
an iterate enters $\V^{j-1}\setminus \V^{j}$, and then iterate 
$F_{j-1}$ through diffeomorphic branches until we again hit $\V^{j}$. This may repeat several times, but 
the net expansion we obtain is 
$$ ||\Df F_i(x)v||_{\V^{i-1}_{F_i(x)}} \ge  \rho^{t(j)} \tau^{t(j-1)-t(j)} \cdot ||v||_{\V^i_x}.$$
Since $\tau (\rho/\tau)^{\eta}>\tau'$ the first inequality follows. The 2nd inequality then 
holds as in Case 1. 
\end{proof} 

The following lemma is straight-forward. 

\begin{lemma}\label{lem:comparPE}
 There exists $\rho_0>0$ so that if $x\in \V$, $v\in T_x \V$, 
then
$$ ||v||_\V\ge \rho_0 ||v|| .$$
If $d(x,\partial \V)\ge \kappa$, then there exists $\rho_1(\kappa)>0$ so that 
$$ \rho_1 ||v||_\V \le ||v||.$$ \qed
\end{lemma}

Let us now prove Theorem~\ref{thm:mane}. 

\begin{proof}[Proof of Theorem~\ref{thm:mane}] 
We will decompose $F^k$, by first mapping it via iterates of, successively, 
$F_{m_0}$ and then 
of $F_{m_1},\dots,F_{m_{i_0}}$
going down closer and closer to the critical points of $F$ (where 
$m_0<m_1<\dots<m_{i_0}$) 
and then by maps $F_{m_{i_0+1}}, \dots,F_{m_{i_1}}$  (where $m_{i_0}>m_{i_0+1}>\dots>
m_{i_1}$)  up to when the point $F^k(x)$ has been reached. 

As we will iterate $x$ not only by $F$ but also by the maps $F_m$,  we need to
be sure we stop before we reach the $k$-th iterate of $F$. For this reason we define 
$t=t(x,m,n)$ for $x\in \V^m$ so that  $F_m^n=F^t$ near $x$. This will allow us to 
make sure that if $x_i$ is 
an iterate of $F$ of time $<k$, then we will only iterate this point
with further iterates of $F$ up to time $<k$. 

So let first describe $m_0<\dots<m_{i_0}$.  
Let $m_0\ge 0$ be so that  $x_0:=x\in \V^{m_0}\setminus \V^{m_0+1}$. 
Then let $n_0\ge 0$ be maximal so that $F_{m_0}^i(x_0)\in \V^{m_0}\setminus \V^{m_0+1}$ for all $0\le i< n_0$ and so that $t(x_0,0):=t(x_0,m_0,n_0)\le k$.  If $n_0=0$, then define $i_0=0$.
If $n_0>0$, then  let $x_1:=F_{m_0}^{n_0}(x_0)=F^{t(x_0,0)}(x_0) \in \V^{m_0+1}$,  let $m_1$ be so that $x_1\in \V^{m_1}\setminus \V^{m_1+1}$ and 
let $n_1\ge 0$ be maximal so that $F_{m_1}^{n_1}(x_1)\in \V_{m_1}$ and so that $t(x_0,1):= t(x_1,m_1,n_1)+t(x_0,0)\le k$. Define $x_2:=F_{m_1}^{n_1}(x_1)=F^{t(x_0,1)}(x_0) \in \V_{m_1}$.
If $n_1>0$ we continue and define inductively in a similar way
$x_i,m_i,n_i,t(x_0,i)$ for $i=0,\dots,i_0-1$ and $x_{i_0}$ where $i_0$ is chosen maximal, 
i.e.  $x_{i_0}\in \V^{m_{i_0}}$, but there is no $F$-iterate of $x_0$  (up to $k$ of $F$)
 which  enters $\V^{m_{i_0}+1}$. In this case let $n_{i_0}\ge 0$ be 
 maximal so that $F^{n_{i_0}}(x_{i_0})\in \V^{m_{i_0}}$ and so that 
 $t(x_0,i_0):=t(x_{i_0},m_{i_0},n_{i_0}) + t(x_0,i_0-1) \le k$. Define 
  $x_{i_0+1}=F_{m_{i_0}}^{n_{i_0}}(x_{i_0})=F^{t(x_0,i_0)} \in \V^{m_{i_0}}$.
  Note that $x_i\in \V^{m_i}$ for $i=0,\dots,i_0$ but that also $x_{i_0+1}\in \V^{m_{i_0}}$. Moreover, 
  $0\le m_0<m_1<\dots<m_{i_0}$ and $n_0,\dots,n_{i_0-1}>0$.   
  Next we will define $m_{i_0+1},\dots,m_{i_1}$. 
Define $m_{i_0+1}$  maximal so that that $t(x_{i_0+1},m_{i_0+1},1) + t(x_0,i_0) \le k$, 
i.e. so that the iterate $F_{m_{i_0+1}}(x_{i_0+1})$ is before time $k$. Note that $m_{i_0+1}<m_{i_0}$ and so 
$x_{i_0+1}\in \V^{m_{i_0}}\subset  \V^{m_{i_0+1}}$. 
Then take $n_{i_0+1}$ to be maximal so that 
$t(x_0,i_0+1):=t(x_{i_0+1},m_{i_0+1},n_{i_0+1}) + t(x_0,i_0) \le k$  and define 
$x_{i_0+2}=F^{n_{i_0+1}}_{m_{i_0+1}}(x_{i_0})=F^{t(x_0,i_0+1)}(x_0)\in \V^{m_{i_0+1}}$.  
Next take $m_{i_0+2}$ maximal so that 
$t(x_{i_0+1},m_{i_0+1},1) + t(x_0,i_0+1) \le k$ 
and define $n_{i_0+2}$ to be maximal so that $t(x_0,i_0+2):=j(x_{i_0+2},m_{i_0+2},n_{i_0+2}) + t(x_0,i_0+1) \le k$.  Similarly, define $m_i,n_i,t(x_0,i)$ for $i=i_0+1,\dots,i_1$
where $i_1$ is maximal so that {\lq}time runs out{\rq} (so further iterates would consider
more than $k$ iterates under $F$ of $x$). 

Note that for $i=i_0+1,\dots,i_1$,  the maximality of $m_{i}$ implies that 
\begin{equation} n_{i}>1\implies  F_{m_{i}}^j(x_{i})\notin \V^{m_{i}+1}\mbox{ for }j=0,\dots,n_i-1\end{equation} 
because otherwise one could iterate with $F_{m_{i}+1}$ (before reaching time $k$).

Let us now discuss expansion in terms of some Poincar\'e metrics along this orbit. 
For simplicity choose $\tau>1$ and $\rho>0$
so that these bound from below the  constants $\tau_i,\rho_i$,  $i=0,\dots,\nu$ 
from Lemma~\ref{lem:poinc-exp/cont} and so that the conclusion of 
Proposition~\ref{prop:exp-realtime} holds.  
For $i=0,\dots,i_0$ and $0\le j < n_i$, by definition  $F^j_{m_i}(x_i)$ is not 
contained in a critical domain of $\V^{m_i}$ and so by 
Proposition~\ref{prop:exp-realtime}    we get  
$$||\Df F_{m_i}(F^j_{m_i}(x_i))v||_{\V_{F^{j+1}_{m_i}(x_i)}} \ge \tau^{t(F^{j}_{m_i}(x_i),x_i)} \cdot 
||v||_{\V_{F^{j+1}_{m_i}(x_i)}} 
$$
where $\tau>1$. 
This gives for $i=0,\dots,i_0$, 
$$ 
||\Df F_{m_i}^{n_i}(x_i)v||_{\V^{m_i}_{x_{i+1}}} \ge \tau^{t(x_{i+1},x_i)} \cdot ||v||_{\V^{m_i}_{x_i}}. $$
Note that $x_{i_0}=F_{m_0-1}^{n_{i_0-1}}\circ \dots \circ F_{m_0}^{n_0}(x_0)  \in \V^{m_{i_0}}$ and that  the  above statement gives
$$||\Df F_{m_{i_0-1}}^{n_{i_0-1}} \cdots \Df F_{m_0}^{n_0}(x_0)(v)||_{\V^{m_{i_0-1}}_{x_{i_0}}}\ge \tau^{t(x_{i_0},x_0)}\cdot ||v||_{\V^{m_0}_{x_0}}.$$
If $n_{i_0}=0$, then $x_{i_0+1}=x_{i_0}$ and we obtain from Part (5) of Lemma~\ref{lem:poinc-exp/cont}
that 
$$||\Df F_{m_{i_0}}^{n_{i_0}} \cdots \Df F_{m_0}^{n_0}(x_0)(v)||_{\V^{0}_{x_{i_0+1}}}\ge \tau^{t(x_{i_0},x_0)}\cdot \rho \cdot ||v||_{\V^{m_0}_{x_0}}.$$
If $n_{i_0}\ge 1$, then note  that $x_{i_0+1}\in \V^{m_{i_0}}$ and so for all $v\in T_{x_{i_0}}\U_{x_{i_0}}^{m_{i_0-1}}$
$$||\Df F_{m_{i_0}}^{n_{i_0}}(x_{i_0})(v)|| _{\V^{m_{i_0}-1}_{x_{i_0+1}}}
\ge  \tau^{t(x_{i_0},x_0)} \cdot \rho \cdot ||v||_{\V^{m_{i_0-1}}_{x_{i_0}}}
$$ 
where we use for the first $n_0-1$ iterates of $F_{m_{i_0}}$ Proposition~\ref{prop:exp-realtime}
and for the last iterate Part (4)  of Lemma~\ref{lem:poinc-exp/cont}. 
So in all cases we have
$$||\Df F_{m_{i_0}}^{n_{i_0}} \cdots \Df F_{m_0}^{n_0}(x_0)(v)||_{\V^{0}_{x_{i_0+1}}}\ge \tau^{t(x_{i_0+1},x_0)}\cdot \rho^2 \cdot ||v||_{\U^{m_0}_{x_0}}.$$
Using the same argument for $i=i_0+1,\dots,i_1$ we thus obtain
$$||\Df F^k(x_0) v||_{\V^0_{F^k(x_0)}} = 
||\Df F_{m_{i_1}}^{n_{i_1}} \cdots \Df F_{m_0}^{n_0}(x_0)(v)||_{\V^{0}_{x_{i_0+1}}}\ge \tau^{t(F^k(x_0),x_0)}\cdot (\rho^2/\tau)^\nu \cdot ||v||_{\U^{m_0}_{x_0}}.$$
Using $d(F^k(x),\partial \V)\ge \kappa$ and Lemma~\ref{lem:comparPE}
we can transfer this statement to one in terms of the Euclidean norm on $\C$ and obtain 
$$||\Df F^k(x_0)v||\ge C\cdot  \tau^{t(F^k(x_0),x_0)} \cdot ||v||$$
where $C= (\rho^2/\tau)^\nu (\rho_1\rho_0)$ and where  $\nu$ is fixed (as in the statement of the theorem). 
\end{proof}

\appendix

\section{Some basic facts}
\label{A1}

In this appendix, we present some facts that are standard and go without saying for the experts. However, for the readers who are just entering the field, not knowing these facts might be an undesirable obstacle towards understanding our paper. With this larger audience in mind, in this appendix we have collected some basic statements that are used throughout the paper and that would be very useful to keep in mind while reading. The proofs can be found in good textbook sources (e.g.~\cite{RealBook, LyuBook, McM}) or in some research papers (e.g.~\cite{DS}).  

\subsection{First entry and return constructions}

One of the most fundamental constructions to study recurrent orbits is that of \emph{first return} and \emph{first entry}. The next three lemmas explain in steps the construction and its properties. 

\begin{lemma}[Orbits of entry components are disjoint]
\label{Lem:A1}
Let $f \colon U \to V$, $U \subset V$ be a map between two sets. Pick a set $A \subset U$ and for each point $x \in U$ so that the orbit of $x$ under $f$ intersects $A$, define $k_x$ to be the \emph{entry time} in $A$, i.e.\ the smallest integer $k \ge 1$ such that $f^k(x) \in A$. For $n \in \mathbb N$, define $D_n := \{x \in U \colon k_x = n\}$. Then the sets
\[
D_n, f(D_n), \ldots, f^{n-1}(D_n)
\]  
are pairwise disjoint for every $n$.
\end{lemma}

\begin{proof}
Immediately from the definition, $f^p (D_n) \subseteq D_{n-p}$ and $D_n \cap D_m = \emptyset$ for every $p < n$ and $m \neq n$. Therefore, if $f^p(D_n) \cap f^s(D_n) \neq \emptyset$ for some $p \neq s$, then $D_{n - p} \cap D_{n - s} \neq \emptyset$, a contradiction.
\end{proof}

\begin{lemma}[Boundary maps to boundary for nice sets]
\label{Lem:A2}
If in the previous lemma we additionally assume that $f$ is continuous and $A$ is a nice open set, then every connected component of $D_n$:
\begin{itemize}
\item
is nice and open;
\item
is either contained in $A$, or is disjoint from $A$;
\item
after exactly $n$ steps maps over a connected component of $A$.
\end{itemize} 
\end{lemma}

\begin{proof}
For simplicity assume that $D_n$ and $A$ are connected. Then $D_n$ is nice since $A$ is nice. By continuity of $f$, $D_n$ must be open. Because $A$ is nice, $D_n \cap \partial A = \emptyset$. This yields the second property. The last property follows since $f^n(\partial D_n) = \partial A$, again by continuity. 
\end{proof}

Recall the definitions of the first entry, landing and return maps from Section~\ref{SSSec:Return}. It is often the case that the first return map to a \emph{nice} union of topological disks has the structure of a complex box mapping. The following two lemmas should give the reader a ``flavor'' of what kind of formal statements one might expect in this direction (recall from Section~\ref{SSSec:Return} the definition of a strictly nice set).   

\begin{lemma}[Box mappings from strictly nice unions]
\label{Lem:A3}
Let $U \subset V \subset \Cc$ be a pair of domains, and let $f \colon U \to V$ be a holomorphic map with finitely many critical points. Let $\V \subset U$ be a \emph{strictly} nice set such that 
\begin{itemize}
\item
$\V$ is a union of finitely many open Jordan disks with disjoint closures;
\item
$\Crit(f) \subset \V$. 
\end{itemize}
Then the first return map to $\V$ under $f$ is a complex box mapping. \qed
\end{lemma}

In Lemma~\ref{Lem:A3}, the condition of being strictly nice ensures that the components of the first return domain are compactly contained in $\V$. One can weaken the conditions $\Crit(f) \subset \V$ and $\# \Crit(f) < \infty$ in the previous lemma, for example, as follows.    

\begin{lemma}[Box mappings from strictly nice unions, version 2]
\label{Lem:A3}
Let $U \subset V \subset \Cc$ be a pair of domains, and let $f \colon U \to V$ be a holomorphic map. Let $\V \subset U$ be a \emph{strictly} nice set such that 
\begin{itemize}
\item
$\V$ is a union of finitely many open Jordan disks with disjoint closures;
\item 
no iterate of a critical point intersects $\partial \V$. 
\end{itemize}
Let $\W$ be the union of $\V$ and all the components of the first landing domain to $\V$ under $f$ that intersect $\Crit(f)$. Assume that $\# (\W \cap \Crit(f)) < \infty$. Then the first return map to $\W$ under $f$ is a complex box mapping. \qed
\end{lemma}

\begin{remark}
Note that $\W$ in the lemma above is a nice, but not necessarily strictly nice set.
\end{remark}

In the context of box mappings, Lemmas~\ref{Lem:A1} and \ref{Lem:A2} imply that every connected component $X$ of the first entry map to a nice union of puzzle pieces is a puzzle piece itself, and moreover, the branch of the entry map restricted to $X$ is a branched covering of uniformly bounded degree (with the bound that depends only on the number of critical points of the starting map and their multiplicities). The latter follows immediately from Lemma~\ref{Lem:A1} as the sequence $D_n, f(D_n), \ldots, f^{n-1}(D_n)$ can meet a critical point at most once. We summarize some properties of the entry, as well as landing and return, maps in the following lemma.  

%

\begin{lemma}[First return, entry and landing maps to a nice set]
\label{Lem:FirstReturnConstruction}
Let $F \colon \U \to \V$ be a complex box mapping, and let $ W \subset \U$ be a nice union of puzzle pieces of $F$. Then the first entry map $E \colon \DomE \to W$, the first return map $R\colon\Dom\to W$, and the first landing map $L\colon\DomL \to W$ have the following properties.:
\begin{enumerate}
\item
\label{It:FRM1}
$\DomE$ and $ \Dom$ are disjoint unions puzzle pieces of $F$, and hence are open;
\item
\label{It:FRM2}
for every component $Y'$ of $\DomE$ there is a component $W'$ of $W$ and an integer $n$ so that $E|_{Y'} = F^n|_{Y'}$, and $E\colon Y'\to W'$ is a proper map. The same property follows for $R$ and $L$ by restriction.

\item
\label{It:FRM3}
the local degrees of the  maps $R$, $E$, and $L$ are bounded in terms of $F\colon \U\to \V$. Moreover, if $\Crit(F) \subset  W$, then $\Crit(R)$, $\Crit(E)$ and $\Crit(L)$ are contained in  $\Crit(F)$;
\item
\label{It:FRM4}
If $F^n(x) \in W$ and $F^m(x) \in W$ for some $m > n$ and $x \in \U$, then $F^n(x) \in \Dom$. In particular, every orbit $x, F(x), F^2(x), \ldots$ that intersects $W$ infinitely often, lands in the set $\{x \in  \Dom \colon R^n(x) \in  \Dom \text{ for all }n\ge0\}$ (the \emph{non-escaping set} of $R$).   \qed
\end{enumerate}
\end{lemma}

\subsection{Koebe Distortion Theorem, bounded geometry and qc mappings}

\begin{theorem}[Invariant Koebe Distortion Theorem]
Let $\Delta \subset D \subset \C$ be a pair of open topological disks with $\modulus (D \setminus \overline{\Delta}) \ge \mu >0$. Let $f \colon D \to \C$ be a univalent map. Then there is a constant $C(\mu)$ such that the \emph{distortion} of $f$ on $\Delta$ is bounded by $C(\mu)$:
\[
\frac{|f'(x)|}{|f'(y)|} \le C(\mu) \quad \text{ for every }\quad x,y \in \Delta.
\]
Moreover, $C(\mu) \to 1$ as $\mu \to \infty$. \qed
\end{theorem}


\begin{lemma}[Annulus pull-back and geometry control under branched covering]
\label{Lem:Fact}
Let $f \colon U' \to V'$ be a branched covering of degree at most $D$ between open topological disks, and suppose $V \subset V'$ is a topological disk with $\modulus(V' \sm \ovl V) \ge \mu > 0$. Let $U \subset U'$ be a component of $f^{-1}(V)$. Then 
\[
\modulus(U' \sm \ovl U) \ge \frac{\mu}{D}.
\]
If we additionally assume that $V$ has $\eta$-bounded geometry, then $U$ has $\eta' = \eta'(\eta, \mu, D)$-bounded geometry. \qed
\end{lemma}

There are several equivalent definitions for the notion of quasiconformal homeomorphism. 
The one which seems to be the most natural in our context is the one 
given by Heinonen and Koskela \cite{HK}: 

\begin{definition} For open domains $\Omega, \tilde \Omega \subset \C$ and $K \ge 1$, a homeomorphism  $\phi \colon \Omega\to \tilde{\Omega}$  is called 
{\em $K$-quasiconformal} (or \emph{$K$-qc} for short) if for each $x\in \Omega$, 
$$\liminf_{r\to 0}\frac{\sup_{|y-x|=r}|\phi(y)-\phi(x)|}
{\inf_{|y-x|=r}|\phi(y)-\phi(x)|}\le K < \infty.$$
\end{definition} 

This definition also motivates why the notion of bounded geometry 
appears in the QC-Criterion (Theorem~\ref{Thm:QCCriterion}).


\end{document}